\documentclass[letterpaper, 11pt,  reqno]{amsart}

\usepackage[margin=1.2in,marginparwidth=1.5cm, marginparsep=0.5cm]{geometry}

\usepackage{amsmath, mathrsfs, amssymb,amscd,amsthm,amsxtra, esint}
\usepackage{color}

\usepackage[implicit=true]{hyperref}

\usepackage{cases}
\allowdisplaybreaks[2]

\definecolor{gr}{rgb}   {0.,   0.69,   0.23 }
\definecolor{bl}{rgb}   {0.,   0.5,   1. }
\definecolor{mg}{rgb}   {0.85,  0.,    0.85}
\definecolor{yl}{rgb}   {0.8,  0.7,   0.}
\definecolor{or}{rgb}  {0.7,0.2,0.2}

\newtheorem{theorem}{Theorem} [section]

\newtheorem{lemma}[theorem]{Lemma}
\newtheorem{proposition}[theorem]{Proposition}
\newtheorem{remark}[theorem]{Remark}

\newtheorem{definition}[theorem]{Definition}

\newtheorem{oldtheorem}{Theorem}

\DeclareMathOperator*{\intt}{\int}

\DeclareMathOperator{\tr}{tr}

\DeclareMathOperator{\Id}{{\bf Id}}

\newcommand{\noi}{\noindent}
\newcommand{\Z}{\mathbb{Z}}
\newcommand{\R}{\mathbb{R}}

\newcommand{\T}{\mathbb{T}}
\newcommand{\bul}{\bullet}

\let\P= \undefined
\newcommand{\P}{\mathbf{P}}

\newcommand{\EE}{\mathcal{E}}

\renewcommand{\H}{\mathcal{H}}

\newcommand{\D}{\mathcal{D}}

\renewcommand{\L}{\mathcal{L}}

\newcommand{\Nf}{\mathfrak{N}}

\newcommand{\NB}{\mathbb{N}}

\newcommand{\ze}{\zeta}

\newcommand{\al}{\alpha}
\newcommand{\be}{\beta}
\newcommand{\dl}{\delta}

\newcommand{\eps}{\varepsilon}
\newcommand{\kk}{\kappa}

\newcommand{\G}{\Gamma}

\newcommand{\s}{\sigma}

\newcommand{\ft}{\widehat}
\newcommand{\Ft}{{\mathcal{F}}}
\newcommand{\wt}{\widetilde}

\newcommand{\dx}{\partial_x}
\newcommand{\dt}{\partial_t}

\newcommand{\LRA}{\Longrightarrow}

\newcommand{\ta}{\theta}

\renewcommand{\l}{\ell}

\newcommand{\Gdl}{\mathcal{G}_{\dl} }
\newcommand{\Tdl}{\mathcal{T}_{\dl} }
\newcommand{\Qdl}{\mathcal{Q}_{\dl}}

\newcommand{\Gd}{\wt{\mathcal{G}}_\dl}

\newcommand{\Tfr}{\mathfrak{T} }

\newcommand{\les}{\lesssim}
\newcommand{\ges}{\gtrsim}

\newcommand{\jb}[1]
{\langle #1 \rangle}

\newcommand{\ind}{\mathbf 1}

\newcommand{\gf}{\mathfrak{g}}

\newcommand{\M}{\mathcal{M}}

\newcommand{\N}{\mathbb{N}}

\newcommand{\NN}{\mathcal{N}}

\newcommand{\xii}{\pmb{\xi}}

\newcommand{\pb}{\mathfrak{p}}
\newcommand{\sbb}{\mathfrak{s}}

\newcommand{\uu}{\mathbf{u}}
\newcommand{\vv}{\mathbf{v}}
\newcommand{\vvd}{\mathbf{v}_\delta}
\newcommand{\ftvvd}{\ft {\mathbf{v}}_\delta}

\newcommand{\KDV}{\text{\rm KdV} }

\newcommand{\too}{\longrightarrow}

\newcommand{\TT}{\mathcal{T}}

\newcommand{\If}{\mathfrak{I}}

\newtheorem*{ackno}{Acknowledgements}

\numberwithin{equation}{section}
\numberwithin{theorem}{section}

\makeatletter
\@namedef{subjclassname@2020}{\textup{2020} Mathematics Subject Classification}
\makeatother

\begin{document}
\baselineskip = 14pt

\title[Shallow-water convergence of ILW in  $L^2$]
{Shallow-water convergence of the intermediate long wave equation in $L^2$}

\author[A.~Chapouto, G.~Li, T.~Oh, and T.~Zhao]
{Andreia Chapouto, 
Guopeng Li,  
Tadahiro Oh, 
and
Tengfei Zhao}

\address{
Andreia Chapouto, 
CNRS, Laboratoire de math\'ematiques de Versailles, UVSQ, Universit\'e Paris-Saclay, CNRS, 45 avenue des 
\'Etats-Unis, 78035 Versailles Cedex, France, 
and School of Mathematics, Monash University, VIC 3800, Australia}

\email{andreia.chapouto@monash.edu}

\address{Guopeng Li, 
School of Mathematics and Statistics, Beijing Institute of Technology, Beijing 100081, China}

\email{guopeng.li@bit.edu.cn}

\address{
Tadahiro Oh, 
School of Mathematics\\
The University of Edinburgh\\
and The Maxwell Institute for the Mathematical Sciences\\
James Clerk Maxwell Building\\
The King's Buildings\\
Peter Guthrie Tait Road\\
Edinburgh\\
EH9 3FD\\
 United Kingdom\\
and School of Mathematics and Statistics, Beijing Institute of Technology, Beijing 100081, China}

%

\email{hiro.oh@ed.ac.uk}

\address{Tengfei Zhao, 
 School of Mathematics and Physics, University of Science and Technology
Beijing, Beijing, 100083, China}

\email{zhao.tengfei@ustb.edu.cn}



\begin{abstract}

We continue our study on the convergence issue 
of the intermediate long wave
equation (ILW)   on both the real line and  the circle.
In particular, we establish
 convergence
of the scaled ILW dynamics to that 
of the Korteweg-de Vries equation (KdV)
in the shallow-water limit
at the $L^2$-level.
Together with the recent work by the first three authors and D.\,Pilod (2024)
on the deep-water convergence in $L^2$, 
this work completes the well-posedness
and convergence study of ILW
on both geometries
within the $L^2$-framework.
Our proof equally applies to both geometries
and is based
on the following two ingredients:
the complete integrability of ILW
and 
the 
normal form method.
More precisely, by making use of 
the Lax pair structure and the perturbation determinant
for ILW, recently introduced by 
Harrop-Griffiths, Killip, and Vi\c{s}an (2026), 
we first establish weakly uniform (in small depth parameters) equicontinuity in $L^2$ of 
solutions to the scaled ILW, providing a control on the high frequency part of solutions.
Then, we treat the low frequency part by implementing a perturbative argument based on 
an infinite iteration of normal form reductions
for KdV.

\end{abstract}

%
\maketitle
%


\tableofcontents


\newpage
\section{Introduction}
\label{SEC:1}

\subsection{Intermediate long wave equation}
\label{SUBSEC:1.1}

In this paper, we study  the following intermediate long wave equation (ILW)
on $\M = \R$ or $\T = (\R/ 2 \pi \Z)$:
\begin{align}
\begin{cases}
\dt u - \mathcal{G}_{\dl}\dx^2 u =\dx(u ^2) \\
u \vert_{t=0}  =u_0,
\end{cases}
\quad 
(t,x)\in \R\times \mathcal{M}. 
\label{ILW1}
\end{align}

\noi
The equation \eqref{ILW1}
was  introduced in \cite{joseph, KKD}
as a  model
 describing the propagation of an internal wave at the interface of a stratified fluid of 
finite depth $\dl \in (0, \infty)$, 
where
 the unknown $u :\R\times \M \to \R$
denotes the amplitude of the internal wave at the interface;
see also Remark~1.1 in \cite{LOZ}.
The  operator $\Gdl $ in \eqref{ILW1}
is defined by
\begin{align}
\Gdl   = 
\Tdl - \dl^{-1} \dx^{-1},  
\label{GG1}
\end{align}

\noi
where $\Tdl$
is given by 
\begin{align}
\ft{\Tdl f}(\xi)  =  -  i \coth(\dl \xi ) \ft f(\xi), 
\quad \xi \in\ft \M.
 \label{GG2}
\end{align}

\noi
Here, $\ft \M$ denotes 
the Pontryagin dual of $\M$, i.e.
$\ft \M = \R$ if $\M = \R$, 
and 
$\ft \M = \Z$ if $\M = \T$.
From the analytical point of view,
 \eqref{ILW1} is  of great  interest due to its rich structure;
 it is a dispersive equation, admitting soliton solutions.
  Moreover,  it is completely integrable with an infinite number of conservation laws \cite{KAS82, KSA, S89}
  and Lax pair structures  \cite{HKV}.
See~\cite{S19, KS21} for an overview of the subject and the references therein.

It is known that the ILW equation \eqref{ILW1}
 serves as an ``intermediate'' equation of finite depth $ \dl \in (0, \infty)$,
bridging the following two equations:

\smallskip
\begin{itemize}
\item
Benjamin-Ono equation, modeling fluid of infinite depth  ($\dl = \infty$), and

\smallskip
\item
Korteweg-de Vries equation, modeling shallow water ($\dl = 0$), 
\end{itemize}

\smallskip

\noi
both of which are also completely  integrable.
We note that 
the operator $\Tdl$ in \eqref{GG1}
 is the inverse of the 
so-called Tilbert transform (with an extra $-$ sign), 
appearing in the study of water waves
of finite depth \cite{BLS, HIT, AIT}.
In particular,  in the deep-water limit ($\dl \to \infty$), 
$\Tdl$ converges to the Hilbert transform $\H$
(associated with 
the Dirichlet-to-Neumann map 
in infinite depth), 
indicating possible convergence of ILW~\eqref{ILW1}
to 
the Benjamin-Ono equation~(BO):
\begin{align}
\dt u -
 \H \partial_x^2 u  =  \dx (u^2).
\label{BO1}
\end{align}

\noi
See Subsection \ref{SUBSEC:1.2}
for a further discussion.
Next, let us turn our attention to 
 the shallow-water limit ($\dl\to0$).
Recalling that  $u$ denotes the amplitude of the internal wave 
at the interface of  a stratified fluid of finite depth $\dl > 0$,
it follows that 
the amplitude $u$ of the internal wave at the interface is $O(\dl)$,  which also  tends to~$0$
in the shallow-water limit.
This suggests that, 
in order to observe any meaningful limiting behavior,
we need to magnify the fluid amplitude by a factor $\sim \frac 1\dl$.
This motivates us to consider the following scaling~\cite{ABFS}:
\begin{equation}
v(t,x)  = \tfrac3\dl u\big(\tfrac3\dl t, x\big),
\label{scale1}
\end{equation}

\noi
which transforms \eqref{ILW1} to  the following scaled ILW equation:
\begin{align}
\dt v   -    \Gd   \dx^2 v= \dx(v^2), 
\qquad \text{where }\ \Gd := \frac3\dl \Gdl.
\label{sILW1}
\end{align}

\noi
Then, under suitable assumptions, 
the scaled ILW~\eqref{sILW1} is known to  converge to the 
Korteweg-de Vries equation (KdV)
in the shallow-water limit ($\dl\to 0$)
\cite{ABFS, Li24, LOZ}:
\begin{align}
\dt v + \dx^3 v = \dx(v^2) . 
\label{KDV}
\end{align}

In recent years, there have been intensive research activities
on mathematical analysis of ILW, 
 from both deterministic and statistical viewpoints;
see  \cite{ABFS, MST, MV, MPV,  
MPS, Li24, LOZ, IS,
CLOP, CFLOP, FLZ, LP, GL} for the known well-\,/\,ill-posedness results
as well as  results on convergence
of  (scaled) ILW to BO in the deep-water limit ($\dl \to \infty$)
and to KdV in the shallow-water limit ($\dl \to 0$);
see also 
\cite{CGLLO, GLLO1, GLLO2}
for results on ILW with modulated dispersion.
On the one hand, 
as seen in  the recent progresses \cite{CLOP, GL}, 
the deep-water convergence of ILW~\eqref{ILW1} to BO~\eqref{BO1}
is relatively well understood;
 see Subsection \ref{SUBSEC:1.2}.
On the other hand, 
our understanding of the shallow-water convergence
of the scaled ILW~\eqref{sILW1}
has been rather poor, 
due to the much worse convergence properties
of the associated linear operators;
in the deep-water limit,
the Fourier multiplier of $\Gdl \dx$ converges
to that of $\H \dx$, uniformly in frequencies, 
whereas 
such uniform convergence does not hold
in the shallow-water limit
(see, for example,  \cite[Subsection~2.1]{LOZ} and 
\cite[Subsection~2.2]{CLO}). 
In this sense,
the shallow-water limit
is  {\it singular} (especially when compared to the deep-water limit), 
making the shallow-water convergence problem much harder.
We also point out that, 
as observed in \cite{LOZ, CLO},  
this singular nature of the shallow-water limit
also appears in
convergence of  conservation laws 
and invariant measures of the scaled ILW
to those of KdV
(exhibiting 
2-to-1 collapses
with regularity jumps and singular supports of measures);
see 
\cite[Remark~1.2]{CLO}.
See also Remark \ref{REM:Ben} below.

Our main goal in this paper is to 
establish shallow-water convergence of 
the scaled ILW~\eqref{sILW1} to 
KdV~\eqref{KDV}
in $L^2(\M)$, 
by 
presenting a unified argument 
for both geometries
$\M = \R$ and $\M = \T$;
see Theorem \ref{THM:1}.
We achieve this goal by 
presenting a novel approach, 
based on the following
 two key ingredients:

\smallskip

\begin{itemize}
\item[(i)]
complete integrability approach, in particular, perturbation determinants
and related quantities, recently developed by 
R.\,Killip and M.\,Vi\c{s}an with their collaborators \cite{KVZ, KV19, HKV0, KLV, HKV}, and

\smallskip

\item[(ii)]
normal form approach, 
in particular, an infinite iteration of normal form reductions, 
developed by the third author 
with his collaborators \cite{GKO, OST, KOY, OW}. 

\end{itemize}

\smallskip

\noi
In Subsection \ref{SUBSEC:1.2}, 
we provide a brief overview of 
the deep-water regime
(namely, analysis on ILW~\eqref{ILW1}
and BO~\eqref{BO1} as its deep-water limit).
In Subsection \ref{SUBSEC:1.3}, 
we then consider 
the shallow-water regime
(namely, analysis on the scaled ILW~\eqref{sILW1}
and KdV~\eqref{KDV} as its shallow-water limit)
and state our main result (Theorem \ref{THM:1}).

\subsection{Deep-water regime}
\label{SUBSEC:1.2}

In this subsection, we briefly go over
recent well-posedness 
and deep-water convergence
 results  for ILW~\eqref{ILW1}.
It is known (see \cite{MST, KTz2}) that,  
just like BO~\eqref{BO1}, 
ILW~\eqref{ILW1}
behaves quasilinearly in the sense that 
a contraction argument 
can not be used 
for constructing solutions,
which makes the  well-posedness question 
of ILW 
rather challenging, especially in the low-regularity setting.

The main idea in this deep-water regime is to view
ILW as a perturbation of BO~\eqref{BO1}.
Fix $0 < \dl < \infty$.
Recall that  the Galilean transform
\begin{align}
w(t, x) =  u(t, x + \dl^{-1}t)
\label{gauge0} 
\end{align}

\noi
eliminates the term 
$\dl^{-1}\dx^{-1}$ in \eqref{GG1},
thus allowing us 
to  rewrite \eqref{ILW1}
as the following renormalized ILW:
\begin{align}
\dt w -
 \Tdl \dx^2 w =  \dx (w^2)   
\label{ILW2}
\end{align}

\noi
with $w|_{t = 0} = u_0$, 
where $\Tdl$ 
is as in \eqref{GG2}.
By writing $\Qdl =( \Tdl - \H)\dx$, 
where $\H$ denotes the Hilbert transform, 
we can write \eqref{ILW2} as 
\begin{align}
\dt w - \H \dx^2 w =  \dx (w^2) + \Qdl\dx w.
\label{ILW3}
\end{align}

\noi
Compare \eqref{ILW3} with \eqref{BO1}.
The key observation here is that 
$\Qdl$ has a smoothing property of any order
with a uniform control in $\dl \ges 1$
(see  \cite[Lemma 2.3]{CLOP}),
allowing 
us to view the second term on the right-hand side of \eqref{ILW3}
as a perturbation term.

In a recent work \cite{IS}, 
by combining 
this perturbative viewpoint
with 
the BO gauge transform  and the (quasilinear) normal form method, 
Ifrim and Saut 
proved  global well-posedness of~\eqref{ILW1}
in $L^2(\R)$ for each fixed $0 < \dl< \infty$
(along with long-time behavior of small solutions). 
This argument 
relies on dispersive estimates
and thus is not applicable to the periodic setting.
In \cite{CLOP}, 
the first three authors and Pilod
also 
exploited this perturbative viewpoint~\eqref{ILW3}, 
following~\cite{IS}.
By adapting
a unified
approach
to the $L^2$ well-posedness
of BO~\eqref{BO1}
on both the real line and the circle, 
 developed by Molinet and Pilod~\cite{MP}, 
the authors in \cite{CLOP}
proved global well-posedness of  ILW~\eqref{ILW1}
 in $L^2(\M)$, $\M = \R$ or $\T$. 
Furthermore, 
by exploiting the decay property of $\Qdl$ as $\dl \to \infty$, 
the  authors in  \cite{CLOP}
 proved the following 
deep-water convergence result.

\begin{oldtheorem}\label{THM:A}

Let $\M = \R$ or $\T$
and $s\ge 0$.
Given  $u_0 \in H^s(\M)$, 
let $u$ be the global solution to the BO equation \eqref{BO1}
with $u|_{t = 0} = u_0$ constructed in \cite{MP}.
Let  $\{u_{0, \dl}\}_{1 \le \dl <  \infty}\subset H^s(\M)$
such that 
$u_{0, \dl}$ converges to $u_0$ in $H^s(\M)$ as $\dl \to \infty$, 
and 
let $u_\dl$ be the global solution to the ILW equation~\eqref{ILW1} with 
$u_\dl|_{t = 0} = u_{0, \dl}$
constructed in \cite{CLOP}.
Then, as $\dl \to \infty$, 
$u_\dl$ converges to $u$
in 
 $C(\R; H^s(\M))$, 
endowed with the compact-open topology \textup{(}in time\textup{)}.

\end{oldtheorem}

We also mention  \cite{CFLOP, FLZ}, 
where 
the perturbative viewpoint \eqref{ILW3}
once again played a crucial role.
In the former paper \cite{CFLOP}, 
by combining this 
perturbative viewpoint
with  the complete integrability of BO, 
the first three authors, Forlano,  and Pilod
 established a global-in-time a priori bound
on the $H^s$-norm of a solution to  ILW for   $ - \frac 12 < s < 0$.
They also showed that 
ILW~\eqref{ILW1}  is ill-posed in $H^s(\M)$ for  $s < - \frac 12$, 
indicating that $s = -\frac 12$ is a critical regularity 
for ILW~\eqref{ILW1} (just like BO \eqref{BO1}).
In the latter paper \cite{FLZ}, 
the second and fourth authors with Forlano
established an unconditional deep-water convergence result
in $H^s(\M)$
for $s > 3 -  \sqrt{\frac{33}{4}} \approx 0.1277$.
We point out that 
the results
in \cite{CFLOP, FLZ}
hold on  both  the real line and the  circle.

\medskip

Before proceeding further, let us recall the notion of 
{\it sensible weak solutions}; see
\cite{OW1, OW, BOW}.

\begin{definition}\label{DEF:sol} \rm
Let $B(\M)\subset \D'(\M)$ be a Banach space of functions on $\M$.
Given $u_0 \in B(\M)$, 
we say that
 $u \in C((t_0, t_1); B(\M))$
is a sensible weak solution
to an equation
on a time interval $(t_0, t_1)$, $-\infty \leq t_0 < 0 < t_1 \leq\infty$, if,
for any sequence $\{u_{0, n}\}_{n \in \NB}$ of smooth functions
tending to $u_0$ in $B(\M)$, 
the corresponding (classical) solutions $u_n(t)$ with $u_n|_{t = 0} = u_{0, n}$
converge to $u(t)$ in $C((t_0, t_1); B(\M))$, 
endowed with the compact-open topology \textup{(}in time\textup{)}.

\end{definition}

Note that,  by definition, 
  sensible weak solutions are unique.
We also point out that the notion of sensible weak solutions 
is  rather weak; 
 sensible weak solutions do not have to satisfy 
the equation even in the distributional sense.
For example, in the case of ILW, BO, and KdV, 
the nonlinearity $\dx (u^2)$
does not need to make sense as a space-time distribution
even if $u$ is a sensible weak solution. 
However, note that, given a sequence $\{u_n\}_{n \in \N}$
of smooth approximating solutions, 
it follows from 
Definition \ref{DEF:sol} 
and the equation that 
$\dx (u_n^2)$
converges to some limit  as space-time distributions.
Namely, 
 a sensible weak solution~$u$ (in the case of ILW~\eqref{ILW1}) satisfies
\[ \dt u - \Gdl \dx^2 u =\lim_{n \to \infty} \dx(u_n^2)  \]
 
\noi
as a space-time distribution, where the right-hand side is independent
of a choice of smooth approximating solutions $\{u_n\}_{n \in \N}$.
 Namely, we interpret the nonlinearity
only as the (unique) distributional limit of the nonlinearities
of smooth approximating solutions.
On the one hand,  the notion of sensible weak solutions is 
weaker than the usual notion of (strong) solutions for which
we directly make  sense of the nonlinearity as a space-time distribution.
On the other hand, 
the notion of sensible weak solutions is 
widely accepted, 
for example, 
in the low regularity well-posedness
of integrable dispersive PDEs 
\cite{KT1, KV19, GKT, KLV} 
and singular stochastic PDEs
\cite{Hairer, GKO1, GKO2}.
See~\cite{OW} for a further discussion
on various notions of weak solutions.

There are two recent breakthrough results \cite{GKT, KLV}
on optimal well-posedness of BO~\eqref{BO1}, 
corresponding to the $\dl = \infty$ case.
In  \cite{GKT}, 
G\'erard, Kappeler, and Topalov
proved 
global well-posedness
of BO \eqref{BO1}
in the entire subcritical regime
$H^s(\T)$, $s > -\frac 12$,
via the Birkhoff map approach.
In the same paper, they also proved ill-posedness
of BO~\eqref{BO1} in 
$H^{-\frac 12}(\T)$.
In~\cite{KLV}, 
Killip, Laurens, and Vi\c{s}an 
extended the global well-posedness
for $s > -\frac 12$
to the real line case
via the commuting flow method, 
which equally applies to the periodic case.
Note that solutions constructed in \cite{GKT, KLV}
are interpreted only as sensible weak solutions
in the sense of Definition \ref{DEF:sol}.

In a very recent preprint \cite{GL}, 
Gassot and Laurens
studied well-posedness of ILW~\eqref{ILW1} 
in negative Sobolev spaces on the circle.
By  combining
the perturbative viewpoint~\eqref{ILW3}
following  \cite{IS, CLOP, CFLOP, FLZ}, 
the Birkhoff map for BO \cite{GKT}, 
and the quantity
$\be(\kk; u)$
(related to the perturbation determinant)
defined in 
\cite[Proposition 4.3]{KLV}, 
they proved that ILW~\eqref{ILW1}
is globally well-posed in $H^s(\T)$, $s > - \frac12$, 
and that deep-water convergence 
(as in Theorem~\ref{THM:A})
holds in the same regularity range. 
Note that these results in~\cite{GL}
hold  only in the sense of 
sensible weak solutions defined in Definition~\ref{DEF:sol}.
Since the approach in \cite{GL} relies on
the work~\cite{GKT} for the periodic setting, 
the corresponding results 
in negative Sobolev spaces on the real line are still open at this point.

The perturbative approaches in \cite{CLOP, CFLOP, FLZ, GL}
are
based 
on the Galilean transform~\eqref{gauge0} to remove $\dl^{-1}\dx^{-1}$
in~\eqref{GG1} (this term plays a crucial role in the shallow-water limit, 
cancelling a certain divergence)
and thus are 
 not amenable to studying shallow-water convergence
since the spatial translation by $\dl^{-1} t$ diverges as $\dl \to 0$.
Furthermore, the operator $\Qdl$ in~\eqref{ILW3}
diverges as $\dl \to 0$.
Hence, in order to establish shallow-water convergence
in the low regularity setting,  
one needs to develop a novel approach, 
which we discuss in the next subsection.

We conclude this subsection by mentioning
yet another very recent work \cite{HKV}
by 
Harrop-Griffiths, Killip, and Vi\c{s}an, 
where they 
made a 
substantial progress in understanding  the complete integrability structure
of ILW \eqref{ILW1}
in terms of the Lax pair structure and the perturbation determinant, 
and developed a 
{\it non-perturbative} approach
for {\it both} the deep-water and shallow-water regimes
on {\it both} geometries.
We will discuss the work \cite{HKV} further in the next subsection.

\subsection{Shallow-water regime: main result}
\label{SUBSEC:1.3}

In this subsection, we consider the scaled ILW~\eqref{sILW1}.
We first note that, for each fixed $0 < \dl < \infty$, 
the ILW equation \eqref{ILW1} 
and the scaled ILW equation \eqref{sILW1}
are equivalent.
In particular, 
it follows from the discussion presented in the previous subsection that 
the scaled ILW \eqref{sILW1}
is globally well-posed in $L^2(\M)$, $\M = \R$ or $\T$, 
and on the circle, this global well-posedness
extends to the range $s > - \frac 12$
(in the sense of sensible weak solutions).

In terms of the shallow-water convergence
of the scaled ILW \eqref{sILW1} to KdV \eqref{KDV}
as $\dl \to 0$, 
the second author \cite{Li24} established
it in $H^s(\M)$ for $s \ge \frac 12$ on both geometries.\footnote{In 
\cite{Li24}, the main statement (\cite[Theorem 1.4]{Li24}) states
shallow-water convergence for $s > \frac 12$.
For the endpoint case $ s= \frac 12$, see \cite[Remark 1.6]{Li24}.}
The shallow-water convergence problem for $s < \frac12$ 
has remained
a challenging open problem.

We now state our main result, 
establishing shallow-water convergence in $L^2(\M)$.

\begin{theorem}
\label{THM:1}

Let $\M = \R$ or $\T$.
Given  $v_0 \in L^2(\M)$, 
let $v$ be the global solution to the KdV equation \eqref{KDV}
with $v|_{t = 0} = v_0$.
Let  $\{v_{0, \dl}\}_{0 <  \dl \le 1}\subset L^2(\M)$
such that 
$v_{0, \dl}$ converges to $v_0$ in $L^2(\M)$ as $\dl \to 0$, 
and 
let $v_\dl$ be the global solution to the scaled ILW equation~\eqref{sILW1} with 
$v_\dl|_{t = 0} = v_{0, \dl}$
constructed in \cite{CLOP}.
Then, as $\dl \to 0 $, 
$v_\dl$ converges to $v$
in 
 $C(\R; L^2(\M))$, 
endowed with the compact-open topology \textup{(}in time\textup{)}.

\end{theorem}

Theorem \ref{THM:1} significantly improves
the previously known results
on shallow-water convergence
and 
completes the well-posedness
and convergence study of ILW
on {\it both} geometries
within the $L^2$-framework
(where strong solutions are known to exist \cite{IS, CLOP}).
As mentioned in Subsection \ref{SUBSEC:1.1}, 
our strategy is based on the following
two ingredients:

\smallskip

\begin{itemize}
\item[(i)]
``weakly uniform (in small depth parameters)'' equicontinuity in $L^2$ of 
solutions to the scaled ILW~\eqref{sILW1}
(see Proposition \ref{PROP:EC} below), 
via the complete integrability approach, 
to treat the high frequency part, and 

\smallskip

\item[(ii)]
a perturbative argument
 based on 
an infinite iteration of normal form reductions
for KdV
to treat the low frequency part.

\end{itemize}

\smallskip

\noi
We
present 
 a proof of Theorem \ref{THM:1}
 in 
 Section \ref{SEC:PF}.

\medskip

Let us first discuss the first ingredient (i)
on 
weakly uniform (in small $\dl \ge 0$) equicontinuity in $L^2$
via  the complete integrability approach.
Before proceeding further, let us recall the notion of equicontinuity.

\begin{definition}\label{DEF:EC}
\rm 
Let $s \in \R$.
We say that a bounded set $E \subset H^s(\M)$ is equicontinuous
if 
\begin{align}
 \lim_{\eps \to 0+} \sup_{f \in E} \sup_{|y|< \eps} 
\| f (\,\cdot + y) - f(\,\cdot\,)\|_{H^s} = 0.
\label{EEC1}
\end{align}

\noi
By Plancherel's identity, \eqref{EEC1} is equivalent to 
\begin{align*}
 \lim_{N\to \infty} \sup_{f \in E}
 \| \P_{> N}f\|_{H^s} = 0
\end{align*}

\noi
where $\P_{> N}$ is the frequency projector onto frequencies
$\{ |\xi| > N\}$.

\end{definition}

In a recent work \cite{HKV}, Harrop-Griffiths, Killip, and Vi\c{s}an
rediscovered the following Lax pair for ILW~\eqref{ILW1}:\footnote{The convention in \cite{HKV}
is slightly different from ours.  The solution $q$ in \cite{HKV} corresponds to $-u$ in our setting.}
\begin{align}
\begin{split}
L_{u(t)} & = - i \dx + \tfrac 1{2\dl}(e^{2 i \dl \dx} - 1) + u(t), \\
P_{u(t)} & = - \tfrac  1 \dl \dx - i \dx^2 + \tfrac 1\dl \dx e^{2 i \dl \dx}
+ i (\Tdl\dx u)(t) + \big(\dx u(t) + u(t) \dx \big), 
\end{split}
\label{LA1}
\end{align}

\noi
which was originally introduced
in \cite{GK}, 
and made an important progress
in the shallow-water regime.
Note that $u$ is a solution to \eqref{ILW1} if and only if
\[ \frac d{dt} L_{u(t) } = \big[P_{u(t) }, L_{u(t) }\big].\]

\noi
See \cite[(8)]{HKV}.
As in the KdV case \cite{KVZ}, 
the central object is the following (logarithm
of the renormalized) perturbation determinant:
\begin{align}
\al(\kk; u) =  \sum_{j = 2}^\infty
 \frac{(-1)^j}{j} \tr\Big\{
\big(\sqrt{R_\dl(\kk)}u
\sqrt{R_\dl(\kk)}\big)^j
\Big\}, 
\label{LA2}
\end{align}

\noi
where $R_\dl(\kk)$ denotes the free resolvent given by\footnote{Here, 
we changed the notation from \cite{HKV} to emphasize
the $\dl$-dependence.
The factor $(-1)^j$ in \eqref{LA2}
appears due to the change of variable $q = -u$
from the convention in \cite{HKV}.}
\begin{align}
R_\dl(\kk) = \big(-i \dx + \tfrac 1{2\dl}(e^{2i \dl \dx} - 1) + \kk\big)^{-1}.
\label{LA3}
\end{align}

\noi
Under a smallness condition, 
the quantity $\al (\kk; u)$ is conserved under the flow of ILW~\eqref{ILW1};
see
Lemma \ref{LEM:AL1}.
Using suitable weighted averages (in $\kk$) of $\al(\kk; u)$, 
Harrop-Griffiths, Killip, and Vi\c{s}an \cite{HKV} established
the following uniform (in small $\dl$) equicontinuity of solutions
to the scaled ILW~\eqref{sILW1} in negative Sobolev spaces:

\begin{oldtheorem}\label{THM:B}
Let $\M = \R$ or $\T$
and  $- \frac 12 < s < 0$.
Then, given  $K > 0$, there exists $\dl_K  = \dl_K(s)> 0$
such that the following holds.
Suppose that a set $E \subset H^\infty(\M)$
is bounded and equicontinuous in $H^s(\M)$
such that $\|f \|_{H^s}\le K$ for all $f \in E$.
Then, denoting by the set 
\begin{align*}
E_\dl  = \big\{v_\dl(t; f): t \in \R, \, f \in E \big\}
\end{align*}

\noi
the orbits emanating from $E$ under the flow of the scaled ILW \eqref{sILW1}
with the depth parameter $\dl$,
the set $\bigcup_{0 < \dl \le \dl_K} E_\dl$ is equicontinuous in $H^s(\M)$.
Here, $v_\dl(t; f)$ denotes the \textup{(}smooth\textup{)} solution to the scaled ILW \eqref{sILW1}
\textup{(}with the depth parameter $\dl$\textup{)}
with initial data $f$, evaluated at time $t$.

\end{oldtheorem}

See \cite[Theorem 1.2]{HKV}, where they also stated an a priori
$H^s$-bound on a solution to the scaled ILW \eqref{sILW1}.
In the same paper, they also establish an analogous
uniform (in large $\dl$)
 equicontinuity of solutions
to  ILW~\eqref{ILW1} in $H^s(\M)$, $- \frac 12 < s \le 0$;
see \cite[Theorem 1.1]{HKV}.

We note that, unlike the deep-water regime (\cite[Theorem 1.1]{HKV}), 
Theorem \ref{THM:B}
(namely \cite[Theorem~1.2]{HKV}) misses the endpoint case $s= 0$.
This is due to the fact that the convergence of $\kk F(\xi; \kk, \dl)$
(see \eqref{AL6} for the definition of $F$) as $\kk \to \infty$
is uniform only in   $\dl$ away from  $\dl = 0$ (see \cite[Lemma 4.2]{HKV})
and such uniformity seems false for $0 < \dl \ll 1$.

Our first key ingredient for proving Theorem \ref{THM:1}
is the following
weakly uniform (in small $\dl \ge 0$) equicontinuity in $L^2$
of 
solutions to the scaled ILW~\eqref{sILW1}.
Hereafter, 
 it is understood that the $\dl = 0$ case corresponds to solutions to KdV \eqref{KDV}.

\begin{proposition}\label{PROP:EC}
Let $\M = \R$ or $\T$.
Suppose that a set $E \subset L^2(\M)$
is bounded and equicontinuous in $L^2(\M)$
such that 
\begin{align*}
\sup_{f \in E}
\|  f\|_{L^2} \le K
\end{align*}

\noi
for some $K > 0$.
Then, given any small $\eps > 0$, there exist
$N_1 = N_1(\eps, K, E) \gg 1$
and small $\dl_1 = \dl_1(\eps, K, E) > 0$ such that 
\begin{align}
\sup_{0 \le \dl \le \dl_1}
 \sup_{t\in \R} \sup_{f \in E}
 \| \P_{> N}v_\dl(t; f) \|_{L^2} < \eps
\label{LA5}
\end{align}

\noi
for any $N \ge N_1$, 
 where $v_\dl(t; f)$ is as in Theorem \ref{THM:B}.

\end{proposition}

The main approach is to use the positive normalized  weight
(see \eqref{weight1})
associated with the quadratic term of $\al(\kk;u)$
at a single value of $\kk$,
together with conservation of $\al$ in \eqref{LA2} and of the $L^2$-norm.
We present a proof of Proposition \ref{PROP:EC} 
in Section~\ref{SEC:EC}.

\medskip

Theorem \ref{THM:1} 
on shallow-water convergence
follows once we prove the following statement;
given any $T > 0$ and any small $\eps > 0$, there exists $\dl_* = \dl_*(\eps, T)> 0$ such that 
\begin{align}
\|v-  v_\dl \|_{C_T L^2_x} <  \eps
\label{LA6}
\end{align}

\noi
for any $0 < \dl < \dl_*$, 
where 
$ v = v(v_0)$ and $v_\dl = v_\dl(v_{0, \dl})$ 
denote the solutions to 
KdV~\eqref{KDV} 
and 
the scaled ILW \eqref{sILW1}
(with the depth parameter $\dl$)
with initial data
 $ v|_{t = 0} = v_0$
and $v_\dl|_{t = 0} = v_{0, \dl}$, respectively.
In view of the assumed convergence of $v_{0,\dl}$
to $v_0$ in $L^2$,
it follows from Remark~\ref{REM:convEC} that 
there exist $N = N(\eps,v_0) \gg 1$ and small 
$\dl_1 = \dl_1\big(\eps, v_0, \{v_{0, \dl}\}_{0 < \dl \le 1}\big)> 0$
such that
\begin{align*}
\| \P_{> N} (v - v_\dl)\|_{C_T L^2_x} <  \tfrac \eps2
\end{align*}

\noi
for any $0 < \dl < \dl_1$.
Then, 
the bound
\eqref{LA6} follows
once we prove
that, given any small $\eps > 0$, 
there  exists
 small $\dl_2 = \dl_2\big(\eps, T, N, v_0, 
  \{v_{0, \dl}\}_{0 < \dl \le 1}
 \big) > 0$ such that 
\begin{align}
\| \P_{\le N} (v - v_\dl)\|_{C_T L^2_x} <  \tfrac \eps2 
\label{LA7}
\end{align}

\noi
for any $0 < \dl < \dl_2$, where
we may need to take possibly larger $N = N(\eps, T, v_0)$.
Here, $\P_{\le N}$ denotes the
frequency projector  onto frequencies
$\{ |\xi| \le N\}$.
In Section \ref{SEC:NF}, 
we implement a novel {\it perturbative} argument
based on an infinite iteration of normal form reductions
for KdV and prove the bound \eqref{LA7}
(at the level of the interaction representations;
see 
Proposition~\ref{PROP:low}
for a precise statement).
Here, the word ``perturbative''
means that 

\smallskip

\begin{itemize}
\item
we apply a normal form reduction 
(namely, integration by parts in time) only 
to the KdV multilinear dispersive terms, 
and 

\smallskip

\item
whenever a time derivative hits 
the solution $v_\dl$ to the scaled ILW \eqref{sILW1}, 
we simply replace the trilinear dispersive term
$e^{i t \Xi_\dl(\xi, \xi_1, \xi_2)}$
for the scaled ILW \eqref{sILW1}
by 
the trilinear dispersive term
$e^{i t \Xi_\KDV(\xi, \xi_1, \xi_2)}$
for KdV~\eqref{KDV}, 
where
$\Xi_\dl(\xi, \xi_1, \xi_2)$
and $\Xi_\KDV(\xi, \xi_1, \xi_2)$
are the resonance functions
(for the scaled ILW and KdV, respectively)
as in \eqref{Xi1}, 
and
control the difference 
$e^{i t \Xi_\dl(\xi, \xi_1, \xi_2)}
- e^{i t \Xi_\KDV(\xi, \xi_1, \xi_2)}$, 
corresponding to the error term 
$\EE_\dl^{(j)}$
defined in 
\eqref{ER1} and \eqref{NF13};
see Proposition \ref{PROP:NN}\,(iii).

\end{itemize}

\noi
Moreover, in handling the high frequency parts
on each multilinear term, arising in normal form reductions, 
we apply
 the weakly uniform 
equicontinuity (Proposition \ref{PROP:EC})
(but with much higher frequency cutoff; see \eqref{PE7c}).
See Sections~\ref{SEC:NF}, 
\ref{SEC:per}, 
and  \ref{SEC:Euc} for further details.


\begin{remark}\label{REM:high}\rm
(i) 
Theorem \ref{THM:1}
establishes shallow-water convergence in $L^2(\M)$.
We expect that Theorem \ref{THM:1} extends to the range $0 < s <  \frac 12$.
More precisely, by employing 
the differencing technique as in \cite[Lemma 3.5]{KVZ}
or the approach as in \cite{OW1}, 
we expect that 
Proposition~\ref{PROP:EC}
extends to the 
range $0 < s <  \frac 12$
together with a uniform (in small $\dl$) a priori bound
on the $H^s$-norm of solutions to the scaled ILW~\eqref{sILW1}.
As for Proposition \ref{PROP:low}, 
our perturbative normal form approach
easily extends to the case $0 < s  < \frac 12 $, 
provided that 
Proposition~\ref{PROP:EC}
and 
the aforementioned uniform (in small $\dl$) a priori $H^s$-bound
hold for
$0 < s  < \frac 12 $.
For simplicity of the presentation, however, 
we do not pursue this issue further in this paper.

For $- \frac 12 < s < 0$, 
the question of shallow-water convergence in $H^s(\M)$
remains a challenging open problem.
In this regime, 
  uniform (in small $\dl$) equicontinuity of solutions
to the scaled ILW~\eqref{sILW1} 
holds true in $H^s(\M)$ (Theorem \ref{THM:B}).
However, the normal form approach completely breaks down
in negative Sobolev spaces.
It would be of interest to investigate this issue, 
by further exploiting  the completely integrable structure
of the equation (such as the Lax pair structure in \eqref{LA1}).

\medskip

\noi
(ii) 
By adapting the approach in 
 \cite{Li24}
to the modified energy method developed in  \cite{MPV, MPV2}, 
we expect that shallow-water convergence in $H^s(\M)$, $\M = \R$ or $\T$,  follows
for  the range $\frac 14 < s < \frac 12$
without using any input from the complete integrability of ILW.

\end{remark}

\begin{remark}\rm
(i) In \cite{BO2}, Bourgain proved global 
well-posedness of KdV \eqref{KDV} in $L^2(\M)$, $\M = \R$ or $\T$, 
where solutions are shown 
to be unconditionally unique
(namely, uniqueness holds in the entire class $C(\R; L^2(\M))$)
in later works
\cite{Zhou, BIT}.\footnote{In \cite{BIT}, Babin, Ilyin, and Titi
presented a new approach to construct solutions to KdV in $L^2(\T)$, without referring to unconditional uniqueness
in $L^2(\T)$. However, since their construction does not involve any 
auxiliary function spaces, it yields unconditional uniqueness
of solutions to KdV \eqref{KDV} in $L^2(\T)$.
This point was first noted in \cite{KO, GKO}.}
This is the reason that  we did not specify how
the $L^2$-solution $v$ to KdV \eqref{KDV}
was constructed in Theorem \ref{THM:1}.

\medskip

\noi
(ii) 
As mentioned above, Zhou \cite{Zhou}
proved unconditional uniqueness in $L^2(\R)$
of solutions to KdV \eqref{KDV} on the real line
by working within the framework of the Fourier restriction norm method.
In Sections \ref{SEC:NF}, \ref{SEC:per}, and \ref{SEC:Euc}, we implement an infinite iteration
of normal form reductions in $L^2(\M)$.
We expect that, by elaborating the argument presented
in Sections~\ref{SEC:NF} and~\ref{SEC:Euc}, it would provide
an alternative proof of the unconditional uniqueness in $L^2(\R)$
of solutions to KdV~\eqref{KDV} on the real line.
We, however, do not pursue this issue in this paper.

\end{remark}

\begin{remark}\label{REM:NF1} \rm

(i)
In a seminal work \cite{BIT}, Babin, Ilyin, and Titi
proved unconditional well-posedness of the periodic KdV in $L^2(\T)$
by implementing two iterations of normal form reductions.\footnote{In \cite{BIT}, 
the authors called their approach
 ``differentiation by parts'' without  addressing a normal form method.
This association with normal form reductions was made in a later work \cite{GKO}.}
As mentioned above, 
 our proof of  the low frequency estimate \eqref{LA7}
 (more precisely, Proposition \ref{PROP:low})
is based on (a perturbative argument via) normal 
form reductions for KdV, 
and,  as such, two iterations indeed allow
us to prove 
Proposition \ref{PROP:low} in the periodic setting.
As seen in the modified KdV case \cite{KO, KOY},
however, an infinite iteration seems to be necessary 
in the real line case.
In order to present a unified
approach for both geometries, 
we  decided to 
 present a proof of Proposition \ref{PROP:low}
by implementing an infinite iteration of normal form reductions, 
even in the periodic setting.

\medskip

\noi
(ii) 
In \cite{Ki2}, 
Kishimoto 
introduced an iterative approach 
to an infinite iteration of normal form reductions, 
where one applies basic nonlinear\footnote{bilinear in our case.} estimates;
see also \cite{KOY}.
It seems that such an iterative approach, relying
on basic bilinear estimates,  is not suitable for our problem, 
since 
our argument involves careful analysis
on derivative losses and multilinear dispersions
from  different generations.
See, for example,~\eqref{PQ9}, 
\eqref{PE3}, 
and 
analysis on $G_{\TT_j}$ defined in \eqref{RE4}.

\medskip

\noi
(iii) 
The normal form method 
presented in Section \ref{SEC:NF}
corresponds to the so-called 
 Poincar\'e-Dulac normal form reduction;
see \cite[Section 1]{GKO}.
The Poincar\'e-Dulac normal form reductions have been 
used in various contexts:
small data global well-posedness \cite{Shatah}, 
unconditional uniqueness \cite{BIT, KO, GKO, OW, Ki2, Ki3, MP2, FLZ}
also including the stochastic and modulated settings
\cite{GLLO1,  OSW}, 
nonlinear smoothing and the nonlinear Talbot effect \cite{ETz, ETz3, ETz2, CD}, 
existence of global attractors~\cite{ETz3}, 
reducibility~\cite{CGKO}, 
improved energy estimates 
in both the deterministic and probabilistic contexts
\cite{TT, GO, OW2, OST}, 
and numerical schemes~\cite{CFO}.
Our proof of Theorem~\ref{THM:1} 
presents  a novel  application 
of the 
 Poincar\'e-Dulac normal form reductions.
See Section~\ref{SEC:NF}
for a further discussion.

\end{remark}

\begin{remark}\label{REM:Ben}\rm

In \cite{CHLO}, 
the first three authors with B.\,Harrop-Griffiths
investigated 
regularity properties of the solution map
for the scaled ILW \eqref{sILW1} 
 on the real line.
By decomposing the dynamics into 
the low frequency part and the residual part, 
they  showed that, when the depth parameter is sufficiently small,  
the solution map for the low frequency part
is analytic  in $L^2(\R)$, 
while the 
 solution map for the residual part fails to be~$C^2$.
Moreover, they established
shallow-water convergence  in $L^2(\R)$
of  the low frequency dynamics
 to  KdV.
This explained 
how this regularity gain 
of the solution map 
propagates in frequency,
namely, from the low frequency dynamics 
supported on $\{|\xi|\le \dl^{-1}\}$ 
to KdV, 
which 
 illustrates another aspect of the singular nature
of the 
 shallow-water limit.
We emphasize that this propagation of the regularity 
gain of the solution map in the shallow-water limit
is a phenomenon specific to  the real line case.
See \cite[Remark 1.3]{CHLO}
for a further discussion.

\end{remark}

\section{Preliminaries}

\subsection{Notations}
We use $A\les B$ if there exists $C>0$ such that $A \le CB$, and $A\sim B$ if $A \les B$ and $B \les A$. 
We also write $A \ll B$ if $A \le c B$ for some small constant $c>0$. 
We may write  $\les_{\al}$, $\sim_{\al}$, and $\ll_\al$ to 
emphasize the dependence  on an external parameter $\al$.
We use $C>0$ to denote various constants, which may vary line by line,
and we may write $C_\al$ or $C(\al)$ to  signify  dependence on an external parameter
$\al$.

Given $\M = \R$ or $\T$, 
let   $\ft \M$ denote the Pontryagin dual of $\M$, 
i.e.~
\begin{align*}
\ft \M = \begin{cases}
 \R, & \text{if }\M = \R, \\
 \Z, & \text{if } \M = \T.
\end{cases}
\end{align*}

\noi
When $\ft \M = \Z$, 
we endow it with the counting measure.
Our conventions for the Fourier transform are as follows:
\[ \ft f(\xi) = \frac{1}{\sqrt{2\pi}} \int_\R f(x) e^{-ix \xi} dx
\qquad\text{and}\qquad
 f(x) = \frac{1}{\sqrt{2\pi}} \int_\R \ft f(\xi) e^{ix \xi} d\xi
\]

\noi
for functions on the real line $\R$
and 
\[ \ft f(\xi) =  \frac 1{\sqrt{2\pi}}\int_0^{2\pi } f(x) e^{-ix \xi} dx
\qquad\text{and}\qquad
 f(x) = \frac 1{\sqrt{2\pi}}
 \int_{ \Z}\ft f(\xi) e^{ix \xi} d\xi 
\]

\noi
for functions on the circle $\T$, 
where the latter integral on $\Z$ is with respect to the counting measure:
\begin{align*}
\int_{ \Z} f(\xi) d\xi 
:= 
\sum_{\xi \in   \Z} f(\xi) .
\end{align*}

\noi
Then, 
Plancherel's identity is expressed as 
\begin{align}
 \|f \|_{L^2(\M)} = \| \ft f \|_{L^2(\ft \M)}.
\label{int2}
 \end{align}

\noi
We also  record the following identity:
\begin{align}
\ft{fg}(\xi) = \frac{1}{\sqrt{2\pi}} \int_{\ft \M} \ft f(\eta) \ft g(\xi - \eta) d\eta, 
\quad \xi \in \ft \M.
\label{int3}
\end{align}

\noi
On $\T$, we work with mean-zero functions in Sections
 \ref{SEC:NF} and  \ref{SEC:per}.
For this purpose, we set 
 \[\Z_* : = \Z \setminus\{0\}.\]


Given $N\in\N$, $\P_{\le N}$ denotes the Dirichlet projector onto spatial frequencies 
$\{ |\xi| \le N\}$ defined by 
\begin{align}
\P_{\le N} f(x )  =  \frac{1}{\sqrt{2 \pi}} \int_{\ft \M} \ind_{|\xi|\le N} \cdot \ft{f}(\xi) e^{ix\xi}d\xi. 
\label{P1}
\end{align}

\noi
We set $\P_{> N} = \Id - \P_{\le N}$.
Given dyadic $N \in 2^\Z$, we use $\P_N$
to denote the Littlewood-Paley projector
onto frequencies $\{|\xi|\sim N\}$.

We use $S(t)$ and $S_\dl(t)$, $0 < \dl < \infty$, 
to denote the linear propagators
for KdV \eqref{KDV} and the scaled ILW \eqref{sILW1}, respectively:
\begin{align}
S(t) = e^{- t \dx^3} 
\qquad 
\text{and}
\qquad 
S_\dl(t) = e^{t \Gd \dx^2 }, 
\label{lin1}
\end{align}

\noi
where $\Gd$ is as in \eqref{sILW1}; see also \eqref{GG1} and \eqref{GG2}.

\subsection{On the basic  multiplier}

We recall some results in \cite{Li24, LOZ, CLO} on the dispersive operator $\Gd\dx^2$
in \eqref{sILW1}; see also \eqref{GG1} and \eqref{GG2}.
Given $0<\dl < \infty$, let $\L_\dl(\xi)$ be the multiplier for $\Gd\dx$ given by
\begin{align}
\L_\dl(\xi)
= i   \xi \ft{\Gd}(\xi)
=  \tfrac 3\dl \big( \xi \coth(\dl \xi) - \dl^{-1}\big),
\label{L1}
\end{align}

\noi
where $ \ft{\Gd}(\xi)$ denotes the multiplier for $ \Gd$.

\begin{lemma}\label{LEM:multip}
Let $0< \dl< \infty$. Then, the following statements hold.

\smallskip
\noi{\rm(i)} $0 < \L_\dl(\xi) \le \min( \xi^2, \tfrac3\dl |\xi| )$.

\smallskip
\noi{\rm(ii)} We have
\begin{align*}
\L_\dl(\xi) 
&
= 6 \xi^2 \sum_{k=1}^\infty \frac{1}{k^2 \pi^2 + \dl^2 \xi^2}
\le \xi^2
\end{align*}

\noi
for any $0 < \dl < \infty$.
Moreover, for fixed $\xi \in \ft \M$, $\L_\dl(\xi)$ is decreasing in $\dl$, and 
\begin{align*}
\lim_{\dl\to 0} \L_\dl(\xi) = \xi^2, \quad \xi  \in \ft \M. 
\end{align*}

\smallskip
\noi{\rm(iii)} We can also write $\L_\dl(\xi)$ as follows
\begin{align}
\L_\dl(\xi) 
=  \xi^2 \big( 1 - h(\dl, \xi) \big), 
\label{LW1}
\end{align}
where $h$ is given by 
\begin{align}
\label{hdef}
h(\dl, \xi)
&
= 6 \dl^2 \xi^2 \sum_{k=1}^\infty \frac{1}{k^2 \pi^2 (k^2 \pi^2 + \dl^2 \xi^2)} \in (0,1], 
\\
\label{limh}
\sum_{\xi\in\Z_*} h^2(\dl, \xi) 
&= \infty
, 
\quad 
\text{and}
\quad 
\lim_{\dl\to 0} h(\dl, \xi) =0. 
\end{align}

\end{lemma}

We note that \eqref{hdef} implies the following bound:
\begin{align}
h(\dl, \xi)  \le \dl^2 |\xi|^2
\label{L3}
\end{align}

\noi
for any $\xi \in \ft \M$.

\section{Shallow-water convergence in $L^2$}
\label{SEC:PF}

In this section, we present a proof of our main result (Theorem~\ref{THM:1})
on 
shallow-water convergence of $L^2$-solutions to the scaled ILW \eqref{sILW1} to 
that to KdV \eqref{KDV}. 
Given  $v_0 \in L^2(\M)$, 
let $v  \in C(\R; L^2(\M))$ be the  solution to KdV~\eqref{KDV}
with $v|_{t = 0} = v_0$.
Let  $\{v_{0, \dl}\}_{0 <  \dl \le 1}\subset L^2(\M)$
such that 
$v_{0, \dl}$ converges to $v_0$ in $L^2(\M)$ as $\dl \to 0$, 
and 
let $v_\dl \in  C(\R; L^2(\M)) $ be the  solution to the scaled ILW~\eqref{sILW1} with 
$v_\dl|_{t = 0} = v_{0, \dl}$.
We then denote by  $\vv$ and $\vvd$
 the interaction representations  of  $v$ and $v_\dl$ defined by 
\begin{align}
\label{inter}
\vv(t)  = S(-t) v(t)
\qquad 
\text{and}
\qquad 
\vvd(t) = S_\dl(-t) v_\dl(t) , 
\end{align}

\noi
respectively, 
where
$S(t)$ and $S_\dl(t)$ are as in \eqref{lin1}.

As mentioned in Section \ref{SEC:1}, 
the high frequency part is handled by 
the weakly uniform (in small $\dl$)
equicontinuity (Proposition \ref{PROP:EC})
whose proof is presented in 
Section \ref{SEC:EC}.
For the low frequency part of the difference, we have the following statement.

\begin{proposition}
\label{PROP:low}
Let $\M = \R$ or $\T$ and fix $T > 0$.
Given  $v_0 \in L^2(\M)$, 
let $v$, $\{v_{0, \dl}\}_{0 <  \dl \le 1}$, and $v_\dl$ be as above.
Then, 
given $\eps>0$, there exists $N_2=N_2(\eps,  T, v_0) \gg1$
such that
the following statement holds; given any $N \ge N_2$, 
there exists
small 
$\dl_2 = \dl_2\big(\eps, T,  N, v_0, \{v_{0, \dl}\}_{0 <  \dl \le 1}\big) >0$
such that 
\begin{align}
\label{low00}
\| \P_{\le N  } (\vv - \vv_\dl ) \|_{C_T L^2_x}  
&
< \eps
\end{align}

\noi
for any $0 < \dl \le \dl_2$, 
where $\vv$ and $\vvd$ are the interaction representations 
of $v$ and $v_\dl$, respectively, defined in \eqref{inter}.
\end{proposition}

In Section \ref{SEC:NF}, we present a proof of Proposition \ref{PROP:low}
by 
implementing
a perturbative argument, based on an infinite iteration of normal form reductions
for KdV.
We note that, in order to control the high frequency parts
on each multilinear term, arising in normal form reductions, 
we need to apply
 the weakly uniform 
equicontinuity (Proposition \ref{PROP:EC}), 
but with much higher frequency cutoff; see \eqref{PE7c}.

\medskip

We now present a proof of Theorem \ref{THM:1} by 
assuming Remark~\ref{REM:convEC} 
(see also Proposition~\ref{PROP:EC}) and Proposition \ref{PROP:low}.

\begin{proof}[Proof of Theorem \ref{THM:1}]
Given  $v_0 \in L^2(\M)$, 
let $v$, $\{v_{0, \dl}\}_{0 <  \dl \le 1}$, and $v_\dl$ be as above.
Let  $\vv$ and $\vvd$ be the interaction representations 
of $v$ and $v_\dl$ as in \eqref{inter}, respectively.

Fix large $T > 0$ and small $\eps> 0$. 
Our goal is to prove the bound
\eqref{LA6}.
From \eqref{inter}
and the unitarity in $L^2(\M)$ of $S_\dl(t)$, we have 
\begin{align}
\begin{split}
\| v - v_\dl \|_{C_T L^2_x} 
&
\le 
\sup_{|t|\le T}\| (S(t) - S_\dl(t) ) \P_{\le N} \vv(t) \|_{L^2_x} \\
&\quad  +\| \P_{\le N} (\vv - \vvd ) \|_{C_T L^2_x} 
+
\| \P_{> N} (v- v_\dl  ) \|_{C_T L^2_x}.
\end{split}
\label{PF1}
\end{align}

\noi
From 
Remark \ref{REM:convEC}, 
there exist $N_1 = N_1(\eps, v_0) \gg1$
and small $\dl_1 = \dl_1\big(\eps, v_0,  \{v_{0, \dl}\}_{0 < \dl \le 1}\big)> 0$ 
such that 
\begin{align}
\| \P_{> N} (v- v_\dl  ) \|_{C_T L^2_x} < \frac \eps 3
\label{PF2}
\end{align}

\noi
for any $N \ge N_1$ and $0 < \dl \le \dl_1$.
From Proposition \ref{PROP:low}, 
we can then choose 
 $N=N(\eps, T,  v_0) \gg N_1$
and small 
$\dl_2 = \dl_2\big(\eps, T, N,  v_0,  \{v_{0, \dl}\}_{0 < \dl \le 1}\big) >0$
such that 
\begin{align}
\| \P_{\le N  } (\vv - \vv_\dl ) \|_{C_T L^2_x}  
< \frac \eps3
\label{PF3}
\end{align}

\noi
for any $0 < \dl \le \dl_2$.

We now consider the first term on the right-hand side of \eqref{PF1}.
From \eqref{lin1}, the mean value theorem, 
\eqref{L1}, 
and 
Lemma \ref{LEM:multip}\,(ii), we have
\begin{align}
\big| \Ft_x\big( (S(t) -  S_\dl(t) ) f \big) (\xi) \big|
\le 
|t\xi| |\xi^2 - \L_\dl(\xi) |
|\ft f(\xi)|
\label{LW2}
\end{align}
as $\dl\to0$
for each fixed $t \in \R$ and $ \xi \in \ft \M$.
Then, 
in view of  \eqref{LW2}, \eqref{LW1}, \eqref{L3}, 
and the  conservation of the $L^2$-norm, we can choose
 small 
$\dl_3 = \dl_3\big(\eps, T, N,   \|v_0\|_{L^2}\big) >0$
such that 
\begin{align}
\sup_{|t|\le T}\| (S(t) - S_\dl(t) ) \P_{\le N} \vv(t) \|_{L^2_x} 
\le T \dl^2 N^5 \|v_0\|_{L^2}< \frac \eps 3
\label{PF5}
\end{align}

\noi
for any $0 < \dl \le \dl_3$.

Therefore, putting \eqref{PF1}, 
\eqref{PF2}, \eqref{PF3}, 
 and \eqref{PF5}
together, we obtain
\begin{align*}
\|v-  v_\dl \|_{C_T L^2_x} <  \eps
\end{align*}

\noi
for any $0 < \dl < \dl_*
 := \min(\dl_1, \dl_2, \dl_3)$.
This proves \eqref{LA6}.
\end{proof}

\section{High frequency analysis: weakly uniform equicontinuity}
\label{SEC:EC}

In this section, we establish
the weakly uniform (in small $\dl$) equicontinuity
(Proposition~\ref{PROP:EC}).
For this purpose, we first recall some notations
and results from \cite{KVZ, HKV}.

For an operator $A$ on $L^2(\M)$ with a continuous  integral kernel $K(x,y)$, 
we  define its trace by
\[\tr(A)=\int_{\M} K(x,x) dx.\]

\noi
In particular, if $A$ is a Hilbert-Schmidt operator with an integral kernel $K(x,y)\in L^2(\M^2)$, 
then we have 
\[\tr (A^2) = \iint_{\M^2} K(x,y) K(y,x) dxdy.\]

\noi
We also set
\begin{align*}
\| A \|_{ \If_2}^2 =  \tr(A^* A) = \iint_{\M^2} | K(x,y)|^2 dxdy.
\end{align*}

\noi
In particular, if $A$ is self-adjoint, 
then we have 
\begin{align}
\tr (A^2) = \| A \|_{ \If_2}^2.
\label{AL1}
\end{align}

\noi
Recall from   \cite[Lemma 1.4]{KVZ} that
\begin{align}
|\tr(A_1 \cdots A_k)| \leq \prod_{j = 1}^k \|A_j\|_{\If_2}
\label{AL2}
\end{align}

\noi
for an integer $k \ge 2$.

In the following, we summarize important properties of $\al(\kk; u)$ 
defined in \eqref{LA2}.
We first note  from \eqref{AL1} that    the leading term (corresponding to $j = 2$)
of the series expansion 
of $\al(\kk; u)$ 
in~\eqref{LA2} is given by 
\begin{align}
\frac 12 
\tr\Big\{
\big(\sqrt{R_\dl(\kk)}u
\sqrt{R_\dl(\kk)}\big)^2
\Big\}
= 
\frac 12 
\big\|\sqrt{R_\dl(\kk)}u
\sqrt{R_\dl(\kk)}\big\|_{\mathfrak{I}_2(\M)}^2.
\label{AL3}
\end{align}

\noi
Given $0 < \dl < \infty$, 
define $a_\dl (\xi)$ to be the Fourier multiplier
of the inverse of the free resolvent $R_\dl(\kk)$ in \eqref{LA3}
with $\kk = 0$, given by 
\begin{align}
a_\dl (\xi) = \xi + \tfrac 1{2\dl}(e^{-2\dl \xi} - 1), \quad \xi \in \ft \M.
\label{AL4}
\end{align}

\noi
The next lemma summarizes the basic properties
of the leading term in \eqref{AL3}.
See
\cite[Lemma~2.4 and Proposition 4.1]{HKV}
for the proof.
In the periodic setting, our convention
of the domain and the Fourier transform is different
from that in \cite{HKV}.
In particular, Plancherel's identity \eqref{int2}
and the identity \eqref{int3}
give the factor of $2\pi$ in 
\eqref{AL6} even in the periodic setting;
see, for example,  the proof of 
\cite[Lemma 3.2]{OW1}
for an analogous computation.

\begin{lemma}\label{LEM:AL1}
\textup{(i)}
For $\dl, \kk > 0$, we have 
\begin{align}
\big\|\sqrt{R_\dl(\kk)}u
\sqrt{R_\dl(\kk)}\big\|_{\mathfrak{I}_2(\M)}^2
= \int_{\ft \M} F(\xi; \kk, \dl)|\ft u(\xi)|^2 d\xi, 
\label{AL5}
\end{align}

\noi
where $F(\xi; \kk, \dl)$ is defined by 
\begin{align}
F(\xi; \kk, \dl)
= \frac 1{2\pi}\int_{\ft \M} \big(a_\dl(\eta) + \kk\big)^{-1}
\big(a_\dl(\xi+ \eta) + \kk\big)^{-1} d\eta.
\label{AL6}
\end{align}

\noi
Then, the function
$F(\xi; \kk, \dl)$ is smooth, non-negative, and even in $\xi$.

\smallskip

\noi
\textup{(ii)}
Let $u$ be a smooth solution to ILW \eqref{ILW1}.
Then, we have
\begin{align*}
\frac d {dt} \al(\kk; u(t)) = 0, 
\end{align*}

\noi
provided that $\kk > 0$ satisfies 
\begin{align}
\big\|\sqrt{R_\dl(\kk)}u(0)
\sqrt{R_\dl(\kk)}\big\|_{\mathfrak{I}_2(\M)}^2
< \tfrac 1{36}.
\label{AL8}
\end{align}

\end{lemma}

Next, we introduce the normalized weight:
\begin{align}
w_{\dl,\mu}(\xi)
= \frac{F(\xi; \kk, \dl)}{F(0; \kk, \dl)}
\qquad
\label{weight1}
\end{align}

\noi
for  $\mu \ge 1$, $0 < \dl \le \mu^{-1}$,
and $\kk = \dl\mu^2$.
Also, set 
\begin{align}
\beta_{\dl,\mu}
= \frac{\dl^2}{9} F(0; \kk, \dl).
\label{weight1a}
\end{align}

\noi
The following lemma summarizes
basic bounds for $w_{\dl, \mu}$ and $\be_{\dl, \mu}$.

\begin{lemma}\label{LEM:weight}
Let $\mu \ge 1$, $0 < \dl \le \mu^{-1}$,
and $\kk = \dl\mu^2$ as above.
Then, we have 
\begin{align}
& 0 < w_{\dl,\mu} \le 1,
\qquad 
1 - w_{\dl,\mu}(\xi)
\les \min(1, \mu^{-2}\xi^2), 
 \qquad \text{and}
\qquad 
\beta_{\dl,\mu} \les \mu^{-3}.
\label{weight2}
\end{align}

\noi
Moreover, there exists  $A \ge 2$, independent of $\dl$ and $\mu$,  such that
\begin{align}
1 - w_{\dl,\mu}(\xi) \ge \tfrac12
\label{weight3}
\end{align}

\noi
for any $|\xi| \ge A\mu$.

\end{lemma}

\begin{proof}
Set $r(\eta) = (a_\dl(\eta) + \kk)^{-1}$.
Since $\kk = \dl \mu^2$, we have 
\begin{align}
0 < r(\eta) \le \frac 1\kk = \frac 1{\dl \mu^2}.
\label{weight3a}
\end{align}

\noi
The first bound in \eqref{weight2}
follows from 
 \eqref{AL6} and Cauchy-Schwarz's inequality.

Next, we prove the second bound in \eqref{weight2}.
A direct computation with \eqref{AL6} and~\eqref{weight1}
yields the following identity:
\begin{align}
1 - w_{\dl,\mu}(\xi)
= \frac{\|r(\,\cdot + \xi) - r\|_{L^2(\ft\M)}^2}
{2\|r\|_{L^2(\ft\M)}^2}.
\label{weight4}
\end{align}

\noi
We first consider the case $\M = \R$.
Since $a_\dl'(\eta) = 1 - e^{-2\dl\eta}$ and
$a_\dl''(\eta) = 2\dl(1 - a_\dl' (\eta))$,
integration by parts 
(note that all boundary terms vanish since $r(\eta) = O(\eta^{-1})$
as $\eta \to \infty$ and $r$ decays exponentially
as $\eta \to -\infty$)
yields
\begin{align}
\int_\R |r'|^2 d\eta
& = \frac13 \int_\R a_\dl'' \, r^3 d\eta
= \frac{2\dl}{3} \int_\R r^3 d\eta
\le \frac{2}{3\mu^2} \int_\R r^2 d\eta, 
\label{weight4a}
\end{align}

\noi
where 
we used $\int_\R a_\dl' r^3 d\eta = 0$
in the second step
and \eqref{weight3a} in the last step.
Then, the second bound in \eqref{weight2}
follows from 
\eqref{weight4}, 
the fundamental theorem of calculus, 
and 
 \eqref{weight4a}.
Next, consider the case $\M = \T$. 
 Given $\xi  \in \Z$, the fundamental theorem of calculus 
 and Cauchy-Schwarz's inequality give
\begin{align}
\sum_{n \in \Z} |r(n + \xi) - r(n)|^2
\le \xi^2 \int_\R |r'(\eta)|^2 d\eta.
\label{weight4b}
\end{align}

\noi
Moreover, noting that $r$ is increasing on $(-\infty,0]$
and decreasing on $[0,\infty)$, we have
\begin{align}
\int_\R r^2 d\eta
\le \sum_{n \in \Z} r(n)^2 + r(0)^2
\le 2 \sum_{n \in \Z} r(n)^2.
\label{weight4c}
\end{align}

\noi
Hence, the second bound in  \eqref{weight2} 
when $\M = \T$
follows from \eqref{weight4}, 
\eqref{weight4b},
\eqref{weight4a}, and~\eqref{weight4c}.

We now carry out some preliminary computations
in order to prove the last bound \eqref{weight2}
and also \eqref{weight3}.
Using 
$e^{-2x}\le 1-2x+2x^2$ for $x \ge0$, 
we  have
\begin{align}
a_\dl(\eta) \le \dl\eta^2
\label{WW1}
\end{align}

\noi
for $\eta \ge 0$.
A similar consideration yields
\begin{align}
a_\dl(\eta) \ge \dl\eta^2
\label{WW2}
\end{align}

\noi
for $\eta \le 0$.
From the fundamental theorem of calculus and
$1-e^{-2x}
\ge1-\frac1{1+2x}
=\frac{2x}{1+2x}$ for $x \ge 0$, we have
\begin{align}
a_\dl(\eta) = \int_0^\eta 1 - e^{-2\dl\tau}d\tau
\ge \frac {1}{1+2\dl \eta}\int_0^\eta 2\dl \tau d\tau 
= \frac {\dl \eta^2 }{1+2\dl \eta}
\label{WW3}
\end{align}

\noi
for  $\eta \ge 0$.

From \eqref{WW1} with $\kk = \dl\mu^2$, we have 
\begin{align}
r(\eta) = \frac1{a_\dl(\eta)+\kk} \ge \frac{1}{2\kk}
\label{WW3a}
\end{align}

\noi
for 
$0 \le \eta \le \mu$.
Since 
 $\dl\mu \le 1$
and $\kk = \dl\mu^2$,
 it follows from \eqref{WW2} and \eqref{WW3}
that  
\begin{align}
a_\dl(\pm\mu) \ge \tfrac \kk 4.
\label{WW4}
\end{align}
Since $a_\dl$ is convex and 
$a_\dl(0) = 0$, 
we have $\eta^{-1}a_\dl(\eta)$ non-decreasing
(where it is interpreted as $a_\dl'(0)$ when $\eta = 0$).
Thus, using \eqref{WW4}, we obtain
\begin{align}
a_\dl(\eta) \ges \tfrac{\kk |\eta|}{\mu}
\quad \text{which yields}\quad 
r(\eta) \les \tfrac{\mu}{\kk |\eta|}
\label{WW5}
\end{align}

\noi
for $|\eta|\ge \mu$.
Hence, from \eqref{WW5}
and the non-negativity of $a_\dl$, 
we obtain
\begin{align}
r(\eta) \les \tfrac 1\kk \min \big(1, \tfrac{\mu}{|\eta|}\big)
\label{WW6}
\end{align}

\noi
for any $\eta \in \R$.

Using \eqref{WW3a} and \eqref{WW6}, 
integration or summation (with $\mu \ge 1$) yields
\begin{align}
\begin{split}
\|r\|_{L^2(\ft\M)}^2 & \sim \mu\kk^{-2}, \\
\|\ind_{|\eta| \ge B\mu} \cdot r\|_{L^2(\ft\M)}^2
& \les B^{-1}\mu\kk^{-2}
\end{split}
\label{weight5}
\end{align}

\noi
for any $B \ge 1$.
Recalling 
$\kk = \dl\mu^2$, 
the last bound in \eqref{weight2}
follows from \eqref{weight1a}
and the first estimate in \eqref{weight5}.

Lastly, 
for 
 $|\xi| \ge 2\mu$, split the integral in \eqref{AL6}
into $|\eta| \ge |\xi|/2$ and its complement.
On the latter set, we have $|\eta + \xi| \ge |\xi|/2$.
Then, by Cauchy-Schwarz's inequality and \eqref{weight5}, we have
\begin{align}
w_{\dl,\mu}(\xi)
\le \frac{2\|\ind_{|\eta| \ge |\xi|/2}\cdot r\|_{L^2(\ft\M)}}
{\|r\|_{L^2(\ft\M)}}
\les \Big(\frac \mu{|\xi|}\Big)^{\frac12}.
\label{WW7}
\end{align}

\noi
Then, 
by choosing $A \ge 2$ sufficiently large, 
the bound 
 \eqref{weight3} follows from \eqref{WW7}.
\end{proof}

\medskip

We are now ready to present a proof of Proposition \ref{PROP:EC}.

\begin{proof}[Proof of Proposition \ref{PROP:EC}]

In the following, we  consider $\dl > 0$.
We can then obtain the KdV case 
by passing to the limit $\dl \to 0$, 
which we discuss toward the end of this proof.
%

Let $E \subset L^2(\M)$
be  bounded and equicontinuous in $L^2(\M)$
such that 
\begin{align}
\sup_{f \in E}
\|  f\|_{L^2} \le K
\label{BL2}
\end{align}

\noi
for some $K > 0$.
Fix small $\eps > 0$.
Then, it follows from the equicontinuity of $E$ in $L^2(\M)$ that there exists 
$M =  M(\eps, E) \gg1 $ such that 
\begin{align}                                     
\sup_{f \in E}
\| \P_{> M} f\|_{L^2} < \frac \eps 2.
\label{BL3}
\end{align}

Let $\mu \ge 1$ and $0 < \dl \le \mu^{-1}$,
and set $\kk = \dl\mu^2$.
Define
\begin{align}
D_{\dl,\mu}(f)
= \int_{\ft\M} (1 - w_{\dl,\mu}(\xi))
|\ft f(\xi)|^2 d\xi, 
\label{defect1}
\end{align}

\noi
where 
$w_{\dl, \mu}$ is as  in \eqref{weight1}.
We first consider smooth initial data $f$ with 
\begin{align}
\|f\|_{L^2} \le K.
\label{ZX0}
\end{align}

\noi
Given $0 < \dl \le 1$, let $v_\dl$ be the solution to the scaled ILW
\eqref{sILW1} (with the depth parameter $\dl$)
with $v_\dl|_{t = 0} = f$.
By setting 
\begin{equation}
u_\dl(t,x) = \tfrac\dl3 v_\dl\big(\tfrac\dl 3t,x\big), 
\label{scale2}
\end{equation}

\noi
we see that $u_\dl$ is 
the solution  to ILW \eqref{ILW1}
with $u_\dl|_{t = 0} = \frac\dl3 f$.
For simplicity of notation, set
\begin{align}
B(t) = \sqrt{R_\dl(\kk)}
u_\dl\big(\tfrac3\dl t\big)\sqrt{R_\dl(\kk)}.
\label{ZX1}
\end{align}

\noi
Then, 
it follows from  \eqref{AL5}, \eqref{weight1}, \eqref{weight1a},
Lemma \ref{LEM:weight}, 
the conservation of the $L^2$-norm, and~\eqref{ZX0}
that
\begin{align}
\|B(t)\|_{\If_2}^2
= \beta_{\dl,\mu}
\int_{\ft\M} w_{\dl,\mu}(\xi)
|\ft v_\dl(t,\xi)|^2 d\xi
\le \beta_{\dl,\mu}K^2
\les \mu^{-3}K^2
\label{defect2}
\end{align}

\noi
for any $t \in \R$.
Now, 
choose $\mu \gg 1$ sufficiently large (depending on $K$)
such  that the right-hand side of \eqref{defect2} is less than $\frac1{36}$, 
verifying  \eqref{AL8}.

Let $\mathcal{R}$ denote the contribution
from higher order terms in \eqref{LA2}, defined by 
\begin{align}
\mathcal{R}(t)
= \sum_{j = 3}^\infty \frac{(-1)^j}{j}\tr\big\{B(t)^j\big\}.
\label{ZX2}
\end{align}

\noi
Then, from  \eqref{AL2} and \eqref{defect2}, we have 
\begin{align}
|\mathcal R(t)|
\les \|B(t)\|_{\If_2}^3
\le \beta_{\dl,\mu}^{\frac32}K^3.
\label{defect3}
\end{align}

\noi
Noting that  $B(t)$ in \eqref{ZX1} is self-adjoint, 
it follows
from \eqref{LA2}, \eqref{ZX2}, 
 \eqref{AL1}, and \eqref{defect2} that
\begin{align*}
\al\big(\kk;u_\dl\big(\tfrac3\dl t\big)\big)
= \frac{\beta_{\dl,\mu}}2
\int_{\ft\M} w_{\dl,\mu}(\xi)
|\ft v_\dl(t,\xi)|^2 d\xi + \mathcal R(t).
\end{align*}

\noi
Then, by 
Lemma \ref{LEM:AL1}\,(ii)
and the conservation of the $L^2$-norm, we obtain the following identity:
\begin{align*}
D_{\dl,\mu}(v_\dl(t)) - D_{\dl,\mu}(f)
= \frac{2}{\beta_{\dl,\mu}}
\big(\mathcal R(t) - \mathcal R(0)\big).
\end{align*}

\noi
Hence, from  \eqref{defect3} and Lemma \ref{LEM:weight},
we  obtain 
\begin{align}
D_{\dl,\mu}(v_\dl(t))
\le D_{\dl,\mu}(f) + CK^3\mu^{-\frac32}
\label{defect5}
\end{align}

\noi
for any $t \in \R$.

Next, consider  general $f \in L^2(\M)$ with $\|f\|_{L^2} \le K$.
Given $J \in \N$, 
let $v_\dl^J$ be the smooth solution to the scaled ILW \eqref{sILW1}
with $v_\dl^J|_{t = 0} = \P_{\le J}f$, 
which in particular satisfies \eqref{defect5}
(with $v_\dl$ and $f$ replaced by $v_\dl^J$ and $\P_{\le J}f$, respectively).
For each fixed $\dl > 0$ and $T > 0$,
it follows from 
the $L^2$-solution theory
in \cite{CLOP} and the scaling~\eqref{scale2}
that 
$v_\dl^J$ converges to $v_\dl$
in $C([-T,T];L^2(\M))$ as $J \to \infty$.
Recalling from Lemma \ref{LEM:weight} that  $0 \le 1 - w_{\dl,\mu} \le 1$,
we see that $D_{\dl,\mu}$ in \eqref{defect1} is continuous on $L^2(\M)$.
Passing to the limit, we obtain \eqref{defect5}
for every $t \in \R$.

Now, let $f \in L^2(\M)$ with $\|f\|_{L^2} \le K$.
Let 
$M \ge 1$, and $N \ge A\mu$,
where $A$ is as in \eqref{weight3}.
From \eqref{defect1}, \eqref{weight3}, \eqref{defect5},
\eqref{BL3}, 
and \eqref{weight2}, we have
\begin{align}
\begin{split}
\|\P_{>N} & v_\dl(t;f)\|_{L^2}^2
\le 2D_{\dl,\mu}(v_\dl(t;f))\\
&\le 2\|\P_{>M}f\|_{L^2}^2
+2\int_{|\xi|\le M}(1-w_{\dl,\mu}(\xi))
|\ft f(\xi)|^2d\xi
+CK^3\mu^{-\frac32}\\
&\le 2\|\P_{>M}f\|_{L^2}^2
+C\frac{M^2}{\mu^2}K^2+CK^3\mu^{-\frac32}
\end{split}
\label{defect6}
\end{align}

\noi
for any $t \in \R$, $\mu \ge \mu_0(K)$,
and $0 < \dl \le \mu^{-1}$.

Next, we  extend \eqref{defect6} to $\dl = 0$.
Fix $f \in L^2(\M)$ with $\|f\|_{L^2} \le K$
and set $f_J = \P_{\le J}f$.
For each fixed $J \in \N$ and $T > 0$,
it follows from the  shallow-water convergence
in the case of smooth data
that 
$v_\dl(\,\cdot\,;f_J)$ converges to  $v_0(\,\cdot\,;f_J)$
in $C([-T, T]; L^2(\M))$ as $\dl \to 0$.
Now, apply~\eqref{defect6} to $v_\dl(\,\cdot\,;f_J)$ and let $\dl \to 0$.
Then, 
by taking  $J \to \infty$,
it follows from the  continuity in $L^2(\M)$ of the KdV flow
that 
 \eqref{defect6} also holds for $\dl = 0$
 (for any  $t \in \R$,
with the same constants, on both $\R$ and $\T$).

For $f \in E$, it follows from \eqref{BL3}
and \eqref{defect6} that
\begin{align*}
\|\P_{>N}v_\dl(t;f)\|_{L^2}^2
\le \frac{\eps^2}{2}
+C\frac{M^2}{\mu^2}K^2+CK^3\mu^{-\frac32}
\end{align*}

\noi
for any $t \in \R$ 
and $0 < \dl \le \mu^{-1}$.
Hence, by choosing $\mu \gg 1$ sufficiently large
(depending on $\eps$, $M$, and $K$), we obtain, 
from~\eqref{defect6}, that 
\begin{align*}
\sup_{0\le\dl\le\mu^{-1}}
\sup_{t\in\R}\sup_{f\in E}
\|\P_{>N}v_\dl(t;f)\|_{L^2}<\eps
\end{align*}

\noi
for any $N \ge A\mu$.
Finally, set $\dl_1 = \mu^{-1}$
and choose an integer $N_1 \ge A\mu$.
This gives \eqref{LA5}
and concludes the proof of Proposition~\ref{PROP:EC}.
\end{proof}

\begin{remark}\label{REM:convEC}\rm
Let $v_0$, $\{v_{0,\dl}\}_{0<\dl\le1}$, $v$, and $v_\dl$
be as in Theorem~\ref{THM:1}.
(Note that $\{v_{0,\dl}\}_{0<\dl\le1}$ is not necessarily equicontinuous
in $L^2(\M)$.)
The proof of Proposition~\ref{PROP:EC}
also yields the following bound; 
given small $\eps > 0$, there exist
$N_1 = N_1(\eps,v_0) \gg 1$ and small 
$\dl_1 = \dl_1\big(\eps, v_0, \{v_{0, \dl}\}_{0 < \dl \le 1}\big)> 0$
 such that
\begin{align}
\|\P_{>N}v\|_{L^\infty_tL^2_x}
+\sup_{0<\dl\le\dl_1}
\|\P_{>N}v_\dl\|_{L^\infty_tL^2_x}<\eps
\label{convEC}
\end{align}

\noi
for any $N \ge N_1$.

To see this, set $K = \|v_0\|_{L^2}+1$.
First, choose $M \gg 1$ such that
$\|\P_{>M}v_0\|_{L^2}<\frac\eps8$.
Then,  choose $0 < \dl_a \le 1$ such that
\begin{align*}
\|v_{0,\dl}-v_0\|_{L^2}<\frac\eps8
\qquad\text{and}\qquad
\|v_{0,\dl}\|_{L^2}\le K
\end{align*}

\noi
for any $0 < \dl \le \dl_a$.
In particular, $\|\P_{>M}v_{0,\dl}\|_{L^2}<\frac\eps4$.
Choose $\mu \ge \mu_0(K)$ sufficiently large such that
\begin{align}
C\frac{M^2}{\mu^2}K^2+CK^3\mu^{-\frac32}<\frac{\eps^2}{8}.
\label{convEC2}
\end{align}

\noi
Set $N_1 = \lceil A\mu\rceil$
and $\dl_1 = \min(\dl_a,\mu^{-1})$.
Then, \eqref{convEC}
follows from 
\eqref{defect6} 
with  \eqref{convEC2}.
\end{remark}

\section{Low frequency analysis: normal form method}
\label{SEC:NF}

In this section, we present a proof of Proposition \ref{PROP:low}
based on an infinite iteration of normal form reductions.
More precisely, we 
implement a  novel perturbative argument
based on an infinite iteration of normal form reductions
for KdV.
Namely, 
we apply 
normal form reductions 
only to KdV Duhamel integral terms; 
when we encounter 
the scaled ILW  Duhamel integral operator, 
we replace it by that for KdV and view the difference
as an error term (see
$\EE_\dl^{(j)}$
in \eqref{ER1}
and \eqref{NF13})
which is then handled by 
the  convergence of the resonance function
$\Xi_\dl (\bar \xi)$
for the scaled ILW (see \eqref{Xi1})
to the resonance function
$\Xi_{\KDV} (\bar \xi)$
for KdV, 
together with
 the weakly uniform 
equicontinuity (Proposition \ref{PROP:EC})
controlling a very high frequency part;
see \eqref{PE7c}.
An infinite iteration of normal form reductions
was first introduced in \cite{GKO}
and was further developed in \cite{OW,OST, KOY, Ki2, Ki3}. 
As mentioned in Remark \ref{REM:NF1}, 
it has been applied in various contexts.
As such, we will keep our presentation brief
on 
what is now standard,  and 
aim to emphasize the novel perturbative aspect 
of the argument.

\medskip

In the periodic setting, 
in order to avoid the problem at the zeroth frequency, 
it is crucial to work with mean-zero functions, 
(which makes the periodic KdV non-resonant).
In view of the conservation of the spatial mean, 
the following Galilean transform  (with $\mu (u) = \frac 1{2\pi}\int_\T u(0, x) dx$):
\begin{align*}
\wt u(t, x) 
= u (t, x- 2\mu(u) t) - \mu (u)
\end{align*}

\noi
allows us to  reduce the problem to the mean-zero case.
Hence, we assume that, in the periodic case, 
all the functions have spatial mean zero in the following.
For this reason, we set 
\begin{align*}
\Z_* = \Z \setminus\{0\}\qquad \text{and}
\qquad 
\ft \M_* = \begin{cases}
 \R, & \text{if }\M = \R, \\
\Z_*,& \text{if } \M = \T.
\end{cases}
\end{align*}

We emphasize that our goal is {\it not} to prove unconditional uniqueness of solutions.
Hence, by a density argument, 
we may assume that  solutions to KdV and the scaled ILW
are nice, namely, Schwartz solutions in the real line case
and $C^\infty$-solutions in the periodic case, 
allowing us to justify all the formal steps
appearing  in  normal form reductions.\footnote{As in the usual
implementation of an infinite iteration of normal form reductions, 
the $L^2$-regularity
is indeed sufficient to justify all the formal steps.
We, however, do not worry about this point.}
Lastly, for simplicity of the presentation, 
we restrict our attention to 
positive times in the remaining part of this paper.

In 
Subsection \ref{SUBSEC:NF1}, 
we describe our perturbative approach
in the first step of normal form reductions.
After introducing notations, 
we then go over the general step in 
Subsection \ref{SUBSEC:NF2}.
In Subsection \ref{SUBSEC:NF3}, 
we first state
key multilinear estimates
and convergence results
 (Proposition~\ref{PROP:NN})
whose proof is postponed to 
Sections~\ref{SEC:per}
and 
\ref{SEC:Euc}.
We then 
present a proof of Proposition~\ref{PROP:low}.

\subsection{Perturbative approach: first step}
\label{SUBSEC:NF1}

Fix  $v_0 \in L^2(\M)$, 
and let  $\{v_{0, \dl}\}_{0 <  \dl \le 1}\subset L^2(\M)$
such that 
$v_{0, \dl}$ converges to $v_0$ in $L^2(\M)$ as $\dl \to 0$.
Then, 
let $v, v_\dl   \in C(\R; L^2(\M))$ be the  solutions to KdV~\eqref{KDV}
and the scaled ILW~\eqref{sILW1}
with $v|_{t = 0} = v_0$
and $v_\dl|_{t = 0} = v_{0, \dl}$, respectively.
We denote by  $\vv$ and $\vvd$
 the interaction representations  of  $v$ and $v_\dl$ defined in
\eqref{inter}.
In the remaining part of this paper, 
we suppress the normalization factors, involving $2\pi$.

 From  \eqref{KDV}
 and 
 \eqref{sILW1}
 with 
 \eqref{inter}, 
the interaction representations $\vv$ and $\vvd$ 
satisfy the following integral equations:
\begin{align}
\begin{split}
\ft\vv(t, \xi) 
&= 
\ft v _0(\xi)
+ \int_0^t 
\int_{ \G(\xi)} 
e^{it'\Xi_\KDV(\bar \xi)} i \xi  \, \ft\vv(t', \xi_1) \ft\vv(t', \xi_2)  d\xi_1dt', \\
\ftvvd(t, \xi) 
&= 
\ft v_{0, \dl}(\xi)
+ \int_0^t 
\int_{ \G(\xi)} 
e^{it'\Xi_\dl(\bar \xi)} i \xi  \, \ftvvd(t', \xi_1) \ftvvd(t', \xi_2) d\xi_1dt'
\end{split}
\label{NF1}
\end{align}

\noi
for each  $\xi\in\ft \M_*$ , 
where 
 $\G(\xi)$ is defined as
\begin{align*}
\G(\xi)
&
=
\big\{ 
(\xi_1, \xi_2) \in (\ft\M_*)^2:
\xi = \xi_1 + \xi_2,
\ \xi \xi_1 \xi_2 \ne 0 
\big\}, 
\end{align*}
with the last condition being relevant in the periodic case  $\M = \T$, 
and we used the following short-hand notation:
\[\int_{ \G(\xi)} f(\bar \xi) d\xi_1 
:=
\int_{ (\xi_1, \xi_2) \in \G(\xi)} f(\bar \xi) d\xi_1 .\]

\noi
Here, 
$\Xi_\KDV(\bar \xi) $ and $\Xi_\dl(\bar \xi) $
denote the resonance functions for KdV and 
the scaled ILW, respectively, given by
\begin{align}
\begin{split}
\Xi_{\KDV} (\bar \xi) 
& = \Xi_{\KDV} (\xi,\xi_1,\xi_2)
 = 
- \xi^3 + \xi_1^3 + \xi_2^3 \\
& =- 3 \xi\xi_1\xi_2, 
\\
\Xi_\dl (\bar \xi)
& = 
\Xi_\dl (\xi, \xi_1,\xi_2)
=
- \xi \L_\dl(\xi) + \xi_1 \L_\dl (\xi_1) + \xi_2 \L_\dl (\xi_2)
\\
&
=
- \xi^3 ( 1 - h(\dl, \xi)) + \xi_1^3 (1 - h(\dl,\xi_1)) + \xi_2^3 (1 - h(\dl, \xi_2))
\\
&
= \Xi_\KDV(\xi,\xi_1,\xi_2) +
\Big( \xi^3 h(\dl,\xi) - \xi_1^3 h(\dl, \xi_1) - \xi_2^3 h(\dl, \xi_2) \Big), 
\end{split}
\label{Xi1}
\end{align}

\noi
where $\L_\dl(\xi)$ and $h(\dl, \xi)$ are as in \eqref{L1} and \eqref{hdef}, while
the last equality for $\Xi_{\KDV}(\bar \xi)$ holds under the condition $\xi=\xi_1 + \xi_2$.  
Recall from \eqref{limh} that 
\begin{align}
\lim_{\dl \to 0}\Xi_\dl (\xi, \xi_1,\xi_2)
= \Xi_\KDV (\xi, \xi_1,\xi_2)
\label{NF2}
\end{align}

\noi
for each   $\xi, \xi_1, \xi_2 \in \ft \M_*$.
Let  $\varphi_\dl$ denote the difference 
of the oscillatory factors in~\eqref{NF1}
by setting 
\begin{align}
\label{phi1}
\varphi_\dl(t, \xi, \xi_1,\xi_2)  =  e^{it \Xi_\dl(\xi,\xi_1,\xi_2)}
- e^{it \Xi_\KDV(\xi,\xi_1,\xi_2)}.
\end{align}

\noi
Then, from \eqref{phi1} and  \eqref{NF2}
we have the following pointwise convergence:
\begin{align}
\lim_{\dl\to0} \varphi_\dl(t, \xi, \xi_1,\xi_2) = 0 
\label{phi2}
\end{align}

\noi
for each  $t\in\R$ and $\xi, \xi_1, \xi_2 \in \ft \M_*$.

Define the  bilinear operator $\NN^{(1)}$  by setting
\begin{align}\label{NN1}
 \NN^{(1)} ( f_1, f_2) (\xi) =  
 \int_{ \G(\xi)} e^{it \Xi_{\KDV}(\bar \xi) } i \xi \ft f_1(\xi_1) \ft f_2(\xi_2) d\xi_1. 
\end{align}

\noi
When two arguments coincide, 
we simply set  $\NN^{(1)}(f) = \NN^{(1)}(f,f)$. 
Moreover, due to the presence of 
the time-dependent phase factor $e^{i t\Xi_\KDV(\bar \xi)}$, 
the bilinear operator  $\NN^{(1)}$
is non-autonomous (depending on $t$). 
For simplicity of notations, however, we suppress such $t$-dependence
when there is no confusion.
When we need to emphasize its dependence on $t$ (as in Proposition \ref{PROP:NN}), 
we write 
\begin{align*}
 \NN^{(1)} (t)(f) (\xi).
\end{align*}

\noi
For a time-dependent function $\uu$, 
we simply set 
\begin{align*}
 \NN^{(1)} (\uu) (t, \xi)
 =  \NN^{(1)}(t) (\uu(t)) (\xi).
\end{align*}

\noi
We apply this convention for all the multilinear operators
appearing in the remaining part of this paper.
Lastly, 
we emphasize 
 that the time-dependence
of all the multilinear operators  
appears only in the phase factors.

Using \eqref{NN1}, we can write the equation for $\ftvvd$ in \eqref{NF1} as 
\begin{align}
\ftvvd(t,\xi)
=  \ft v_{0, \dl}(\xi)
+ 
\int_0^t 
\Big(
\NN^{(1)} ( \vvd) (t', \xi)
+
\EE_\dl^{(1)} ( \vvd) (t', \xi)
\Big)
dt', 
\label{NF3}
\end{align}

\noi
where the error term $\EE_\dl^{(1)} ( \vvd)$
is defined by 
\begin{align}
\label{ER1}
\begin{split}
\EE_\dl^{(1)} (f) ( \xi)
& = \int_{\G(\xi) }
 \varphi_\dl(t, \bar \xi) 
  i \xi \,  \ft f(\xi_1) \ft f(\xi_2) d\xi_1\\
&   = \int_{\G(\xi) }
 \Big( e^{it \Xi_\dl (\bar \xi)}  - e^{it \Xi_\KDV(\bar \xi)}\Big)
  i \xi \,  \ft f(\xi_1) \ft f(\xi_2) d\xi_1, 
\end{split}
\end{align}

\noi
with $ \varphi_\dl(t, \bar \xi) $ as in \eqref{phi1}.
With these notations, we can rewrite KdV \eqref{KDV}
and the scaled ILW~\eqref{sILW1}
as
\begin{align}
\label{dt}
\begin{split}
\dt \ft\vv (\xi)
&
= \NN^{(1)} (\vv) (\xi)
, 
\\
\dt \ftvvd (\xi)
&
= 
 \NN^{(1)} (\vvd) (\xi)
+
 \EE^{(1)}_\dl (\vvd) (\xi)
. 
\end{split}
\end{align}

Let us now turn to the first step of  normal form reductions. 
We point out  that 
for $\vvd$, 
we apply a normal form reduction (namely, integration by parts in time)
only
to\footnote{More precisely, 
to the non-resonant part of $\NN^{(1)}(\vvd)$.} $\NN^{(1)}(\vvd)$ in~\eqref{NF3} in  this first step.
 The error part 
$\EE_\dl^{(1)} ( \vvd)$
will be handled  in a direct manner
thanks to the low frequency cutoff $\P_{\le N}$
in \eqref{low00}
with the convergence~\eqref{phi2}
and  the weakly uniform 
equicontinuity (Proposition \ref{PROP:EC});
see Proposition \ref{PROP:NN}\,(iii).

As in \cite{GKO, OW}, 
we  divide 
 $\NN ^{(1)}$
 into the nearly resonant part 
  $\NN ^{(1)}_1$
 and the (highly) non-resonant part  $\NN ^{(1)}_2$
 as the restrictions
 of  $\NN ^{(1)}$
 onto $A_1(\xi)$ 
 and its complement
$A_1(\xi)^c$, respectively:
\begin{align}
\label{A1}
\begin{split}
A_1(\xi)
& : = 
\big\{ (\xi_1,\xi_2) \in \G(\xi)  : \  |\Xi_{\KDV}(\xi,\xi_1,\xi_2)| \le (3 K)^4 \big\}
, 
\\
A_1(\xi)^c
& : = 
\big\{ (\xi_1,\xi_2) \in \G(\xi)  : \  |\Xi_{\KDV}(\xi,\xi_1,\xi_2)| > (3 K)^4 \big\}, 
\end{split}
\end{align}

\noi
where $K = K(\|v_0\|_{L^2}) \ge 1$ is a constant to be chosen later.
Obviously, we have 
\begin{align}
\NN^{(1)} = \NN^{(1)}_1  + \NN^{(1)}_2.
\label{S1}
\end{align}

\noi
By 
 applying a normal form reduction to the non-resonant part $\NN_2^{(1)}$
in the form of ``differentiation by parts'' \cite{BIT}, 
we have\footnote{Hereafter, we often suppress time dependence of $\uu$ and $\vv$, etc.
for notational simplicity.}
\begin{align}
\begin{split}
\NN^{(1)}_2(\uu)(t, \xi)
&
= \int_{A_1(\xi)^c} \dt \Big( \frac{e^{it \Xi_{\KDV} (\bar \xi)}}{ i \Xi_\KDV(\bar \xi)  } \Big) 
i \xi
\ft\uu(\xi_1) \ft\uu(\xi_2)
d\xi_1
\\
&
= 
\dt \bigg( \int_{A_1(\xi)^c}  \frac{e^{it \Xi_{\KDV} (\bar \xi)}}{ i \Xi_\KDV(\bar \xi)  }  
i \xi \ft\uu(\xi_1) \ft\uu(\xi_2) 
d\xi_1
\bigg)\\
& \quad - 
\int_{A_1(\xi)^c} \frac{e^{it \Xi_{\KDV} (\bar \xi)}}{ i \Xi_\KDV(\bar \xi)  } 
i \xi \dt \big( \ft\uu(\xi_1) \ft\uu(\xi_2) \big)
d\xi_1\\
&
= : \dt \NN_0^{(1)} (\uu)(t, \xi ) + \wt\NN^{(2)}(\uu)(t,\xi)
\end{split}
\label{NF3a}
\end{align}

\noi
for $\uu = \vv$ or $\vvd$.
By applying the product rule and \eqref{dt}, 
we have 
\begin{align}
\begin{split}
\wt\NN^{(2)}(\vv) (t, \xi)
&=- \sum_{\s \in S_2} \int_{A_1(\xi)^c} 
\frac{e^{it \Xi_{\KDV} (\bar \xi)}}{ i \Xi_\KDV(\bar \xi)  } 
i \xi
 \NN^{(1)} (\vv)(\xi_{\s(1)})  \ft\vv(\xi_{\s(2)}) 
d\xi_1
\\
&=: \NN^{(2)} (\vv) (t, \xi)
\end{split}
\label{NF3b}
\end{align}

\noi
for $\uu = \vv$, 
where $S_2$ is the symmetric group on $\{1, 2\}$, 
and 
\begin{align}
\begin{split}
\wt\NN^{(2)}(\vvd) (t, \xi)
&
= 
- \sum_{\s \in S_2}\int_{A_1(\xi)^c} 
\frac{e^{it \Xi_{\KDV} (\bar  \xi)}}{ i \Xi_\KDV(\bar  \xi)  } 
i \xi
 \NN^{(1)} (\vvd)(\xi_{\s(1)})  \ftvvd(\xi_{\s(2)}) 
d\xi_1
\\
&
\quad 
-  \sum_{\s \in S_2}
 \int_{A_1(\xi)^c} 
\frac{e^{it \Xi_{\KDV} (\bar  \xi)}}{ i \Xi_\KDV(\bar  \xi)  }
i \xi
 \EE_\dl^{(1)} (\vvd)(\xi_{\s(1)})  \ftvvd(\xi_{\s(2)}) 
d\xi_1
\\
&
=: \NN^{(2)}(\vvd)(t,\xi)
+ \EE^{(2)}_\dl (\vvd) (t,\xi) 
\end{split}
\label{NF3c}
\end{align}

\noi
for $\uu = \vvd$. 
Hence, 
from \eqref{NF1}, \eqref{S1}, \eqref{NF3}, \eqref{NF3a}, \eqref{NF3b}, and \eqref{NF3c}, 
we obtain
\begin{align}
\begin{split}
\ft\vv(t, \xi) 
& =  \ft v_0(\xi)
+   \NN^{(1)}_0(\vv)(t', \xi) \bigg|_0^t
+ \int_0^t 
\NN^{(1)}_1  (\vv)(t', \xi) 
+ 
\NN^{(2)}  (\vv)(t', \xi) dt'
, 
\\
\ftvvd(t, \xi) 
& =  \ft v_{0, \dl}(\xi)
+   \NN^{(1)}_0(\vvd)(t',  \xi) \bigg|_0^t
+ 
\int_0^t 
\NN^{(1)}_1  (\vvd)(t',  \xi) 
+ 
\NN^{(2)}  (\vvd)(t',  \xi) 
dt'
\\
&
\quad 
+
\sum_{j=1}^2
\int_0^t
 \EE^{(j)}_\dl (\vvd) (t',  \xi)
dt'
\end{split}
\label{NF4}
\end{align}

\noi
 after the first step of normal form reductions.
As for the term  $\NN ^{(2)}(\uu)$
 in \eqref{NF3b} and~\eqref{NF3c}, 
we need to split it into the nearly resonant and non-resonant parts
and apply a normal form reduction on the non-resonant part.
We then iterate this process indefinitely.
In the next subsection, 
 we first 
recall the notion of (binary) ordered trees and relevant  definitions
from \cite{GKO, OW2}
and then present a general step of normal form reductions.

\subsection{General step: ordered trees}
\label{SUBSEC:NF2}

We recall the notion of  ordered trees 
from \cite{GKO, OW2}
(but adapted to  binary trees).

\begin{definition}\label{DEF:tree2}\rm
(i) Let $\TT$ be  a partially ordered set  with partial order $\le$.
Given two distinct elements $a, b \in \TT$, 
 we say that $b\le a$ is a child of $a$ if $b \le c \le a$ with $b \ne a$ implies that 
 either $c = a$ or $c = b$.
If the latter condition holds, we also say that $a$ is the parent of $b$
and write 
 $\pb(b) = a$.

\smallskip

\noi
(ii) A tree $\TT$ is a finite partially ordered set satisfying the following properties:

\smallskip
\begin{itemize}
\item Let $a_1, a_2, a_3, a_4 \in \TT$.
If $a_4 \leq a_2 \leq a_1$ and  
$a_4 \leq a_3 \leq a_1$, then we have $a_2\leq a_3$ or $a_3 \leq a_2$,

\smallskip

\item
A node $a\in \TT$ is called terminal, if it has no child.
A non-terminal node $a\in \TT$ is a node 
with  exactly two children denoted by $a_1$ and $a_2$.
In the latter case, 
we use
$\sbb(a_1) = a_2$ and 
$\sbb(a_2) = a_1$, 
where
$\sbb(a) $
 denotes the (unique) sibling node of a child $a$,

\smallskip
\item There exists a maximal element $r \in \TT$ (called the root node) such that $a \leq r$ for all $a \in \TT$.
We assume that the root node is non-terminal,

\smallskip

\item $\TT$ consists of the disjoint union of $\TT^0$ and $\TT^\infty$,
where $\TT^0$ and $\TT^\infty$
denote  the collections of non-terminal nodes and terminal nodes, respectively.
We use the following convention;
given $a \in \TT^0$, we use $a_1$ and $a_2$ to denote
its two children, 
where $a_1$ (and $a_2$) denotes
the left child (and the right child, respectively)
 in the planar graphical representation of $\TT$.

\end{itemize}

\smallskip

\noi
The number of nodes $|\TT|$ in tree $\TT$ is $2j+1$ for some $j\in\N$, where $|\TT^0| = j$ and $|\TT^\infty|=j+1$. 

\noi 
Given $j\in \N$, we use $T(j)$ to denote  the collection of trees of $j$ generations:
\begin{align*}
T(j)
 = 
\big\{
\TT: \  \TT \text{ is a tree with } |\TT| = 2j+1
\big\}.
\end{align*}

\smallskip

\noi(iii) (ordered tree) We say that a sequence $\{ \TT_j \}_{j=1}^J$ is a chronicle of $J$ generations if 

\smallskip

\begin{itemize}
\item $\TT_j \in T(j)$ for each $j=1,\ldots,J$,

\smallskip

\item $\TT_{j+1}$ is obtained by changing 
one of  the terminal nodes in $\TT_j$, 
denoted by $p^{(j)}$, 
into a non-terminal node (with two children), $j=1, \ldots, J-1$. 

\smallskip

\end{itemize}
Given a chronicle $\{\TT_j\}_{j=1}^J$ of $J$ generations, we refer to $\TT_J$ as an \textit{ordered tree} of the $J$th generation. 
We denote the collection of ordered trees of the $J$th generation by $\Tfr(J)$.
Note that the cardinality of $\mathfrak{T}(J)$ is given by 
\begin{align}
\label{TJ0}
| \Tfr(J) | = J!. 
\end{align}

\end{definition}

Roughly speaking, 
 ordered trees are the trees
 that remember how they ``grew''.
As seen in  \cite{GKO, OW}, 
this property is suitable for  encoding 
successive applications
of the product rule for differentiation.
In the following, we simply refer to an ordered tree $\TT_J$ of the $J$th generation
but 
it is understood that there is an underlying chronicle $\{ \TT_j\}_{j = 1}^J$.

\begin{definition} \label{DEF:tree3}\rm
Given an ordered tree $\TT_J \in \Tfr(J)$ 
with a chronicle $\{\TT_j\}_{j = 1}^J$, 
we define  ``projections'' $\Pi_j$ and $\pi_j$, $j = 1, \dots, J$,  from $\TT_J$
to subtrees in $\TT_J$ 
by setting
\begin{itemize}

\item[$\bullet$]
$\Pi_j(\TT_J) = \TT_j$, 

\smallskip

\item[$\bullet$]
 $\pi_1(\TT_J) = \TT_1$, 

\smallskip

\item[$\bullet$] 
$\pi_j(\TT_J)$ to be the tree 
of one generation, 
formed by the two terminal nodes in $\TT_j \setminus \TT_{j-1}$
and their parent,  $j = 2, \dots, J$.   
Intuitively speaking, 
$\pi_j(\TT_J)$ is the tree added in transforming $\TT_{j-1}$ into $\TT_j$.

\end{itemize}

\noi
We use $r^{(j)}$ to denote the root node of $\pi_j(\TT_J)$
and refer to it as the {\it $j$th root node}.
By definition, we have
\begin{align*}
r^{(j)} = p^{(j-1)}, 
\qquad j = 2, \dots, J, 
\end{align*}

\noi
where 
$p^{(j-1)}$ is as in Definition
\ref{DEF:tree2}\,(iii).
Note that
$r^{(j)}$  is not necessarily a node in 
$\pi_{j-1}(\TT_J)$.
Given a non-terminal node $a \in \TT_j^0$, 
we have $a = r^{(k)}$ for some $k$.
We then define  the function 
$\gf: \TT_j^0 \to \{1, \dots, j\}$
by setting $\gf(a) = k$ with 
$a = r^{(k)}$.
Namely,  
$\gf(a)$ denotes the generation
where $a$ appears as a parent.

\end{definition}

\medskip

Given a tree $\TT$, we associate each terminal node $a\in\TT^\infty$ with 
the  Fourier transform of the interaction representation $\vv$ (or $\vvd$) and integrate over all possible frequency choices. 
For this purpose, we introduce the index function $\xii$ assigning frequencies to 
all the nodes in $\TT$ in a consistent manner.

\begin{definition}[index function]\label{DEF:index}
\rm 

Let  $J\in\N$.  Given  a tree $\TT\in T(J)$, 
we define an index function $\xii: \TT \to \ft\M_*$ such that
$\xi_a = \xi_{a_1} + \xi_{a_2}$ for $a\in \TT^0$, where $a_1$ and $a_2$ denote the children of $a$.
%
%
%
%
Here, 
 we identified $\xii: \TT \to\ft\M_*$ with $\{\xi_a\}_{a\in\TT}\in (\ft\M_*)^\TT$. We use $\Nf(\TT) \subset (\ft\M_*)^\TT$ to denote the collection of such index functions $\xii$. 

\end{definition}


\medskip

Given an ordered tree 
$\TT_J$ of the $J$th generation with a chronicle $\{ \TT_j\}_{j = 1}^J$ 
and associated index functions $\xii \in \mathfrak{N}(\TT_J)$,
 we use superscripts to 
  denote  ``generations'' of frequencies.
 The first generation $\TT_1$ has frequencies:
\begin{align*}
\big(\xi^{(1)} , \xi_1^{(1)}, \xi_2^{(1)}\big) : = (\xi_r ,  \xi_{r_1}, \xi_{r_2}), 
\end{align*}

\noi
where $r$ denotes the root node
with its  children $r_1$ and $r_2$.
The ordered tree of the second generation $\TT_2$ is obtained from $\TT_1$ by changing
one of its terminal nodes $a = r_k \in \TT^\infty_1$ for some $k \in \{1, 2\}$
 into a non-terminal node (with two children $a_1$ and $a_2$). 
 Then, we define
the second generation of frequencies by setting
\[\big(\xi^{(2)}, \xi^{(2)}_1, \xi^{(2)}_2\big) :=(\xi_a, \xi_{a_1}, \xi_{a_2}).\]

\noi
Note that  we have $\xi^{(2)} = \xi^{(1)}_k = \xi_{r_k}$.
This corresponds to introducing a new set of frequencies
after the first step of normal form reductions.

After  $j - 1$ steps, the ordered tree $\TT_j$ 
of the $j$th generation is obtained from $\TT_{j-1}$ by
changing one of its terminal nodes $a  \in \TT^\infty_{j-1}$
into a non-terminal node (with two children $a_1$ and $a_2$).
Then, we define
the $j$th generation of frequencies by setting
\begin{align}
\big(\xi^{(j)}, \xi^{(j)}_1, \xi^{(j)}_2\big) :=(\xi_a, \xi_{a_1}, \xi_{a_2}).
\label{fre1}
\end{align}

\noi
Note that these frequencies
satisfy  Definition \ref{DEF:index}:
$\xi^{(j)} = \xi_{1}^{(j)} + \xi_{2}^{(j)}$.

\medskip

Lastly, we introduce the resonance function $\mu_j$ in the $j$th generation:
\begin{align}
\begin{split}
\mu_j = \mu_j\big(\xi^{(j)}, \xi^{(j)}_1, \xi^{(j)}_2\big) 
: \!& = \Xi_\KDV\big(\xi^{(j)}, \xi^{(j)}_1, \xi^{(j)}_2\big)\\
& = -3\xi^{(j)} \xi^{(j)}_1 \xi^{(j)}_2, 
\end{split}
\label{mu1}
\end{align}

\noi
where $\Xi_\KDV$ is as in \eqref{Xi1}. 
Then,  by setting
\begin{align}
\wt\mu_0 = 0\qquad \text{and} \qquad \wt\mu_j = \sum_{k=1}^j \mu_k, 
\label{mu2}
\end{align}

\noi
we  define the nearly resonant set $A_j$ for the $j$th generation, $j \ge 2$:
\begin{align}
\label{A2}
A_j 
= \big\{ 
\xii \in \Nf(\TT_j): | \wt\mu_j| \le (j+2)^4  | \wt\mu_{j-1}|
\big\}.
\end{align}

\noi
Its complement $A^c_j$
is defined by 
\begin{align}
A^c_j
=
\big\{ 
\xii \in \Nf(\TT_j): | \wt\mu_j| > (j+2)^4 | \wt\mu_{j-1}|
\big\}, 
\label{A3}
\end{align}

\noi
corresponding to the (highly) non-resonant set  for the $j$th generation.
Compare \eqref{A2} and~\eqref{A3} with \eqref{A1}.
From \eqref{A1}, \eqref{mu2}, \eqref{A2}, and \eqref{A3}
with \eqref{mu1}, we have 
\begin{align}
\label{A4}
\begin{split}
&  \xii \in \bigcap_{k=1}^{j} A_k^c  \\
&   \LRA  \ \ 
 |\wt\mu_k|
 \sim |\mu_k|
 \sim |\xi^{(k)}\xi^{(k)}_1 \xi^{(k)}_2|
\ges \big((k+2)! K\big)^4
 , \ k=1, \ldots, j, 
\end{split}
\end{align}

\noi
and
\begin{align}
\begin{split}
& \xii \in A_j \cap \bigg(\bigcap_{k=1}^{j-1} A_k^c \bigg)\\
&  \LRA  \ \ 
 |\wt\mu_k| \sim |\mu_k|  
  \sim |\xi^{(k)}\xi^{(k)}_1 \xi^{(k)}_2| \ges \big((k+2)! K\big)^4, \ 
 k=1, \ldots, j-1, 
\\
& \hphantom{\LRA} \ \ 
\ \
\text{and}
\ \ 
|\mu_{j-1}| \ges j^{-4} |\mu_j|.
\end{split}
\label{A4a}
\end{align}

\begin{remark} \label{REM:terminal}
\rm

Note that $\xii = \{\xi_a\}_{a\in\TT}$ is completely determined
once we specify the values $\xi_a$ for $a \in \TT^\infty$.
Moreover, 
given an ordered tree $\TT_j$ of the $j$th generation, we have 
\begin{align*}
\xi_ r = \sum_{a \in \TT_j^\infty}  \xi_{a}.
\end{align*}

\noi
Recalling 
$|\TT^\infty_j| = j + 1$, 
this implies that the symbol 
\[\intt_{\substack{\xii\in \mathfrak{N}(\TT_{j})\\\xi_r=\xi}}\]

\noi
means that we are performing $j$ integrations
in  $\xi_a$, $a \in \TT_j^\infty \setminus \{b\}$
(satisfying $\xi_r = \xi$)
by first fixing some $b\in \TT_j^\infty$.

\end{remark}

Let us now go back to the normal form procedure.
As mentioned above, we split $\NN^{(2)}$ in~\eqref{NF4}
into the nearly
resonant part
 $\NN^{(2)}_1$ (as the restriction
 onto $A_2$ defined in \eqref{A2}) 
and  the non-resonant part
 $\NN^{(2)}_2$
  (as the restriction
 onto $A_2^c$ defined in \eqref{A3}) 
  and apply the second step of normal form reductions 
 to 
$\NN^{(2)}_2$.
As for $\NN^{(2)}_2(\vv)$, 
the  procedure is precisely the same as those in \cite{GKO, OW2}.
As seen in \eqref{NF4}, 
however, 
two terms appear when we apply a normal form reduction to 
 $\NN^{(2)}_2(\vvd)$
 and replace $\partial_t \ftvvd$, using \eqref{dt}, 
which is a novel point of our perturbative argument.

After $j-1$ steps, we have the following non-resonant part
\begin{align*}
\NN^{(j)}_2 (\uu) (\xi) 
 =
(-1)^{j-1}
\sum_{\TT_j \in \Tfr(j)} \intt_{\substack{\xii \in \Nf(\TT_j) \\ \xi_r = \xi}} \ind_{\bigcap_{k=1}^j A_k^c} \frac{ i e^{it\wt\mu_j} \prod_{k=1}^j  \xi^{(k)} }{\prod_{k=1}^{j-1}  \wt\mu_k} 
\prod_{a\in\TT^\infty_j} \ft\uu(\xi_a).
\end{align*}

\noi
Applying the $j$th normal form reduction
with the product rule, we obtain
\begin{align}
\begin{split}
&  \NN^{(j)}_2 (\uu) (\xi) \\
&
\quad = 
\dt 
\bigg(
(-1)^{j-1}
\sum_{\TT_j \in \Tfr(j)} \intt_{\substack{\xii \in \Nf(\TT_j) \\ \xi_r = \xi}} \ind_{\bigcap_{k=1}^j A_k^c} 
\frac{e^{it\wt\mu_j} \prod_{k=1}^j \xi^{(k)} }{\prod_{k=1}^{j} \wt\mu_k} 
\prod_{a\in\TT^\infty_j} \ft\uu(\xi_a) \bigg)
\\
&
\quad \quad 
+
(-1)^j
\sum_{\TT_j \in \Tfr(j)} \sum_{b\in \TT^\infty_j} \intt_{\substack{\xii \in \Nf(\TT_j) \\ \xi_r = \xi}} \ind_{\bigcap_{k=1}^j A_k^c} \frac{ e^{it\wt\mu_j} \prod_{k=1}^j  \xi^{(k)} }{\prod_{k=1}^{j}  \wt\mu_k} \\
& \hphantom{XXXXXXXXXXXXXX}
\times \big( \dt \ft\uu(\xi_b) \big)
\prod_{a\in\TT^\infty_j \setminus\{b\}} \ft\uu(\xi_a).
\end{split}
\label{NF5}
\end{align}

\noi
We first consider the case $\uu = \vv$.
Then, from \eqref{NF5} and \eqref{dt} with \eqref{NN1} and \eqref{mu2}, we have 
\begin{align}
& \NN^{(j)}_2 (\vv) (\xi) \notag \\
&\quad 
= 
\dt 
\bigg(
(-1)^{j-1}
\sum_{\TT_j \in \Tfr(j)} \intt_{\substack{\xii \in \Nf(\TT_j) \\ \xi_r = \xi}} \ind_{\bigcap_{k=1}^j A_k^c} \frac{ e^{it\wt\mu_j} \prod_{k=1}^j  \xi^{(k)} }{\prod_{k=1}^{j}  \wt\mu_k} 
\prod_{a\in\TT^\infty_j} \ft\vv(\xi_a) 
\bigg)
\notag
\\
&\quad \quad 
+
(-1)^j
\sum_{\TT_{j+1}\in \Tfr(j+1)}  \intt_{\substack{\xii \in \Nf(\TT_{j+1}) \\ \xi_r = \xi}} \ind_{\bigcap_{k=1}^j A_k^c} \frac{ i e^{it\wt\mu_{j+1}} \prod_{k=1}^{j+1}  \xi^{(k)} }{\prod_{k=1}^{j}  \wt\mu_k} 
\prod_{a\in\TT^\infty_{j+1}} \ft\vv(\xi_a) 
\notag
\\
&\quad 
=: \dt \NN_0^{(j)} (\vv) ( \xi ) + \NN^{(j+1)} (\vv) (\xi).
\label{NF6}
\end{align}

\noi
Next, we consider the case $\uu = \vvd$, 
where we use the second equation in \eqref{dt}.
This leads to 
\begin{align}
\label{NF7}
\NN^{(j)}_2 (\vvd) (\xi) 
&
= \dt \NN_0^{(j)}(\vvd)(\xi) + \NN^{(j+1)}(\vvd) (\xi) + \EE^{(j+1)}_\dl (\vvd)(\xi).
\end{align}

\noi
Here,  the last term 
denotes the contribution
from $\EE^{(1)}_\dl(\vvd)$
(see \eqref{ER1})
in replacing 
 $\dt\ftvvd$, and  is given by
\begin{align}
\begin{split}
\EE^{(j+1)}_\dl (\vvd)(\xi)
&
=
(-1)^j
\sum_{\TT_{j+1}\in \Tfr(j+1)}  \intt_{\substack{\xii \in \Nf(\TT_{j+1}) \\ \xi_r = \xi}} \ind_{\bigcap_{k=1}^j A_k^c} \\
& 
\quad \times \varphi_\dl \big(t, \xi^{(j+1)}, \xi^{(j+1)}_1,  \xi^{(j+1)}_2\big) 
\frac{i e^{it\wt\mu_{j}} \prod_{k=1}^{j+1} \xi^{(k)} }{\prod_{k=1}^{j}  \wt\mu_k} 
\prod_{a\in\TT^\infty_{j+1}} \ftvvd(\xi_a), 
\end{split}
\label{NF8}
\end{align}

\noi
where $\varphi_\dl$ is as in \eqref{phi1} and $(\xi^{(j+1)}, \xi^{(j+1)}_1, \xi^{(j+1)}_2)$ denotes the frequencies of the last generation. 
We note that the only difference between $\NN^{(j+1)}$ and $\EE^{(j+1)}_\dl$ in \eqref{NF6} and \eqref{NF8}, respectively, appears in the multiplier, where,  in the latter, 
\begin{align*}
 \varphi_\dl \big(t, \xi^{(j+1)}, \xi^{(j+1)}_1,  \xi^{(j+1)}_2\big)
 e^{it \wt\mu_{j}}
\quad 
\text{replaces} 
\quad 
e^{it \wt\mu_{j+1}} 
= e^{it \wt\mu_{j}} e^{it \mu_{j+1}}
\ \text{in }  \NN^{(j+1)}.
\end{align*}

In order to proceed to the $(j+1)$th normal form reduction, 
 we now split $\NN^{(j+1)}$, 
appearing in~\eqref{NF6} and \eqref{NF7},  into
the nearly resonant part $\NN^{(j+1)}_1$ and the (highly) non-resonant part~$\NN^{(j+1)}_2$:
\begin{align*}
\NN^{(j+1)} = \NN^{(j+1)}_1 + \NN^{(j+1)}_2, 
\end{align*}

\noi
where the first term is 
the restriction of $\NN^{(j+1)}$ onto
 $A_{j+1}$  in \eqref{A2}, 
 while  $\NN^{(j+1)}_{2}$ is the restriction of $\NN^{(j+1)}$
 onto $A_{j+1}^c$ in \eqref{A3}. 
 We then apply a normal form reduction to $\NN^{(j+1)}_2$
 and iterate this process indefinitely.

\subsection{Proof of Proposition \ref{PROP:low}}
\label{SUBSEC:NF3}

In this subsection, we state crucial multilinear estimates
and convergence results
on the multilinear operators
appearing in 
our perturbative normal form procedure (Proposition \ref{PROP:NN})
whose proof is presented in 
Sections \ref{SEC:per}
and \ref{SEC:Euc}.
We then present a proof of Proposition \ref{PROP:low}, 
which is the second fundamental  ingredient
in proving Theorem~\ref{THM:1}.

After $J-1$ normal form reductions,
 the discussion in the previous two subsections 
 leads to  the following formulations:
\begin{align}
\label{NF10}
\begin{split}
\ft\vv(t, \xi) & =  \ft v_0(\xi)
+ \sum_{j=1}^{J-1} \NN_0^{(j)} ( \vv )(t', \xi) \bigg|_{0}^{t} \\
& \quad +
\int_0^t 
\sum_{j=1}^{J}
\NN_1^{(j)} (\vv) (t', \xi)
dt' 
+ 
\int_0^t 
\NN^{(J)}_2 ( \vv) (t', \xi) 
dt'
\end{split}
\end{align}

\noi
for the interaction representation $\vv$ of the solution to KdV \eqref{KDV}, 
and 
\begin{align}
\begin{split}
\ftvvd(t, \xi) 
& =  \ft v_{0, \dl}(\xi)
+ 
\sum_{j=1}^{J-1} \NN_0^{(j)} ( \vvd )(t', \xi) \bigg|_{0}^{t} \\
& \quad +
\int_0^t 
\sum_{j=1}^{J}
\Big(
\NN_1^{(j)} (\vvd) (t', \xi)
+
\EE_\dl^{(j)} (\vvd) (t', \xi)
\Big)
dt' 
\\
&
\quad 
+ 
\int_0^t 
\NN^{(J)}_2 ( \vvd) (t', \xi) 
dt' 
\end{split}
\label{NF11}
\end{align}

\noi
for the interaction representation $\vvd$ of the solution to the scaled ILW \eqref{sILW1},
where 
 $\NN_0^{(j)}$, $\NN_1^{(j)}$, $\NN_2^{(j)}$, and $\EE^{(j)}_\dl$ are
the $(j+1)$-linear operators given by 
\begin{align}
\label{NF12}
\begin{split}
\NN^{(j)}_0 (\uu) (t, \xi) 
&
=
(-1)^{j-1} \sum_{\TT_j \in \Tfr(j)}
\intt_{\substack{\xii \in \Nf(\TT_j) \\ \xi_{r} = \xi} } \ind_{\bigcap_{k=1}^j A_k^c} \frac{ e^{it \wt\mu_j}  \prod_{k=1}^j  \xi^{(k)}  }{\prod_{k=1}^{j}  \wt\mu_k} 
\prod_{a \in \TT^\infty_j} \ft\uu(t, \xi_a)
,
\\
\NN_1^{(j)} (\uu) (t, \xi)
&
=
(-1)^{j-1} \sum_{\TT_j \in \Tfr(j)}
\intt_{\substack{\xii \in \Nf(\TT_j) \\ \xi_{r} = \xi} } \ind_{A_j \cap (\bigcap_{k=1}^{j-1} A_k^c)} \\
& 
\hphantom{XXXXXXXXX}
\times \frac{ i e^{it \wt\mu_j}  \prod_{k=1}^j  \xi^{(k)}  }{\prod_{k=1}^{j-1}  \wt\mu_k} 
\prod_{a \in \TT^\infty_j} \ft\uu(t, \xi_a)
, 
\\
\NN_2^{(j)} (\uu) (t, \xi)
&
 =
(-1)^{j-1} \sum_{\TT_j \in \Tfr(j)}
\intt_{\substack{\xii \in \Nf(\TT_j) \\ \xi_{r} = \xi} } \ind_{\bigcap_{k=1}^{j} A_k^c} \frac{i  e^{it \wt\mu_j}  \prod_{k=1}^j   \xi^{(k)}  }{\prod_{k=1}^{j-1}  \wt\mu_k} 
\prod_{a \in \TT^\infty_j} \ft\uu(t, \xi_a),
\end{split}
\end{align}

\noi
and
\begin{align}
\begin{split}
\EE_\dl^{(j)} (\uu) (t, \xi)
& =(-1)^{j-1} \sum_{\TT_j \in \Tfr(j)}
\intt_{\substack{\xii \in \Nf(\TT_j) \\ \xi_{r} = \xi} } \ind_{\bigcap_{k=1}^{j-1} A_k^c} 
\varphi_\dl\big(t, \xi^{(j)}, \xi^{(j)}_1, \xi^{(j)}_2\big) \\
& \hphantom{XXXXXXXXX}\times
\frac{i  e^{it \wt\mu_{j-1}}  \prod_{k=1}^j  \xi^{(k)}  }{\prod_{k=1}^{j-1}  \wt\mu_k} 
\prod_{a \in \TT^\infty_j} \ft\uu(t, \xi_a).
\end{split}
\label{NF13}
\end{align}

\noi
Here, 
 $\varphi_\dl$ is as in \eqref{phi1}. 

We now state key multilinear estimates
and convergence results
on these multilinear operators.

\begin{proposition}\label{PROP:NN}
Let $\M = \T$ or $\R$ and   $K\ge1$.
Let  
$0 <  \ta < \frac13$ when $\M = \T$
and 
$\frac 1{14} <  \ta < \frac12$ when $\M = \R$.

\smallskip

\noi{\rm(i)} 
We  have 
\begin{align}
\begin{split}
\sup_{t \in \R}
& \big\| \NN_0^{(j)} (t) (\uu_1) - \NN_0^{(j)} (t)(\uu_2)  \big\|_{L^2_\xi} 
\\
& \le 
C
K^{-4 j\ta}
\Big( \| \ft\uu_1\|_{L^2_\xi} + \| \ft\uu_2\|_{L^2_\xi}  \Big)^{j} 
 \|\ft\uu_1 - \ft\uu_2\|_{L^2_\xi},
\end{split}
\label{NX2}
\end{align}

\noi
uniformly in  $j \in \N$ and $K \ge1$.
We also have 
\begin{align}
\begin{split}
\sup_{t \in \R}
& \big\|   \NN^{(j)}_1(t)  (\uu_1) - \NN^{(j)}_1(t) (\uu_2) \big\|_{L^2_\xi}
\\
&  \le C 
 K^{6-4j \ta}
\Big( \| \ft\uu_1\|_{L^2_\xi} + \| \ft\uu_2\|_{L^2_\xi} \Big)^j
 \|\ft\uu_1 - \ft\uu_2\|_{L^2_\xi},
\end{split}
\label{NX1}
\end{align}

\noi
uniformly in  $j \in \N$
and $K \ge1$.

\medskip
\noi{\rm(ii)} 
The following convergence\textup{:}
\begin{align}
\lim_{j\to\infty} 
\sup_{t \in \R}
\big\| \jb{\xi}^{-1}  \NN_2^{(j)} (t)(\uu)  \big\|_{L^\infty_\xi}  = 0
\label{NX3}
\end{align}

\noi
holds uniformly on bounded sets in $L^2(\M)$.

\medskip

\noi{\rm(iii)} 
Suppose that  $\{\uu_\dl\}_{0 < \dl \le 1}$
satisfies the uniform bound\textup{:}
\begin{align}
\sup_{0 < \dl \le 1} \|\uu_\dl\|_{L^2} \le R
\label{NX3a}
\end{align}
for some $R> 0$.
In addition, suppose that, given small $\eps > 0$, 
there exist $N_3 = N_3(\eps, R)  \gg 1 $ 
and small $\dl_3 = \dl_3(\eps, R) > 0$ 
such that 
\begin{align}
\sup_{0 < \dl \le \dl_3} \|\P_{> N} \uu_\dl\|_{L^2} < \eps
\label{NX5}
\end{align}

\noi
for any $N \ge N_3$.
Then, given any $N \in \N$, $T > 0$, and small $\eps > 0$,  
there exist 
small 
$\tau_* = \tau_*(R)> 0$
and 
$\dl_4 = \dl_4(\eps, T,N,  R) > 0$ 
such that 
\begin{align}
\tau  \sum_{j=1}^\infty 
\sup_{|t|\le T }
\| \ind_{|\xi| \le N} \cdot  \EE^{(j)}_\dl (t) (\uu_\dl) \|_{L^2_\xi} < \eps
\label{NX6}
\end{align}

\noi
for any $0 < \tau < \tau_*$ and $0 < \dl < \dl_4$, 
provided that $K = K(R) > 0$ satisfies 
\begin{align}
K^{4\ta} > R.
\label{NX4a}
\end{align}

\noi
The choices of $N_3$ and $\dl_3$ may also depend on
the family $\{\uu_\dl\}_{0<\dl\le1}$.
The dependence of $\dl_4$ and of the depth thresholds
in the proof on these choices is suppressed.

\end{proposition}

In Section~\ref{SEC:per}, 
we present a proof of 
Proposition~\ref{PROP:NN} in the periodic setting.
In Section~\ref{SEC:Euc}, 
we prove
Proposition~\ref{PROP:NN} in the Euclidean setting.
As it is clear from the proof, the multilinear 
terms in \eqref{NF12} and \eqref{NF13}
are indeed absolutely integrable
(and summable in $j \in \N$).

\medskip

By taking the difference of \eqref{NF10} and \eqref{NF11}, 
multiplying by the frequency cutoff 
$\ind_{|\xi| \le N}$, 
and taking a limit as 
 $J\to\infty$, 
 it follows from Proposition \ref{PROP:NN}\,(ii)
 that 
 \begin{align}
\label{NF14}
\begin{split}
\ind_{|\xi| \le N}(\ft\vv- \ftvvd)(t, \xi) 
&
= 
\ind_{|\xi| \le N}(\ft v_{0} - \ft v_{0, \dl})(\xi)\\
& \quad + 
\sum_{j=1}^{\infty} \ind_{|\xi| \le N}\Big( \NN_0^{(j)} ( \vv )(t', \xi) - \NN_0^{(j)} ( \vvd )(t', \xi)  \Big)
\bigg|_0^t 
\\
&
\quad
+
\int_0^t 
\sum_{j=1}^{\infty}
\ind_{|\xi| \le N}\Big( \NN_1^{(j)} (\vv) (t', \xi) - \NN_1^{(j)} (\vvd) (t', \xi)
\Big)
dt' 
\\
&
\quad 
-
\int_0^t 
\sum_{j=1}^{\infty}
\ind_{|\xi| \le N}\cdot \EE_\dl^{(j)} (\vvd) (t', \xi)
dt'  
\end{split}
\end{align}

\noi
for each $\xi \in \ft \M_*$
and also in the $L^2_\xi$-sense thanks to the frequency cutoff.
By assuming Proposition \ref{PROP:NN}, 
we now present a proof of 
 Proposition \ref{PROP:low}.

\begin{proof}[Proof of Proposition \ref{PROP:low}]
Fix small $\eps > 0$ and $T > 0$.
In view of the convergence of $v_{0, \dl}$
to $v_0$ in $L^2(\M)$, 
there exists $R \ge1 $ such that 
\begin{align}
\begin{split}
\| v_{0}\|_{L^2}
+  \sup_{0 <  \dl \le \dl_0 }
\| v_{0, \dl}\|_{L^2}
& \le R, \\
\sup_{0 <  \dl \le \dl_0 }
\| v_0 - v_{0, \dl}\|_{L^2} & < \frac {\eps'} 8
\end{split}
\label{LL1}
\end{align}

\noi
for some small $\dl_0 > 0$, 
where $\eps' = \eps'(\eps,T, R) > 0$ is a small number to be chosen later.
Then, from the conservation of the $L^2$-norm
under KdV~\eqref{KDV} and the scaled ILW~\eqref{sILW1}, 
we have 
\begin{align}
\| \vv\|_{L^\infty_t L^2_x}
+  \sup_{0 <  \dl \le\dl_0 }
\| \vvd \|_{L^\infty_t L^2_x}
\le R.
\label{LL2}
\end{align}

\noi
Also, from 
Remark  \ref{REM:convEC}, 
 there exist
$N_1 = N_1(\eps', v_0)\gg1 $ and small $\dl_1 = \dl_1\big(\eps', v_0, \{v_{0, \dl}\}_{0 < \dl \le 1}\big) > 0$ such that 
\begin{align}
 \| \P_{> N}\vv \|_{L^\infty_t L^2_x}
+ \sup_{0 <  \dl \le  \dl_1}
 \| \P_{> N}\vvd \|_{L^\infty_t L^2_x} < \frac {\eps'}4
\label{LL3}
\end{align}

\noi
for any $N \ge N_1$.
In the following, we assume that $N \ge N_1(\eps', v_0)$.

Write $[0, T] = \bigcup_{\l = 0}^{[T/\tau]} I_\l$, 
where $I_\l = [\l\tau, (\l+1)\tau] \cap [0, T]$.
Here, $[x]$ denotes the integer part of $x \in \R$.
Fix 
small $\tau =  \tau(R) > 0$ (to be chosen later).
Then, 
from \eqref{NF14} (starting at time $\l \tau$), 
we have 
\begin{align}
\label{LL4}
\begin{split}
\| \P_{\le N} (\vv - \vvd) \|_{L^\infty_{I_\l} L^2_x}
&
\le  \| \P_{\le N} (\vv - \vvd) (\l \tau) \|_{L^2_x}\\
& + 
2\sum_{j=1}^\infty \big\| \NN_0^{(j)} (\vv) - \NN^{(j)}_0 (\vvd)  \big\|_{L^\infty_{I_\l} L^2_\xi}
\\
&
\quad 
+
\tau \sum_{j=1}^\infty  \big\|  \NN_1^{(j)} (\vv) - \NN^{(j)}_1 (\vvd)   \big\|_{L^\infty_{I_\l} L^2_\xi}
\\
&
\quad 
+
\tau \sum_{j=1}^\infty \big\| \ind_{|\xi| \le N} \cdot  \EE^{(j)}_\dl (\vvd)  \big\|_{L^\infty_{I_\l} L^2_\xi}.
\end{split}
\end{align}

Fix $\frac14 <  \ta < \frac13$
and choose $K = K(R) \ge 1$ such that 
\begin{align}
K^{4\ta} > R.
\label{LL5}
\end{align}

\noi
Then, from 
 Proposition~\ref{PROP:NN}\,(i)
 with \eqref{LL2},  \eqref{LL5}, and \eqref{LL3}, we have 
\begin{align}
\begin{split}
&
2 \sum_{j=1}^\infty \big\| 
 \NN_0^{(j)} (\vv) - \NN^{(j)}_0 (\vvd)  \big\|_{L^\infty_{I_\l} L^2_\xi}
\\
& \quad 
\le 
C 
\bigg(\sum_{j=1}^\infty (K^{-4\ta} R)^j \bigg)
\| \ft\vv - \ftvvd\|_{L^\infty_{I_\l} L^2_\xi}
\\
& \quad 
\le 
\frac 14
\Big( \| \P_{\le N} (\vv - \vvd)  \|_{L^\infty_{I_\l}L^2_x} + \| \P_{>N} (\vv- \vvd)\|_{L^\infty_{I_\l} L^2_x} \Big)
\\
& \quad 
\le  \frac{1}{4} \| \P_{\le N} (\vv - \vvd)  \|_{L^\infty_{I_\l} L^2_x} + \frac {\eps'} 8
\end{split}
\label{LL6}
\end{align}

\noi
by taking $K = K (R)\ge 1 $ sufficiently large.
Similarly, 
with $K = K(R)\ge 1$ satisfying \eqref{LL5}, 
we have 
\begin{align}
\label{LL7}
\begin{split}
&
\tau  \sum_{j=1}^\infty  \big\|    \NN_1^{(j)} (\vv) - \NN^{(j)}_1 (\vvd)  \big\|_{L^\infty_{I_\l}  L^2_\xi}
\\
& \quad 
\le 
C \tau  K^6
\Big( \| \P_{\le N} (\vv - \vvd)  \|_{L^\infty_{I_\l} L^2_x} + \| \P_{>N} (\vv- \vvd)\|_{L^\infty_{I_\l} L^2_x} \Big)
\\
& \quad 
\le  \frac{1}{4} \| \P_{\le N} (\vv - \vvd)  \|_{L^\infty_{I_\l} L^2_x} + \frac {\eps'} 8
\end{split}
\end{align}

\noi
by taking 
$\tau = \tau (K) > 0$ sufficiently small.
Finally, 
from
Remark \ref{REM:convEC}
(verifying the hypothesis~\eqref{NX5}), 
it follows from Proposition \ref{PROP:NN}\,(iii) that 
there exist small 
$\tau_* = \tau_*(K) = \tau_*(R)> 0$
and 
$\dl_4  = \dl_4\big(\eps', T, N, R, v_0, \{v_{0, \dl}\}_{0 < \dl \le 1}\big)> 0$ 
such that 
\begin{align}
\tau \sum_{j=1}^\infty \big\| \ind_{|\xi| \le N} \cdot  \EE^{(j)}_\dl (\vvd)  \big\|_{L^\infty_T L^2_\xi}
< \frac {\eps'} 8
\label{LL8}
\end{align}

\noi
for any $0 < \tau < \tau_*$ and $0 < \dl < \dl_4$, 
under the condition \eqref{LL5}.

Hence, it follows from \eqref{LL4}, \eqref{LL6}, \eqref{LL7}, 
and
\eqref{LL8}
that 
there exists  small $\dl_5 = \dl_5\big(\eps', T, N, R, v_0,  \{v_{0, \dl}\}_{0 < \dl \le 1}\big)> 0$ 
such that 
\begin{align}
\begin{split}
\| \P_{\le N} (\vv - \vvd) \|_{L^\infty_{I_\l}  L^2_x} 
&
< 
 2 \| \P_{\le N} (\vv - \vvd) (\l\tau) \|_{L^2_x}
+
\frac {3} 4 \eps'
\end{split}
\label{LL8a}
\end{align}

\noi
for any  $0 < \dl \le \dl_5$, 
provided that 
$N = N(\eps', v_0) \gg1$, 
$K = K(R) \gg1$, 
and $0<\tau = \tau(R) \ll 1$.
When $\l = 0$, 
from \eqref{LL8a} with \eqref{LL1}, 
we have 
\begin{align*}
\| \P_{\le N} (\vv - \vvd) \|_{L^\infty_{I_0}  L^2_x} 
< \eps'.
\end{align*}

\noi
By iteratively applying \eqref{LL8a}, 
we obtain
\begin{align*}
\| \P_{\le N} (\vv - \vvd) \|_{L^\infty_{I_\l}  L^2_x} 
\le 
\bigg(\sum_{k = 0}^{\l} 2^k\bigg)\eps'
< 2^{\l+1}\eps'
\end{align*}

\noi
for $\l= 0, 1, \dots, \big[\frac T\tau \big]$.
Therefore, by choosing small $\eps' = \eps'(\eps, T, R)> 0$ 
such that $2^{\frac T\tau + 1} \eps' \le \eps$, 
we obtain
\begin{align*}
\| \P_{\le N} (\vv - \vvd) \|_{C_T  L^2_x} 
< \eps.
\end{align*}

\noi
This proves \eqref{low00}.
\end{proof}

\section{Proof of Proposition \ref{PROP:NN}
on the circle}
\label{SEC:per}

In this section, we present a proof of Proposition~\ref{PROP:NN} 
on the circle $\M=\T$. 
In the following, 
all the estimates hold uniformly in $t \in \R$
and thus 
we drop the $t$-dependence of the multilinear operators
(except for Part (iii)).


\begin{proof}[Proof of Proposition \ref{PROP:NN}
on the circle]
(i)
We first prove 
\begin{align}
\sup_{t \in \R}
& \big\| \NN_0^{(j)} (t) (\uu)   \big\|_{\l^2_\xi} 
 \le 
C
K^{-4j \ta}
 \| \ft\uu\|_{\l^2_\xi}^{j+1},
\label{PQ1}\\
\sup_{t \in \R}
& \big\|   \NN^{(j)}_1(t)  (\uu)  \big\|_{\l^2_\xi}
  \le C 
 K^{6- 4j \ta}
 \| \ft\uu\|_{\l^2_\xi}^{j+1}, 
 \label{PQ2}
 \end{align}

\noi
uniformly in  $j \in \N$ and $K \ge 1$.
We then briefly discuss
how to obtain the difference estimates~\eqref{NX2} and~\eqref{NX1}.

From \eqref{NF12} with \eqref{A4},  we have
\begin{align}
\|  \NN_0^{(j)} (\uu) \|_{\l^2_\xi}
\les 
\bigg\| 
 \sum_{\TT_j \in \Tfr(j)}
\sum_{\substack{\xii \in \Nf(\TT_j) \\ \xi_{r} = \xi} } \ind_{\bigcap_{k=1}^{j} A_k^c} \frac{  \prod_{k=1}^j | \xi^{(k)}|  }{\prod_{k=1}^{j} |\mu_k|} 
\prod_{a \in \TT^\infty_j} |\ft\uu(\xi_a)|
\bigg\|_{\l^2_\xi}.
\label{PQ3}
\end{align}

\noi
From \eqref{A4}, we have
\begin{align}
\begin{split}
 \frac{  \prod_{k=1}^j | \xi^{(k)}|  }{\prod_{k=1}^{j} |\mu_k|} 
&
\le C^j
\frac{\prod_{a \in \TT^{0}_j} |\xi_a| }
{\prod_{a \in \TT^{0}_j  } |\xi_{a} \xi_{a_1} \xi_{a_2}|} \\
& \le C^j
\bigg(\prod_{k = 1}^j(k+2)!\bigg)^{-4\ta} K^{-4j\ta}
\prod_{a \in \TT^{0}_j } \frac{|\xi_a|^\ta}{ | \xi_{a_1} \xi_{a_2}|^{1-\ta}} \\
& \le C^j 
\bigg(\prod_{k = 1}^j(k+2)!\bigg)^{-4\ta} K^{-4j\ta}
\prod_{a \in \TT^{0}_j } \frac{1}{\min(|\xi_{a_1}| ,  | \xi_{a_2}|)^{2-3\ta}}
\end{split}
\label{PQ4}
\end{align}

\noi
on $\bigcap_{k=1}^{j} A_k^c$.
Then, from \eqref{PQ3}, \eqref{PQ4}, 
 \eqref{TJ0}, 
and Cauchy-Schwarz's inequality,  
 we have 
\begin{align}
&
\|  \NN_0^{(j)} (\uu) \|_{\l^2_\xi}
\notag \\
& \quad 
\le C^j
\frac{j!}{\prod_{k = 1}^j\big((k+2)!\big)^{4\ta}} 
K^{-4j \ta}
\notag \\
& \quad 
 \quad 
\times 
\sup_{\TT_j \in \Tfr(j)}
\bigg\| 
\sum_{\substack{\xii\in \mathfrak{N}(\TT_{j})\\ \xi_r=\xi}}
 \bigg( \prod_{a\in\TT^0_j}
 \frac{1}{\min(|\xi_{a_1}|,  | \xi_{a_2}|)^{2-3\ta}} \bigg) 
 \prod_{a \in \TT_{j}^\infty} |\ft\uu(\xi_a)| 
 \bigg\|_{\l^2_\xi} 
\notag \\
& \quad 
\le C^j
\frac{j!}{\prod_{k = 1}^j\big((k+2)!\big)^{4\ta}} 
K^{-4j\ta}  \|\ft\uu\|_{\l^2_\xi}^{j+1} 
\label{PQ5}\\
& \quad 
 \quad 
 \times 
\sup_{\TT_j \in \Tfr(j)}
\bigg\| \bigg(
\sum_{\substack{\xii\in \mathfrak{N}(\TT_{j})\\ \xi_r=\xi}}
 \prod_{a\in\TT^0_j}
 \frac{1}{\min(|\xi_{a_1}|,  | \xi_{a_2}|)^{4-6\ta}} 
\bigg)^\frac 12  \bigg\|_{\l^\infty_\xi} 
\notag \\
& \quad 
\le C^j
\frac{j!}{\prod_{k = 1}^j\big((k+2)!\big)^{4\ta}} 
 K^{-4j \ta} \|\ft\uu\|_{\l^2_\xi}^{j+1} 
\bigg( \sum_{\zeta \in \Z_*} \frac{1}{|\zeta|^{4-6\ta}} \bigg)^{\frac{j}2}
\notag \\
& \quad 
\le
C^{j}
\frac{j!}{\prod_{k = 1}^j\big((k+2)!\big)^{4\ta}} 
 K^{-4j \ta} \|\ft\uu\|_{\l^2_\xi}^{j+1} 
\les
 K^{-4j \ta} \|\ft\uu\|_{\l^2_\xi}^{j+1} , 
\notag
\end{align}

\noi
uniformly in $j \in \N $ and $K \ge1 $, 
 provided that $0  <   \ta < \frac12$. 
This proves
\eqref{PQ1}.

\begin{remark}\label{REM:order}\rm

(i)
With the notations in Definition \ref{DEF:tree3}
and~\eqref{fre1}, 
we have  $\{ a \}_{a \in \TT^0_j}   = \{ r^{(k)}\}_{k = 1}^j$
and $\xi_{r^{(k)}} = \xi^{(k)}$, $k = 1, \dots, j$, 
where $r^{(k)}$ denotes the $k$th root node.
In carrying out the summation in \eqref{PQ5}, 
we start from the last generation.
Namely, for fixed 
$\xi^{(k)}$, $k = 1, \dots, j$, 
we sum over $\xi^{(j)}_1$ (or $\xi^{(j)}_2$).
Then, for fixed 
$\xi^{(k)}$, $k = 1, \dots, j-1$, 
we sum over $\xi^{(j-1)}_1$ (or $\xi^{(j-1)}_2$), 
and so on.
We repeat this process until
we sum over
$\xi^{(1)}_1$ (or $\xi^{(1)}_2$)
for fixed 
$\xi^{(1)} = \xi$.
In handling the other multilinear operators, 
we often proceed in a similar manner.

\smallskip

\noi
(ii)
In view of the gain in the product of factorials
in \eqref{PQ5}, 
any exponential (or factorial) loss in $j$ is negligible and we may drop
such dependence on $j$ in the remaining part of the paper.

\end{remark}


Next, we prove \eqref{PQ2}.
From \eqref{NF12} with \eqref{A4a},  we have
\begin{align}
\|  \NN_1^{(j)} (\uu) \|_{\l^2_\xi}
\les 
\bigg\| 
 \sum_{\TT_j \in \Tfr(j)}
\sum_{\substack{\xii \in \Nf(\TT_j) \\ \xi_{r} = \xi} } \ind_{A_j\cap (\bigcap_{k=1}^{j-1} A_k^c)} \frac{  \prod_{k=1}^j | \xi^{(k)}|  }{\prod_{k=1}^{j-1} |\mu_k|} 
\prod_{a \in \TT^\infty_j} |\ft\uu(\xi_a)|
\bigg\|_{\l^2_\xi}.
\label{PQ6}
\end{align}

\noi
We first consider the   case $j = 1$.
From \eqref{A1}, Young's inequality, and Cauchy-Schwarz's inequality, 
we have 
\begin{align*}
\|  \NN_1^{(1)} (\uu) \|_{\l^2_\xi}
\les K^4 \bigg\|
\sum_{\xi = \xi_1 + \xi_2} \prod_{k = 1}^2 \frac{|\ft \uu(\xi_k)|}{|\xi_k|}
\bigg\|_{\l^2_\xi}
\les K^4 \|\ft  \uu \|_{\l^2_\xi}^2, 
\end{align*}

\noi
yielding \eqref{PQ2} in this case.

Next, we consider the case $j = 2$.
Without loss of generality, assume that $ r^{(2)} = r_1$
(= the left child of the root node $r$ of $\TT_2$), 
which implies that $|\xi^{(2)}| = |\xi^{(1)}_1|$. 
From \eqref{PQ6} with \eqref{A4a}, we have 
\begin{align}
\|   \NN_1^{(2)} (\uu) \|_{\l^2_\xi}
& \les
\bigg\|\sum
_{\substack{\xii \in \Nf(\TT_2) \\ \xi_{r} = \xi} }
\frac{|\ft \uu(\xi^{(1)}_2)|}
{|\xi^{(1)}_2|}
 \ind_{A_2\cap  A_1^c} 
\prod_{k = 1}^2 |\ft \uu(\xi^{(2)}_k)| 
 \bigg\|_{\l^2_{\xi}}.
\label{PQ6a}
\end{align}

\noi
In order to avoid a logarithmic divergence, 
we will make use of dyadic decompositions and Schur's test.
By symmetry, we assume that 
$|\xi^{(2)}_1|\ge |\xi^{(2)}_2|$.
First, suppose that $|\xi^{(1)}_2| \ges |\xi^{(1)}_1|$, 
which implies $|\xi^{(1)}| \les |\xi^{(1)}_2|$.
From \eqref{A4a}, 
we have 
\begin{align}
|\xi^{(1)}_1\xi^{(2)}_1\xi^{(2)}_2|
= |\xi^{(2)}\xi^{(2)}_1\xi^{(2)}_2|
\les |\xi^{(1)}\xi^{(1)}_1\xi^{(1)}_2|, 
\label{PQ6b}
\end{align}

\noi
which in turn implies
$|\xi^{(2)}_2| \les |\xi^{(1)}_2|$
in this case.
Next, we consider the case
 $|\xi^{(1)}_2| \ll |\xi^{(1)}_1|\sim |\xi^{(1)}|$.
If $|\xi^{(2)}_1| \sim |\xi^{(2)}_2| \gg |\xi^{(2)}| = |\xi^{(1)}_1|$, 
then 
it follows from \eqref{PQ6b} that 
\begin{align*}
|\xi^{(1)}|^2 \sim 
|\xi^{(1)}_1|^2
\ll |\xi^{(2)}_1\xi^{(2)}_2|
\les |\xi^{(1)}\xi^{(1)}_2| \ll|\xi^{(1)}|^2
\end{align*}

\noi
which is a contradiction.
Thus, we must have
 $|\xi^{(2)}_2|\les |\xi^{(2)}_1| \sim |\xi^{(2)}|$.
Recall that 
$|\xi^{(2)}| = |\xi^{(1)}_1| \sim |\xi^{(1)}|$.
Hence, from \eqref{PQ6b}, we  have
$|\xi^{(2)}_2| \les |\xi^{(1)}_2|$
once again.

We apply dyadic decompositions to the frequencies
$|\xi^{(k)}_2| \sim N^{(k)}_2\ge 1$, $k = 1, 2$, 
such that 
$N^{(2)}_2\les N^{(1)}_2$.
Then, 
from  \eqref{PQ6a}
and Schur's test, we have 
\begin{align}
\begin{split}
\|  \NN_1^{(2)} (\uu) \|_{\l^2_\xi}
& \les 
\sum_{\substack{N^{(1)}_2 \ges N^{(2)}_2 \ge1 \\\text{dyadic}}}
\frac{1}{N^{(1)}_2}
\bigg\| 
\sum_{\substack{\xi^{(1)} = \xi^{(1)}_1+\xi^{(1)}_2\\
\xi^{(1)}_1 = \xi^{(2)}_1+\xi^{(2)}_2}}
|\ft {\P_{N^{(1)}_2}\uu}(\xi^{(1)}_2)|
|\ft {\P_{N^{(2)}_2}\uu}(\xi^{(2)}_2)|\\
& \hphantom{XXXXXXXXX}
\times |\ft \uu(\xi^{(1)} - \xi^{(1)}_2- \xi^{(2)}_2)|
\bigg\|_{\l^2_{\xi^{(1)}}}\\
& \le
\sum_{\substack{N^{(1)}_2 \ges N^{(2)}_2 \ge1 \\\text{dyadic}}}
\frac{(N^{(2)}_2)^\frac 12}{(N^{(1)}_2)^\frac 12}
\|\ft {\P_{N^{(1)}_2}\uu}\|_{\l^2_\xi}
\|\ft {\P_{N^{(2)}_2}\uu}\|_{\l^2_\xi}
\|\ft\uu\|_{\l^2_\xi}\\
&
 \les
\|\ft\uu\|_{\l^2_\xi}^{3},
\end{split}
\label{PQ6c}
\end{align}

\noi
yielding \eqref{PQ2} when $j = 2$.
Here, 
 $\P_N$ denotes the Littlewood-Paley projector
onto frequencies $\{|\xi|\sim N\}$.

\medskip

Fix $j \ge 3$.
Recalling 
Definition \ref{DEF:tree3}, 
let 
$r^{(k)}$  be the non-terminal node in $\pi_k(\TT_j)$
such that 
$\xi_{r^{(k)}} = \xi^{(k)}$.
Note from 
\eqref{A4a} that  we have 
\begin{align}
|\mu_{j-1}| = 
3|\xi^{(j-1)}\xi^{(j-1)}_1\xi^{(j-1)}_2|
\ges 
j^{-4} |\mu_j|
= 3j^{-4}|\xi^{(j)}\xi^{(j)}_1\xi^{(j)}_2|.
\label{PQ8}
\end{align}

\noi
Recall also from 
Definition~\ref{DEF:tree2}
that 
 $\pb(r^{(j)})$ and $\sbb(r^{(j)})$ denote the parent and sibling nodes of $r^{(j)}$. 
%
Then, from \eqref{A4a} and \eqref{PQ8}, we have, 
on $A_j \cap (\bigcap_{k=1}^{j-1} A_k^c)$, 
\begin{align}
\begin{split}
 \frac{  \prod_{k=1}^j | \xi^{(k)}|  }{\prod_{k=1}^{j-1} |\mu_k|} 
&
\le C^j
\frac{\prod_{a \in \TT^{0}_j} |\xi_a| }
{\prod_{a \in \TT^{0}_j \setminus \{r^{(j)}\} } |\xi_{a} \xi_{a_1} \xi_{a_2}|} 
\\
&\le C^j j^3
\left(\prod_{k=1}^{j-3}(k+2)!\right)^{-4\theta}
K^{-4(j-3)\theta}
\frac{|\xi_{r^{(j)}}|}
{|\xi_{\mathfrak{s}(r^{(j-1)})}|
|\mu_{j-1}|^{\frac 13}|\mu_j|^{\frac 23}}\\
&\quad\times
\prod_{a\in\mathcal T_j^0\setminus
\{\mathfrak p(r^{(j-1)}),r^{(j-1)},r^{(j)}\}}
\frac{|\xi_a|^\theta}
{|\xi_{a_1}\xi_{a_2}|^{1-\theta}}\\
& \le C^j 
j^3 \bigg(\prod_{k = 1}^{j-3} (k+2)!\bigg)^{-4\ta}
 K^{-4(j-3)\ta} 
\prod_{a \in \TT^{0}_j}\frac{1}
{
\min( | \xi_{a_1}|, | \xi_{a_2}|)^{\frac{1}{2}+\eps_0}} 
\end{split}
\label{PQ9}
\end{align}

\noi
for some small $\eps_0> 0$, 
provided that $0 < \ta <  \frac 13$, 
where 
\begin{align*}
\xi^{(k)}_{\max} = \max\big(|\xi^{(k)}|,|\xi^{(k)}_1|, |\xi^{(k)}_2|\big)
\qquad \text{and}  \qquad
\xi^{(k)}_{\min} = \min\big(|\xi^{(k)}|,|\xi^{(k)}_1|, |\xi^{(k)}_2|\big).
\end{align*}

\noi
Then, proceeding as in \eqref{PQ5}
with \eqref{PQ6},  \eqref{PQ9}, and \eqref{TJ0}, we have 
\begin{align*}
&
\|  \NN_1^{(j)} (\uu) \|_{\l^2_\xi}
\\
& \quad 
\le
C^j
j^3\cdot j!\cdot  \bigg(\prod_{k = 1}^{j-3} (k+2)!\bigg)^{-4\ta}
K^{-4(j-3) \ta}\\
& \quad 
 \quad 
\times 
\sup_{\TT_j \in \Tfr(j)}
\bigg\| 
\sum_{\substack{\xii\in \mathfrak{N}(\TT_{j})\\ \xi_r=\xi}}
 \bigg( \prod_{a\in\TT^0_j}
 \frac{1}{\min(|\xi_{a_1}|,  | \xi_{a_2}|)^{\frac 12 + \eps_0}} \bigg) 
 \prod_{a \in \TT_{j}^\infty} |\ft\uu(\xi_a)| 
 \bigg\|_{\l^2_\xi} 
\\
& \quad 
\le  
C^{j}
j^3\cdot j!\cdot  \bigg(\prod_{k = 1}^{j-3} (k+2)!\bigg)^{-4\ta}
 K^{-4(j-3) \ta} \|\ft\uu\|_{\l^2_\xi}^{j+1} \\
& \quad 
\les  
K^{-4(j-3) \ta} \|\ft\uu\|_{\l^2_\xi}^{j+1} ,  
\end{align*}

\noi
uniformly in $j \ge 3$ and $K \ge 1$, 
 provided that $0  <   \ta < \frac13$. 
This proves
\eqref{PQ2}.

Let us briefly discuss
the difference estimate \eqref{NX2}.
A similar consideration holds for~\eqref{NX1}.
For \eqref{NX2},  
by the multilinearity of $\NN_0^{(j)}$, 
we can write 
$\NN_0^{(j)}(\uu_1) - \NN_0^{(j)} (\uu_2) $
as the sum of $O(j)$ differences, 
each of which contains exactly one factor of $\ft \uu_1 - \ft \uu_2$;
see the proof of \cite[Lemma~3.11]{GKO}.
This $O(j)$ loss does not affect the bound \eqref{NX2}
thanks to the fast decay in $j$, appearing in~\eqref{PQ5}.
The same  comment applies to the real line case
presented in the next section.
This concludes the proof of Part (i).

\medskip

\noi
(ii)
From  \eqref{NF12} and \eqref{A4},  we have
\begin{align}
\|  \jb{\xi}^{-1} \NN_2^{(j)} (\uu) \|_{\l^\infty_\xi}
\les 
\bigg\| 
 \sum_{\TT_j \in \Tfr(j)}
\sum_{\substack{\xii \in \Nf(\TT_j) \\ \xi_{r} = \xi} } \ind_{\bigcap_{k=1}^{j-1} A_k^c} \frac{  \prod_{k=2}^j | \xi^{(k)}|  }{\prod_{k=1}^{j-1} |\mu_k|} 
\prod_{a \in \TT^\infty_j} |\ft\uu(\xi_a)|
\bigg\|_{\l^\infty_\xi}.
\label{PE1}
\end{align}

\noi
Here, we replaced $\bigcap_{k=1}^{j} A_k^c$
in \eqref{NF12} by $\bigcap_{k=1}^{j-1} A_k^c$
so that the analysis presented in this part
can be adapted to study $\EE_\dl^{(j)}(\uu)$
in \eqref{NF13}.

Let $j=1$.\footnote{In proving the limit \eqref{NX3}, we do not need to consider the  case
$j = 1, 2$.
We, however, present it here since it is needed to treat the error term $\EE_\dl^{(j)}$
in Part (iii) below.} From Young's inequality, we have
\begin{align}
\| \jb{\xi}^{-1} \NN_2^{(1)} (\uu) \|_{\l^\infty_\xi}
\les  \| \ft\uu\|_{\l^2_\xi}^2.
\label{PE1a}
\end{align}

Next, we consider the case $j = 2$.
Without loss of generality, assume that $ r^{(2)} = r_1$
(=~the left child of the root node $r$ of $\TT_2$), 
which implies that $|\xi^{(2)}| = |\xi^{(1)}_1|$. 
Then, 
from \eqref{PE1} and~\eqref{A4}, we have
\begin{align}
\begin{split}
\| \jb{\xi}^{-1} \NN_2^{(2)} (\uu) \|_{\l^\infty_\xi}
& \les
\bigg\|\sum_{\xi^{(1)} = \xi^{(1)}_1+\xi^{(1)}_2}
\frac{|\ft \uu(\xi^{(1)}_2)|}
{|\xi^{(1)}_2|}
\bigg\|_{\l^\infty_{\xi^{(1)}}}\\
& \quad \times 
 \bigg\|\sum_{\xi^{(1)}_1 = \xi^{(2)}_1+\xi^{(2)}_2}
\prod_{k = 1}^2 |\ft \uu(\xi^{(2)}_k)| \bigg\|_{\l^\infty_{\xi^{(1)}_1}}\\
& \les  \| \ft\uu\|_{\l^2_\xi}^3.
\end{split}
\label{PEE1}
\end{align}

Fix $j\ge3$. 
From 
\eqref{A4}, we have 
\begin{align*}
  |\mu_k| = 3 |\xi^{(k)} \xi^{(k)}_1 \xi^{(k)}_2| 
\sim  |\wt\mu_k |
\ges
\big((k+2)! K\big)^4, 
\quad k=1, \ldots, j-1, 
\end{align*}

\noi
on the domain of integration $\bigcap_{k=1}^{j-1} A_k^c$.
Recalling 
Definition \ref{DEF:tree3}, 
let 
$r^{(j)}$ be the non-terminal node in $\pi_j(\TT_j)$
such that 
$\xi_{r^{(j)}} = \xi^{(j)}$.
Note that 
$r^{(j)}  \ne r$ since $j \ge 3$.
Recall also from 
Definition~\ref{DEF:tree2}
that 
 $\pb(r^{(j)})$ and $\sbb(r^{(j)})$ denote the parent and sibling nodes of $r^{(j)}$. 
Then, 
 we can rewrite the multiplier in \eqref{PE1} as 
\begin{align}
& \frac{  \prod_{k=2}^j | \xi^{(k)}|  }{\prod_{k=1}^{j-1} |\mu_k|} 
\notag\\
&\quad 
\le C^j
\frac{\prod_{a \in \TT^{0}_j\setminus\{r\}} |\xi_a| }
{\prod_{a \in \TT^{0}_j \setminus \{r^{(j)}\} } |\xi_{a} \xi_{a_1} \xi_{a_2}|} 
\notag\\
&\quad  = 
C^j \frac{1}{|\xi_r\xi_{\sbb(r^{(j)})}|}
\prod_{a\in\TT_j^0\setminus\{\pb(r^{(j)}),r^{(j)}\}}
\frac{|\xi_a|}{|\xi_a\xi_{a_1}\xi_{a_2}|}
%
\label{PE3}
\\
&\quad  \le C^j
\bigg(\prod_{k = 1}^{j-2} (k+2)!\bigg)^{-4\ta}
 K^{-4(j-2)\ta}
\frac{1}{| \xi_{\sbb(r^{(j)})}|}
\prod_{a \in \TT^{0}_j\setminus\{r^{(j)}, \pb(r^{(j)})\}}\frac{ |\xi_a|^\ta }
{|\xi_{a_1} \xi_{a_2}|^{1-\ta}} 
\notag\\
&\quad 
\le C^j \bigg(\prod_{k = 1}^{j-2} (k+2)!\bigg)^{-4\ta}K^{-4(j-2)\ta} 
\prod_{a \in \TT^0_j \setminus \{r^{(j)}\}} \frac{1}{\min(|\xi_{a_1}| ,  | \xi_{a_2}|)^{\frac 12 + \eps_0}}
\notag
\end{align}

\noi
for some small $\eps_0> 0$, 
provided that $0 <  \ta <  \frac 12$, 
where, in the last step, we applied the triangle inequality:
$|\xi_a|\les \max(|\xi_{a_1}|, |\xi_{a_2}|)$.
Hence, noting that
any child node has 
a unique parent
and recalling the definition of the projection $\Pi_{j-1}$
from Definition \ref{DEF:tree3}, 
it follows from~\eqref{PE3}, 
\eqref{TJ0}, 
Remark~\ref{REM:terminal}, 
and Cauchy-Schwarz's inequality
that 
\begin{align}
&
\text{RHS of }\eqref{PE1}
\notag \\
& \quad 
\le C^j 
j!\cdot\bigg(\prod_{k = 1}^{j-2} (k+2)!\bigg)^{-4\ta}
K^{-4(j-2)\ta}
\notag \\
& \quad 
 \quad 
\times 
\sup_{\TT_j \in \Tfr(j)}
\Bigg\{
\bigg\| 
\sum_{\substack{\xii\in \mathfrak{N}(\Pi_{j-1}(\TT_{j}))\\ \xi_r=\xi}}
 \bigg( \prod_{a\in\TT^0_j \setminus\{r^{(j)} \}} 
 \frac{1}{\min(|\xi_{a_1}|,  | \xi_{a_2}|)^{\frac 12 + \eps_0}} \bigg) 
 \notag \\
& \hphantom{XXXXXX}
\times 
\bigg( \prod_{a \in \Pi_{j-1}(\TT_{j})^\infty\setminus\{r^{(j)}\}} |\ft\uu(\xi_a)| \bigg)
 \bigg\|_{\l^\infty_\xi} 
 \notag \\
& \hphantom{XXXXXX}
\times 
\sup_{\xi_{r^{(j)}} \in \Z_*}
 \sum_{\xi_{r^{(j)}} = \xi^{(j)}_1 + \xi^{(j)}_2 }
|\ft\uu(\xi_1^{(j)})| |\ft\uu(\xi_2^{(j)})| \Bigg\}
\label{PE4}\\
& \quad 
\le C^j 
j!\cdot
\bigg(\prod_{k = 1}^{j-2} (k+2)!\bigg)^{-4\ta} K^{-4(j-2)\ta}\|\ft\uu\|_{\l^2_\xi}^{j+1} 
\notag \\
& \quad 
 \quad 
 \times 
\sup_{\TT_j \in \Tfr(j)}
\bigg\| \bigg(
\sum_{\substack{\xii\in \mathfrak{N}(\Pi_{j-1}(\TT_{j}))\\ \xi_r=\xi}}
 \prod_{a\in\TT^0_j \setminus\{r^{(j)} \}} 
 \frac{1}{\min(|\xi_{a_1}|,  | \xi_{a_2}|)^{1+2\eps_0}} 
\bigg)^\frac 12  \bigg\|_{\l^\infty_\xi} 
\notag \\
& \quad 
\le C^j
j!\cdot\bigg(\prod_{k = 1}^{j-2} (k+2)!\bigg)^{-4\ta} 
K^{-4(j-2)\ta} \|\ft\uu\|_{\l^2_\xi}^{j+1} 
\bigg( \sum_{\zeta \in \Z_*} \frac{1}{|\zeta|^{1+2\eps_0}} \bigg)^{\frac{j-1}2}
\notag \\
& \quad 
\le
C^j j!\cdot\bigg(\prod_{k = 1}^{j-2} (k+2)!\bigg)^{-4\ta}
K^{-4(j-2)\ta}
 \|\ft\uu\|_{\l^2_\xi}^{j+1} 
  \too 0, 
\notag 
\end{align}

\noi
as $j \to \infty$, 
 provided that $0 <   \ta < \frac12$. 
This proves
\eqref{NX3}.

\medskip

\noi
(iii) We now prove the uniform (in $|t| \le T$) bound \eqref{NX6}, assuming \eqref{NX5}.
Fix small $\eps > 0$ and 
 $N \in \N$.
Then, 
given small $\eps_0 > 0$, 
%
it follows from the assumption
\eqref{NX5}
that 
\begin{align}
 \|\P_{> \dl^{-\frac 25 +\eps_0}} \uu_\dl\|_{L^2} < N^{-\frac 32} \eps
\label{PE7c}
\end{align}

\noi
for any sufficiently small
$\dl = \dl(\eps,   N, R)>0$.
In the following, we only consider the case $j \ge 3$, 
since similar analysis holds for the case 
 $j = 1, 2$ (but is simpler, using
the  bounds~\eqref{PE1a} and~\eqref{PEE1}).

\smallskip

\noi
$\bullet$ {\bf Case 1:}
 $\max\big(|\xi^{(j)}_1|, |\xi^{(j)}_2|\big) \le \dl^{-\frac 25 +\eps_0}$.
\\
\indent
In this case, we have  $|\xi^{(j)}| \les \dl^{-\frac 25 +\eps_0}$.
Then, 
from \eqref{Xi1} and \eqref{L3}, we have 
\begin{align}
\big|\Xi_\KDV \big(\xi^{(j)}, \xi_1^{(j)},\xi_2^{(j)}\big)
- \Xi_\dl \big(\xi^{(j)}, \xi_1^{(j)},\xi_2^{(j)}\big)\big|\les \dl^{5\eps_0}.
\label{PE8}
\end{align}

\noi
Thus, 
from 
\eqref{phi1} and 
the mean value theorem 
with \eqref{PE8}, we have 
\begin{align}
\big|\varphi_\dl\big(t, \xi^{(j)}, \xi^{(j)}_1, \xi^{(j)}_2\big)\big|
\les T \dl^{5\eps_0}
\label{PE9}
\end{align}

\noi
for any $|t| \le T$.
Then, 
by repeating  the computations in \eqref{PE4}
but without throwing away the factor
$\varphi_\dl\big(t, \xi^{(j)}, \xi^{(j)}_1, \xi^{(j)}_2\big)$, 
appearing in \eqref{NF13}, 
and applying \eqref{PE9} and \eqref{NX3a}, 
we have 
\begin{align*}
&
\sum_{j=1}^\infty \| \ind_{|\xi| \le N} \cdot  \EE^{(j)}_\dl(t) (\uu_\dl) \|_{\l^2_\xi} 
\\
& \quad 
\le
N^{\frac32} 
\sum_{j=1}^\infty
C^j j!\cdot\bigg(\prod_{k = 1}^{j-2} (k+2)!\bigg)^{-4\ta}
K^{-4(j-2)\ta}\\
& \quad 
 \quad 
\times \sup_{\TT_j \in \Tfr(j)}
\bigg\| 
\sum_{\substack{\xii\in \mathfrak{N}( \TT_{j})\\ \xi_r=\xi}}
\ind_{\max(|\xi^{(j)}_1|, |\xi^{(j)}_2|) \le \dl^{-\frac 25 +\eps_0}}
\cdot \big|\varphi_\dl\big(t, \xi^{(j)}, \xi^{(j)}_1, \xi^{(j)}_2\big)\big|\\
& \hphantom{XXXXXXX}\times 
 \bigg( \prod_{ a\in\TT^0_j \setminus \{r^{(j)} \}} 
 \frac{1}{\min(|\xi_{a_1}|,  | \xi_{a_2}|)^{\frac 12 + \eps_0}} \bigg) \prod_{a \in \TT^\infty_j} 
 |\ft\uu_\dl(\xi_a) |\bigg\|_{\l^\infty_\xi} 
\\
& \quad 
\les T N^{\frac32} 
 \dl^{5\eps_0}
\sum_{j = 1}^\infty K^{-4(j-2)\ta}
 R^{j+1}  \\
& \quad \les T N^{\frac32} 
 \dl^{5\eps_0}
 K^{12 \ta}, 
\end{align*}

\noi
 provided that $0 <   \ta < \frac12$, 
 where, in the last step,  we used 
\begin{align*}
\sum_{j = 1}^\infty K^{-4(j-2)\ta}
 R^{j+1} 
 \les K^{4\ta} R^2 
<   K^{12\ta}
\end{align*}

\noi
under the assumption \eqref{NX4a}: $K^{4\ta} > R$.
Therefore, 
by first choosing $\tau = \tau(K) = \tau(R)> 0$ sufficiently 
small 
and then $\dl = \dl(\eps,  T, N)>0$ sufficiently small, 
we obtain
\begin{align}
\text{LHS of \eqref{NX6}} \le C \tau 
T N^\frac{3}{2}
 K^{12\ta}
 \dl^{5\eps_0} 
 < \frac \eps 2.
\label{PE10}
\end{align}

\medskip

\noi
$\bullet$ {\bf Case 2:}
 $\max\big(|\xi^{(j)}_1|, |\xi^{(j)}_2|\big) >  \dl^{-\frac 25 +\eps_0}$.
\\
\indent
In this case, we do not make use
of a decay from $\big|\varphi_\dl\big(t, \xi^{(j)}, \xi^{(j)}_1, \xi^{(j)}_2\big)\big|$
but we simply bound it by $2$.
Without loss of generality, assume 
$|\xi^{(j)}_1| >  \dl^{-\frac 25 +\eps_0}$.
We note that 
the node $r^{(j)}_1$ corresponding to the frequency 
$\xi^{(j)}_1$
belongs to $\TT^\infty_j$.
Then, from a slight modification
of~\eqref{PE4} (and summing in $j \in \N$)
with \eqref{PE7c}
(for sufficiently small
$\dl = \dl(\eps,   N, R)>0$),
we obtain 
\begin{align}
\text{LHS of \eqref{NX6}} 
\le C \tau 
 N^\frac{3}{2} K^{8\ta}
 \cdot  N^{-\frac 32} \eps
<   \frac \eps 2
\label{PE11}
\end{align}

\noi
\noi
by  choosing $\tau = \tau(K) = \tau (R)> 0$ sufficiently 
small,  
provided that $K^{4\ta} > R$
such that we have 
\begin{align*}
\sum_{j = 1}^\infty K^{4(2-j) \ta}
 R^{j} 
 \les K^{4\ta} R
<   K^{8\ta}.
\end{align*}

\medskip

Therefore, 
the desired bound 
\eqref{NX6}
follows 
from \eqref{PE10} and \eqref{PE11}.
This concludes the proof of Proposition \ref{PROP:NN}
in the periodic case.
\end{proof}

\section{Proof of Proposition \ref{PROP:NN}
on the real line}
\label{SEC:Euc}

In this section, we present a proof of Proposition~\ref{PROP:NN} 
on the real line  $\M=\R$. 
In the periodic setting presented in Section \ref{SEC:per}, 
thanks to the mean-zero assumption, 
frequencies were lower bounded by 1 in size, 
which is no longer true  on the real line.
In particular, 
we need to deal with  low frequency issues.
In some cases, 
a change of variables comes to the rescue.
In some other cases, however, 
we need to proceed with more intricate analysis;
see, for example, the discussion on 
$G_{\TT_j}$ defined in \eqref{RE4}
 below.

As in Section \ref{SEC:per}, 
all the estimates hold uniformly in $t \in \R$
and thus 
we drop the $t$-dependence of the multilinear operators
(except for Part (iii)).

\begin{proof}[Proof of Proposition \ref{PROP:NN}
on the real line]
(i)
We only prove 
\begin{align}
\sup_{t \in \R}
& \big\| \NN_0^{(j)} (t) (\uu)   \big\|_{L^2_\xi} 
 \le 
C
K^{-4j \ta}
 \| \ft\uu\|_{L^2_\xi}^{j+1},
\label{PK1}\\
\sup_{t \in \R}
& \big\|   \NN^{(j)}_1(t)  (\uu)  \big\|_{L^2_\xi}
  \le C 
 K^{6- 4j \ta}
 \| \ft\uu\|_{L^2_\xi}^{j+1}, 
 \label{PK2}
 \end{align}

\noi
uniformly in  $j \in \N$.
Once we have \eqref{PK1} and \eqref{PK2}, 
the difference estimates
\eqref{NX2} and~\eqref{NX1}
follow from a straightforward modification;
see the discussion at the end of Part (i) in 
the proof of Proposition \ref{PROP:NN}
on the circle, presented in Section \ref{SEC:per}.

Let us first establish the following basic estimate:
\begin{align}
\begin{split}
 \bigg\|
\int_{\xi = \xi_1 + \xi_2}
\ind_{|\xi \xi_1\xi_2|\ge \wt K}
\frac{|\xi|}{|\xi \xi_1\xi_2|} \prod_{k = 1}^2 f_k (\xi_k) d\xi_1 
\bigg\|_{L^2_\xi}
& 
\les \wt K^{-\ta} \prod_{k = 1}^2 \|f_k \|_{L^2_\xi}, 
\end{split}
\label{X1}
\end{align}

\noi
uniformly in $\wt K \ge 1$, 
provided that 
$0 <  \ta < \frac 12$.

By Cauchy-Schwarz's inequality, 
we have 
\begin{align}
\text{LHS of }\eqref{X1}
\les 
\|A(\xi)\|_{L^\infty_\xi}^\frac 12 
\prod_{k = 1}^2 \|f_k \|_{L^2_\xi}, 
\label{X2}
\end{align}

\noi
where $A(\xi)$ is given by 
\begin{align}
A(\xi) = \intt_{\xi = \xi_1 + \xi_2}
\ind_{|\xi \xi_1\xi_2|\ge \wt K}
\frac{|\xi|^2}{\jb{\xi \xi_1\xi_2}^{2}} d\xi_1, 
\label{X3}
\end{align}

\noi
where 
$\jb{x} = (1+x^2)^\frac 12$.
By symmetry, assume  that $|\xi_1| \ge |\xi_2| $, 
which implies $|\xi_1|\ges |\xi|$.
First, suppose that 
  $ |\xi_2| \ges 1$. 
Then, we have 
\begin{align}
\begin{split}
\eqref{X3}
&
\les 
\wt K^{-2\ta}
\intt_{\substack{\xi = \xi_1 + \xi_2\\|\xi_2|\ges 1}}
\frac{|\xi|^2}{\jb{\xi \xi_1\xi_2}^{2(1-\ta)}} d\xi_1
\les \wt K^{-2\ta}
\intt_{|\xi_2|\ges 1}
\frac{d\xi_2}{ | \xi_{2} |^{4-6\ta} } \\
& \les \wt K^{-2\ta}, 
\end{split}
\label{X4}
\end{align}

\noi
provided that $0 < \ta < \frac 12$.
Next, suppose that 
 $|\xi_{2}| \ll 1 \les  |\xi_{1}| \sim |\xi|$. 
In this case, 
we  perform a change of variables
 from $\xi_{1}$ to $\ze = \xi \xi_{1} (\xi - \xi_1)$. 
Note that we have 
\[|\partial_{\xi_{1}} \ze| = |\xi (\xi -2  \xi_{1})| 
= |\xi (\xi_{2} -  \xi_{1})|\sim |\xi|^2.\] 

\noi
Thus, we have 
\begin{align}
\begin{split}
\eqref{X3}
&
\les 
\wt K^{-2\ta}
\intt_{\substack{\xi = \xi_1 + \xi_2\\|\xi_{2}| \ll 1 \les  |\xi_{1}|}}
\frac{|\xi|^2}{\jb{\xi \xi_1\xi_2}^{2(1-\ta)}} d\xi_1
 \les \wt K^{-2\ta}
\int \frac{1}{\jb{\ze}^{2(1-\ta)}} d\ze\\
& \les \wt K^{-2\ta}, 
\end{split}
\label{X5}
\end{align}

\noi
provided that $0 < \ta < \frac 12$.
Hence, the basic estimate \eqref{X1}
follows from \eqref{X2}, \eqref{X3}, \eqref{X4}, and~\eqref{X5}.

\medskip

We are now ready to prove \eqref{PK1}.
From \eqref{NF12} with \eqref{A4},  we have
\begin{align*}
\|  \NN_0^{(j)} (\uu) \|_{L^2_\xi}
\les 
\bigg\| 
 \sum_{\TT_j \in \Tfr(j)}
\intt_{\substack{\xii \in \Nf(\TT_j) \\ \xi_{r} = \xi} } \ind_{\bigcap_{k=1}^{j} A_k^c} \frac{  \prod_{k=1}^j | \xi^{(k)}|  }{\prod_{k=1}^{j} \jb{\mu_k}} 
\prod_{a \in \TT^\infty_j} |\ft\uu(\xi_a)|
\bigg\|_{L^2_\xi}.
\end{align*}

\noi
By iteratively applying the basic estimate \eqref{X1}
to the frequency sets $(\xi_{r^{(k)}}, \xi_{r^{(k)}_1}, \xi_{r^{(k)}_2})$
for the $k$th generation
with the lower bound on $|\mu_k|$ as in \eqref{A4}
and \eqref{TJ0}, 
we obtain
\begin{align}
\begin{split}
\|  \NN_0^{(j)} (\uu) \|_{L^2_\xi}
& \le
C^{j} j!\cdot
\bigg(\prod_{k = 1}^{j} (k+2)!\bigg)^{-4\ta}
 K^{-4j \ta} \|\ft\uu\|_{L^2_\xi}^{j+1} \\
& \les
 K^{-4j \ta} \|\ft\uu\|_{L^2_\xi}^{j+1} , 
\end{split}
\label{PK8}
\end{align}

\noi
 provided that $0  <   \ta < \frac12$. 
This proves
\eqref{PK1}.

\medskip

Next, we prove \eqref{PK2}.
From \eqref{NF12} with \eqref{A4a},  we have
\begin{align}
\|  \NN_1^{(j)} (\uu) \|_{L^2_\xi}
\les 
\bigg\| 
 \sum_{\TT_j \in \Tfr(j)}
\intt_{\substack{\xii \in \Nf(\TT_j) \\ \xi_{r} = \xi} } \ind_{A_j\cap (\bigcap_{k=1}^{j-1} A_k^c)} \frac{  \prod_{k=1}^j | \xi^{(k)}|  }{\prod_{k=1}^{j-1} \jb{\mu_k}} 
\prod_{a \in \TT^\infty_j} |\ft\uu(\xi_a)|
\bigg\|_{L^2_\xi}.
\label{PKK1}
\end{align}

\noi
{\bf Step 1:}
We first consider the   case $j = 1$.
If $|\xi|\les 1$, then Young's inequality yields
\begin{align}
\|  \NN_1^{(1)} (\uu) \|_{L^2_\xi}
\les \| \ft \uu \|_{L^2_\xi}^2.
\label{WWX1a}
\end{align}

\noi
Thus, we assume that $|\xi|\gg 1$.
We first consider the case
$|\xi \xi_1\xi_2|\le 1$.
Without loss of generality, assume $|\xi_1|\ge |\xi_2|$.
Then, we have 
$|\xi_1|\sim |\xi|\gg1$
and $|\xi_2|\les |\xi|^{-2} $, 
from which we have
\begin{align}
\sup_{|\xi|\gg 1} \int_{|\xi - \xi_1|\les |\xi|^{-2}}
\xi^2d\xi_1 \les 1.
\label{WWX1}
\end{align}

\noi
Then, 
the bound \eqref{WWX1a} follows
from Cauchy-Schwarz's inequality and \eqref{WWX1}.

Next, we  consider the case
$|\xi \xi_1\xi_2|> 1$.
From \eqref{A1}, we have $\jb{\xi\xi_1\xi_2} \les K^4$.
Then, by the basic estimate \eqref{X1} with $\ta = 0$, we have 
\begin{align*}
\|  \NN_1^{(1)} (\uu) \|_{L^2_\xi}
& \les K^4 \bigg\|
\intt_{\xi = \xi_1 + \xi_2}
\ind_{|\xi \xi_1\xi_2|>1}
 \frac{|\xi|}{\jb{\xi\xi_1\xi_2}}\prod_{k = 1}^2 |\ft \uu(\xi_k)|
d\xi_1 \bigg\|_{L^2_\xi}\\
& \les K^4 \| \ft \uu \|_{L^2_\xi}^2.
\end{align*}

\noi
This proves  \eqref{PK2} for $j = 1$.

\medskip

\noi
{\bf Step 2:}
Fix $j \ge 2$.
We first consider the case
$|\xi^{(j)}|\les 1$.
From \eqref{PKK1}
with the notations in Definition~\ref{DEF:tree3}, we have 
\begin{align*}
& \|  \NN_1^{(j)} (\uu) \|_{L^2_\xi}\\
& \quad \les 
 \sum_{\TT_j \in \Tfr(j)}
\Bigg\{\bigg\| 
\intt_{\substack{\xii \in \Nf(\Pi_{j-1}(\TT_j)) \\ \xi_{r} = \xi} } \ind_{\bigcap_{k=1}^{j-1} A_k^c} 
\frac{  \prod_{k=1}^{j-1} | \xi^{(k)}|  }{\prod_{k=1}^{j-1} |\mu_k|} 
 \prod_{a \in \Pi_{j-1}(\TT_{j})^\infty \setminus \{r^{(j)}\}} |\ft\uu(\xi_a)|
\bigg\|_{L^2_\xi}\\
& \quad  \hphantom{XXXXXX}
\times 
\bigg\| \intt_{\xi^{(j)}= \xi^{(j)}_1+ \xi^{(j)}_2}
|\ft \uu(\xi^{(j)}_1)||\ft \uu(\xi^{(j)}_2)|d \xi^{(j)}_1
\bigg\|_{L^\infty_{|\xi^{(j)}|\les1}}
\Bigg\}.
\end{align*}

\noi
Noting that the first factor on the right-hand side
essentially corresponds to 
$ \NN_0^{(j-1)} (\uu)$ to which we apply \eqref{PK8} (with $j$ replaced by $j-1$)
and applying Young's inequality to the second factor
on the right-hand side above, 
we obtain
\begin{align*}
\|  \NN_1^{(j)} (\uu) \|_{L^2_\xi}
& \le
C^{j}\cdot j!\cdot 
\bigg(\prod_{k = 1}^{j-1} (k+2)!\bigg)^{-4\ta}
K^{-4(j-1) \ta} \|\ft\uu\|_{L^2_\xi}^{j+1} \\
& 
\les  
K^{-4(j-1) \ta} \|\ft\uu\|_{L^2_\xi}^{j+1} ,  
\end{align*}

\noi
 provided that $0  <   \ta < \frac12$. 
In the following, we assume that $|\xi^{(j)}|\gg 1$.

\medskip

\noi
{\bf Step 3:}
Next, we consider the  case $j = 2$
(assuming that $|\xi^{(j)}|\gg 1$).
Without loss of generality, assume that $ r^{(2)} = r_1$
(= the left child of the root node $r$ of $\TT_2$), 
which implies that $|\xi^{(2)}| = |\xi^{(1)}_1| \gg1 $.
Moreover, by symmetry, we assume that 
$|\xi^{(2)}_1|\ge |\xi^{(2)}_2|$.
Under these assumptions,  arguing as in the $j = 2$ case 
in the proof of  \eqref{PQ2}
presented in Section \ref{SEC:per}, 
we see that 
$|\xi^{(2)}_2| \les |\xi^{(1)}_2|$.
Then, by proceeding as in 
\eqref{PQ6c}
with Schur's test
(which also holds for 
dyadic numbers  $N^{(k)}_2\in 2^\Z$, $k = 1, 2$), 
 we obtain 
\eqref{PK2} for $j = 2$.

\medskip

\noi
{\bf Step 4:}
We now consider the general case $j \ge 3$
(assuming that $|\xi^{(j)}|\gg 1$).
We first consider the case 
 $\pb(r^{(j)})  = r^{(j-1)}$.
In this case, 
from \eqref{PKK1}, we have 
\begin{align}
\begin{split}
\|  \NN_1^{(j)} (\uu) \|_{L^2_\xi}
& \les 
\bigg\| 
 \sum_{\substack{\TT_j \in \Tfr(j)\\\pb(r^{(j)})  = r^{(j-1)}}}
\intt_{\substack{\xii \in \Nf(\Pi_{j-2}(\TT_j)) \\ \xi_{r} = \xi} } \ind_{\bigcap_{k=1}^{j-2} A_k^c} \frac{  \prod_{k=1}^{j-2} | \xi^{(k)}|  }{\prod_{k=1}^{j-2} \jb{\mu_k}} \\
&\hphantom{XX} 
\times \prod_{a \in (\Pi_{j-2}(\TT_j))^\infty \setminus\{r^{(j-1)}\}} 
|\ft\uu(\xi_a)|
\cdot |Q(\xi_{r^{(j-1)}})|
\bigg\|_{L^2_\xi}, 
\end{split}
\label{PKK11}
\end{align}

\noi
where
$Q(\xi_{r^{(j-1)}})$ is given by 
\begin{align*}
Q(\xi_{r^{(j-1)}})
= 
\intt_{\substack{\xi^{(j-1)} = \xi^{(j-1)}_1+\xi^{(j-1)}_2\\
\xi^{(j-1)}_1 = \xi^{(j)}_1+\xi^{(j)}_2}}
\ind_{A_j\cap A_{j-1}^c} 
\frac{|\xi^{(j-1)}\xi^{(j)}|}{\jb{\mu_{j-1}}}
|\ft \uu(\xi^{(j-1)}_2)|
\prod_{k = 1}^2|\ft \uu(\xi^{(j)}_k)|.
\end{align*}

\noi
Here, without loss of generality, we assumed that 
$r^{(j)} = r^{(j-1)}_1$, 
namely, 
$r^{(j)}$
 is the left child of $r^{(j-1)}$
(in the planar graphical representation of the binary tree $\pi_{j-1}(\TT_j)$;
see Definition~\ref{DEF:tree3}
for the notation $\pi_{j-1}(\TT_j)$).

From the $j = 2$ case discussed above, 
we immediately see that 
\begin{align}
\|Q(\xi_{r^{(j-1)}})\|_{L^2_{\xi_{r^{(j-1)}}}}
 \les
\|\ft\uu\|_{L^2_\xi}^{3}.
\label{PKK12}
\end{align}

\noi
Furthermore, by noting 
that the right-hand side of 
 \eqref{PKK11} 
corresponds to $\NN_0^{(j-2)}$ 
(where $\ft \uu (\xi_{r^{(j-1)}})$
is replaced by $Q(\xi_{r^{(j-1)}})$), 
it follows from 
\eqref{PK1} and 
\eqref{PKK12}
that 
\begin{align*}
\|  \NN_1^{(j)} (\uu) \|_{L^2_\xi}
 \le 
C
K^{-4(j-2) \ta}
 \| \ft\uu\|_{L^2_\xi}^{j+1},
 \end{align*}

\noi
yielding
\eqref{PK2} in this case, 
provided that $0  < \ta < \frac 12$.

\medskip

\noi
{\bf Step 5:}
Finally, we consider the case  $\pb(r^{(j)})  \ne r^{(j-1)}$
(assuming that $|\xi^{(j)}|\gg 1$).
In this case,  $\pi_{j-1}(\TT_j)$
and $\pi_{j}(\TT_j)$ are disjoint
and both appear at the bottom of the ordered tree $\TT_j$
(in the planar graphical representation of  $\TT_j$, 
growing downward).
There exists a unique node $r^{(\l)}$ for some 
$\l \in \{1,2, \dots, j-2\}$
such that 
\begin{align}
\pi_{j-1}(\TT_j), 
\pi_{j}(\TT_j)
\subset \TT_{r^{(\l)}}:= \big\{ a\in \TT_j: a\le r^{(\l)}\big\}
\label{XA0}
\end{align}

\noi
but 
no
subtree rooted at a proper descendant of $r^{(\l)}$
contains both 
$\pi_{j-1}(\TT_j)$ and $\pi_{j}(\TT_j)$.
Namely, 
$r^{(\l)}$ 
 is the ``least common ancestor''
in the sense that 
if we consider the unique
paths from $r^{(k)}$, $k = j-1, j$,  to the root node of $\TT_j$,
then these paths meet at $r^{(\l)}$.
We note that both children of $r^{(\l)}$ are non-terminal.
This fact allows us to prove an improved estimate
for the contribution coming from $\pi_\l(\TT_j)$
(see $S_1(f_1, f_2)$ in \eqref{XA2} and~\eqref{XA2a}).
Without loss of generality, we assume that 
\begin{align}
\begin{split}
\pi_{j-1}(\TT_j)
& \subset \TT_{r^{(\l)}_1}:= \big\{ a\in \TT_j: a\le r^{(\l)}_1\big\}, \\
\pi_{j}(\TT_j)
& \subset \TT_{r^{(\l)}_2}:= \big\{ a\in \TT_j: a\le r^{(\l)}_2\big\}
\end{split}
\label{XA1a}
\end{align}

\noi
in the following, 
where $r^{(\l)}_1$
(and $r^{(\l)}_2$) denotes the left child 
(and the right child, respectively) of $r^{(\l)}$
(in the planar graphical representation of  $\TT_j$).

The contribution from the subtree 
$\TT_{r^{(\l)}}$ is bounded by 
\begin{align}
\begin{split}
R(\xi_{r^{(\l)}}) = 
\int
\frac{\prod_{a\in \TT_{r^{(\l)}}^0}|\xi_{a}|}
{\prod_{a\in \TT_{r^{(\l)}}^0\setminus\{r^{(j-1)}, r^{(j)}\}}
\jb{\mu_{\gf(a)}}}
\frac{1}{\jb{\mu_{j-1}}} 
 \prod_{a\in \TT_{r^{(\l)}}^\infty}
|\ft \uu(\xi_a)|, 
\end{split}
\label{XA1}
\end{align}

\noi
where 
the integration is over 
\begin{align*}
\xi_{a} = \xi_{a_1}+\xi_{a_2}, \quad 
a\in \TT_{r^{(\l)}}^0 
\end{align*}

\noi
and
these frequencies satisfy \eqref{A4a}.
Here, $\gf(a)$ is as in Definition \ref{DEF:tree3}
such that   $a = r^{(\gf(a))}$,
namely, 
$\gf(a)$ denotes the generation
(in the original ordered tree $\TT_j$)
where $a\in \TT_{r^{(\l)}}^0\subset \TT_j$ appears as a parent.

In the following, 
we estimate $R(\xi_{r^{(\l)}})$
by successively applying various bilinear estimates.
For this purpose, we need to introduce several bilinear operators.
Given  $0 < \ta < \frac 12$, 
fix  $\al = \al(\ta) > 0$ (to be chosen later).
Define $S_1(f_1, f_2)$ by setting
\begin{align}
S_1(f_1, f_2)(\xi) = 
\int_{\xi = \xi_1 + \xi_2}
\ind_{|\xi \xi_1\xi_2|\ge 1}
\frac{|\xi| |\xi_1|^\al |\xi_2|^\al }{|\xi \xi_1\xi_2|^{1-\ta}} 
\prod_{k = 1}^2 |f_k (\xi_k)| d\xi_1. 
\label{XA2}
\end{align}

\noi
This term represents
the contribution from $\pi_\l(\TT_j)$
(i.e.~the first generation within the subtree
$\TT_{r^{(\l)}}$ defined in \eqref{XA0}), 
where we also absorb $\al$-derivatives
from both children
which are assumed to be non-terminal.
Next, we define $S_2(f_1, f_2)$ by 
\begin{align}
S_2(f_1, f_2)(\xi) 
& = 
\int_{\xi = \xi_1 + \xi_2}
\ind_{|\xi \xi_1\xi_2|\ge 1}
\frac{|\xi|^{1-\al} |\xi_1|^\al }{|\xi \xi_1\xi_2|^{1-\ta}} 
\prod_{k = 1}^2 |f_k (\xi_k)| d\xi_1, 
\label{XA2x}
\end{align}

\noi
where $\al = \al(\ta)> 0 $ is as above.
The term $S_2$ 
represents the 
contribution from 
 subtrees of one generation in $\TT_{r^{(\l)}}$
lying on the path from 
$r^{(\l)}_1$ to $r^{(j-1)}$
(or from $r^{(\l)}_2$ to $r^{(j)}$)
{\it except} for 
the last two subtrees $\pi_{j-1}(\TT_j)$
and  $\pi_{j}(\TT_j)$, 
where we assume that the left child is non-terminal.
Lastly,  we  consider the contribution
from 
the last two subtrees $\pi_{j-1}(\TT_j)$
and  $\pi_{j}(\TT_j)$.
Without loss of generality, assume that 
\begin{align*}
|\xi_1^{(k)}|\ge |\xi_2^{(k)}|,\quad  k = j-1, j.
\end{align*}

\noi
This implies that   
\begin{align}
|\xi_1^{(k)}| \ges |\xi^{(k)}|
\quad \text{and}\quad 
|\xi_1^{(k)}|\gg 1, \quad k = j-1, j, 
\label{XX1b}
\end{align}

\noi
in view of  $|\mu_{j-1}|\gg1$ and $|\xi^{(j)}|\gg1$.
Moreover, 
from  \eqref{A4a}, we have 
\begin{align}
\frac{1}{|\mu_{j-1}|} \les
j^{2-\al}
\frac 1{\jb{\mu_{j-1}}^{\frac 12 - \frac 14 \al}}
\frac {1}{\jb{\mu_{j}}^{\frac 12 - \frac 14 \al}}. 
\label{XX1c}
\end{align}

\noi
In view of \eqref{XX1b} and 
\eqref{XX1c}, we then 
define $S_3(f_1, f_2)$ by 
\begin{align}
S_3(f_1, f_2)(\xi) 
& = 
\int_{\xi = \xi_1 + \xi_2}
\ind_{|\xi_1 |\gg 1}
\frac{|\xi|^{1-\al} }{|\xi \xi_1\xi_2|^{\frac 12 - \frac 14 \al}} 
\prod_{k = 1}^2 |f_k (\xi_k)| d\xi_1, 
\label{XA5}
\end{align}

\noi
representing the 
contribution from 
 $\pi_{j-1}(\TT_j)$
or   $\pi_{j}(\TT_j)$, 
where $\al = \al(\ta)> 0 $ is as above.
The whole point of ``shifting $\al$-derivatives'' to previous generations
for the nodes
lying
 on the path from 
$r^{(\l)}_1$ to $r^{(j-1)}$
(or from $r^{(\l)}_2$ to $r^{(j)}$)
 is to guarantee boundedness of  $S_3$
from $(L^2_\xi)^{\otimes 2}$
 into $L^2_\xi$.
 When $|\xi_2|\ll 1$, 
we need the power of $|\xi_2|$ in the denominator (in \eqref{XA5}) to be less than $\frac12$
(such that we can sum over dyadic blocks
for the frequency $|\xi_2|\ll 1$), 
which in turn forces us to have the power of $|\xi|$ in the numerator
to be less than $1$.
See Case 5.3 below.

We claim that there exists $\al = \al(\ta) > 0$ such that 
\begin{align}
\|S_m(f_1, f_2)\|_{L^2_\xi}
\les 
\prod_{k = 1}^2 \|f_k \|_{L^2_\xi}, \quad m = 1, 2, 3, 
\label{XA2a}
\end{align}

\noi
provided that $\frac 1{14} <  \ta < \frac 12$.

\medskip

\noi
$\bul$ {\bf Case 5.1:} $m = 1$.
\\
\indent
By symmetry,  assume that $|\xi_1|\ge |\xi_2|$, 
which implies $|\xi_1|\ges 1$
under the constraint
$|\xi \xi_1\xi_2|\ge 1$.
We first consider the case  $|\xi_2| \les |\xi|$.
By applying dyadic decompositions $|\xi_k|\sim N_k$, $k = 1, 2$, 
and Cauchy-Schwarz's inequality
to \eqref{XA2}, we have 
\begin{align*}
\|S_1(f_1, f_2)\|_{L^2_\xi}
\les \sum_{\substack{N_1 \ges \max(N_2, 1)\\\text{dyadic}}}
\frac{N_2^{\ta + \al - \frac 12} }{N_1^{1- 2\ta - \al}}
\prod_{k = 1}^2 \|\P_{N_k} f_k \|_{L^2_\xi}.
\end{align*}

\noi
Then, by choosing 
\begin{align}
\al = \frac 23 - \frac 43 \ta> 0,
\label{XA3a} 
\end{align}

\noi
we have 
\begin{align}
\|S_1(f_1, f_2)\|_{L^2_\xi}
\les \sum_{\substack{N_1 \ges \max(N_2, 1)\\\text{dyadic}}}
\frac{N_2^{\frac 14  \al} }{N_1^{\frac 12 \al}}
\prod_{k = 1}^2 \|\P_{N_k}f_k \|_{L^2_\xi}
\les 
\prod_{k = 1}^2 \|f_k \|_{L^2_\xi}.
\label{XA3}
\end{align}

\noi
Here, 
 the  power 
$N_2^{\frac 14 \al}$ in the numerator
is needed to sum over 
$N_2\ll1$.

Next, we  consider the case  $|\xi_2| \gg |\xi|$, 
which implies $|\xi_1|\sim |\xi_2|\ges 1$.
Then, by Cauchy-Schwarz's inequality, we have 
\begin{align}
\|S_1(f_1, f_2)\|_{L^2_\xi}
\les \sum_{\substack{N_1 \sim N_2 \ges 1\\\text{dyadic}}}
\frac{1}{N_1^{\frac 14\al}}
\prod_{k = 1}^2 \|\P_{N_k} f_k \|_{L^2_\xi}
\les 
\prod_{k = 1}^2 \|f_k \|_{L^2_\xi}.
\label{XA4}
\end{align}

\noi
Hence, the bound \eqref{XA2a} for $m = 1$ follows
from \eqref{XA3} and \eqref{XA4}, 
provided that $0 < \ta < \frac 12$.

\medskip

\noi
$\bul$ {\bf Case 5.2:} $m = 2$.
\\
\indent
We first consider the case  $|\xi_1|\gg |\xi_2|$, which implies $|\xi|\sim |\xi_1|$.
In this case, we have
$|\xi|^{1-\al} |\xi_1|^\al \sim |\xi|$
and thus 
\eqref{XA2a} follows from 
(the proof of)
the basic estimate \eqref{X1}.

Next, we consider the case
 $|\xi_1|\les |\xi_2|$.
Then, we have $|\xi_2| \ges 1$
under the constraint
$|\xi \xi_1\xi_2|\ge 1$.
If $|\xi|\ges 1$, then
it follows from \eqref{XA2} and \eqref{XA2x} that  $|S_2(f_1, f_2)| \les 
|S_1(f_1, f_2)| $.
Thus,  \eqref{XA2a}
follows from Case 5.1.
If $|\xi|\ll 1$, then
we have $|\xi_1|\sim| \xi_2|\ges 1 \gg |\xi|$.
By applying dyadic decompositions 
$|\xi|\sim N$ and 
$|\xi_k|\sim N_k$, $k = 1, 2$, 
and Cauchy-Schwarz's inequality
with \eqref{XA3a}, we have 
\begin{align*}
\|S_2(f_1, f_2)\|_{L^2_\xi}
& \les \sum_{\substack{N_1 \sim N_2 \ges 1 \gg N\\\text{dyadic}}}
\frac{N^{1- \frac 74\al} }{N_1^{1+ \frac 12 \al}}
\prod_{k = 1}^2 \|\P_{N_k} f_k \|_{L^2_\xi} \les 
\prod_{k = 1}^2 \|f_k \|_{L^2_\xi}
\end{align*}

\noi
provided that $0 < \al <  \frac 47$, 
namely, $\frac 1{14} <  \ta < \frac 12$ in view of \eqref{XA3a}.

\medskip

\noi
$\bul$ {\bf Case 5.3:} $m = 3$.
\\
\indent
We first consider the case  $|\xi_2| \les |\xi|$, 
which implies
  $|\xi|\sim |\xi_1|\gg 1$.
By applying dyadic decompositions,  
 Cauchy-Schwarz's inequality, and \eqref{XA3a}
to~\eqref{XA5}, we have 
\begin{align}
\|S_3(f_1, f_2)\|_{L^2_\xi}
\les \sum_{\substack{N_1 \ges \max(N_2, 1)\\\text{dyadic}}}
\frac{N_2^{\frac 14 \al} }{N_1^{\frac 12  \al}}
\prod_{k = 1}^2 \|\P_{N_k} f_k \|_{L^2_\xi}
\les 
\prod_{k = 1}^2 \| f_k \|_{L^2_\xi}.
\label{XA6}
\end{align}

\noi
Next, we  consider the case  $|\xi_2| \gg |\xi|$, 
which implies $|\xi_1|\sim |\xi_2|\gg 1$.
Then, by Cauchy-Schwarz's inequality with \eqref{XA3a}, we have 
\begin{align}
\|S_3(f_1, f_2)\|_{L^2_\xi}
\les \sum_{\substack{N_1 \sim N_2 \ges 1\\\text{dyadic}}}
\frac{1}{N_1^{\frac 14\al}}
\prod_{k = 1}^2 \|\P_{N_k} f_k \|_{L^2_\xi}
\les 
\prod_{k = 1}^2 \|f_k \|_{L^2_\xi}.
\label{XA7}
\end{align}

\noi
Hence, the bound \eqref{XA2a} for $m = 3$ follows
from \eqref{XA6} and \eqref{XA7}, 
provided that $0 < \ta < \frac 12$.

\medskip

Now, we turn our attention to estimating  $R(\xi_{r^{(\l)}})$ in \eqref{XA1}.
Under the assumption \eqref{XA1a}, 
consider the  (disjoint) paths
$\mathfrak{P}_1$ and $\mathfrak{P}_2$ in $\TT_{r^{(\l)}}$ given by 

\smallskip
\begin{itemize}
\item
$\mathfrak{P}_1= (a^1, a^2, \dots, a^{m_1})$
from 
$a^1 : = r^{(\l)}_1$ to $a^{m_1} := r^{(j-1)}$, 
where $\pb(a^{\l+1}) =  a^{\l}$, and 

\smallskip
\item 
$\mathfrak{P}_2 = (b^1, b^2,  \ldots,  b^{m_2})$
from 
$b^1 : = r^{(\l)}_2$ to $b^{m_2} := r^{(j)}$ 
where $\pb(b^{\l+1}) =  b^{\l}$.

\end{itemize}

\smallskip

\noi
By symmetry, we  assume that $a^{\l+1}$ 
(and $b^{\l+1}$) is always the left child of $a^\l$ (and of $b^\l$, respectively).
We also write 
$a^k = r^{(\l_k)}$
and 
$b^k = r^{(\wt \l_k)}$
for some $\l_k$ and $\wt \l_k$ which are increasing in $k$
such that $\l_{m_1} = j-1$ and $\wt \l_{m_2} = j$.
With this notation, 
we first apply \eqref{XA2a} with $m = 1$ 
to $\pi_\l(\TT_j)$
and then iteratively apply

\smallskip

\begin{itemize}
\item \eqref{XA2a} with $m = 2$
to $\pi_{\l_k}(\TT_j)$, $k = 1, \dots, m_1-1$, 
and 
$\pi_{\wt \l_k}(\TT_j)$, $k = 1, \dots, m_2-1$,
namely subtrees of one generation 
(consisting of three nodes; one parent with two children, 
at least one of which is non-terminal)
in $\TT_{r^{(\l)}_1}$ or $\TT_{r^{(\l)}_2}$
which have
non-trivial intersection with the path $\mathfrak{P}_1$
or $\mathfrak{P}_2$,

\smallskip

\item
\eqref{XX1c} and
\eqref{XA2a} with $m = 3$
to $\pi_{j-1}(\TT_j)$ and $\pi_{j}(\TT_j)$,

\smallskip

\item the basic estimate \eqref{X1}
to all the other 
subtrees of one generation
in $\TT_{r^{(\l)}_1}$ or $\TT_{r^{(\l)}_2}$
defined in \eqref{XA1a}.

\end{itemize}

\smallskip

\noi
Recalling from \eqref{PKK1}
that we are on $A_j\cap (\bigcap_{k=1}^{j-1} A_k^c)$,
we use 
\eqref{A4a}
in applying \eqref{XA2a} for $m = 1$ or $2$
in the process described above.
This yields
\begin{align}
\begin{split}
\|R(\xi_{r^{(\l)}})\|_{L^2_{\xi_{r^{(\l)}}}}
& \le C^j
j^{2-\al}
\bigg(\prod_{a\in \TT_{r^{(\l)}}^0\setminus\{r^{(j-1)}, r^{(j)}\}}
(\gf(a)+2)! 
\bigg)^{-4\ta}\\
& \quad \times K^{-4(M-3)\ta}
\|\ft \uu\|_{L^2_\xi}^M, 
\end{split}
\label{XA8}
\end{align}

\noi
provided that 
$\frac 1{14} <  \ta < \frac 12$, 
where $M = |\TT_{r^{(\l)}}^\infty|$.
Therefore, 
iteratively applying the basic estimate~\eqref{X1}
and also applying \eqref{XA8}
and \eqref{TJ0}
to the right-hand side of \eqref{PKK1}, 
we obtain~\eqref{PK2} in this case, 
 provided that $\frac 1{14}<  \ta < \frac 12$.
This concludes the proof of Proposition \ref{PROP:NN}\,(i).

\medskip

\noi
(ii) 
From  \eqref{NF12} and \eqref{A4},  we have
\begin{align}
\begin{split}
 \|  \jb{\xi}^{-1} \NN_2^{(j)} (\uu) \|_{L^\infty_\xi}
& \les 
\bigg\| 
 \sum_{\TT_j \in \Tfr(j)}
\intt_{\substack{\xii \in \Nf(\TT_j) \\ \xi_{r} = \xi} }
 \ind_{\bigcap_{k=1}^{j-1} A_k^c} \cdot
 \frac{|\xi^{(1)}|}{\jb{\xi^{(1)}}}
\\
& 
\hphantom{XXXXXXXX}
\times 
 \frac{  \prod_{k=2}^j | \xi^{(k)}|  }{\prod_{k=1}^{j-1} \jb{\mu_k}} 
\prod_{a \in \TT^\infty_j} |\ft\uu(\xi_a)|
\bigg\|_{L^\infty_\xi}.
\end{split}
\label{RE1}
\end{align}

Let $j = 1$.\footnote{As noted in Section \ref{SEC:per}, 
we present the $j = 1$ case since it is needed to treat the error term $\EE_\dl^{(j)}$
in Part (iii) below.} From Young's inequality, we have 
\begin{align*}
\text{RHS of }\eqref{RE1}
\les  \| \ft\uu\|_{L^2_\xi}^2.
\end{align*}

\noi
Fix $j\ge2$.
Recall from 
Definition \ref{DEF:tree3}
that 
$r^{(j)}$ is the non-terminal node in $\pi_j(\TT_j)$
such that 
$\xi_{r^{(j)}} = \xi^{(j)}$.
Then,  from \eqref{RE1}, \eqref{mu1},  and  Cauchy-Schwarz's inequality
(see \eqref{PE4}), 
we have 
\begin{align}
\text{RHS of }\eqref{RE1}
\le C^j 
\| \ft\uu\|_{L^2_\xi}^{j+1}
\sum_{\TT_j \in \Tfr(j)} 
\|
 G_{\TT_j}(\xi) ^\frac 12
\|_{L^\infty_\xi}, 
\label{RE3}
\end{align}

\noi
where
$G_{\TT_j}$ is defined by 
\begin{align}
\begin{split}
 G_{\TT_j}(\xi) 
 &  = \intt_{\substack{\xii \in \Nf(\Pi_{j - 1}(\TT_j)) \\ \xi_{r} = \xi} }
\ind_{\bigcap_{k=1}^{j-1} A_k^c}
\cdot
 \frac{|\xi^{(1)}|^2  |\xi_{r^{(j)}}|^2 }{\jb{\xi^{(1)}}^2 }
 \frac{ \prod_{a \in \TT^0_j \setminus\{r, r^{(j)}\}}|\xi_a|^2}
 { \prod_{a \in \TT^0_j \setminus\{r^{(j)}\}}\jb{\xi_a \xi_{a_1} \xi_{a_2}}^2}, 
\end{split}
\label{RE4}
\end{align}

\noi
where $r$ denotes the root node of $\TT_j$
and $\Pi_{j - 1}(\TT_j)$ is as in Definition \ref{DEF:tree3}.
In the following, 
we estimate  the $j-1$ integrals, appearing  in the definition of $G_{\TT_j}$.

Fix $\TT_j \in \Tfr(j)$.
Consider the unique path in  $\TT_j$ from  
$b^1 : = r $ (= the root node of $\TT_j$) to $b^m := r^{(j)}$ (for some $m \in \{2, \dots, j\}$) given by 
\begin{align*}
(b^1, b^2, \ldots,  b^m), 
\end{align*}

\noi
where $\pb(b^{\l+1}) =  b^{\l}$.
By symmetry, we  assume that $b^{\l+1}$ is always the left child of $b^\l$
(in the planar graphical representation of the binary tree $\TT_j$)
 and thus we write $b^{\l+1} = b^\l_1$, 
 while $\sbb(b^{\l+1}) = b^\l_2$, $\l = 1, \dots, m-1$.
We divide the argument into the following three cases:

\smallskip
\begin{itemize}
\item[\bf (ii.1)]
$|\xi_{b^{m-1}_2}| = |\xi_{\sbb(b^m)}| \ges 1$,

\smallskip
\item[\bf (ii.2)]
there exists $m_* \in \{1, \dots, m-2\}$ such that 
$|\xi_{b^\l_2}| \ll 1 $ for $\l=m_*+1, \ldots, m -1$ and $|\xi_{b^{m_* }_2 }| \ges 1$, 

\smallskip

\item[\bf (ii.3)] 
$|\xi_{b^\l_2}| \ll 1 $ for $\l=1, \ldots, m-1$.
\end{itemize}

\smallskip

\noi
In the following, 
we estimate the integrals in \eqref{RE4}
involving $\xi_{a} = \xi_{a_1} + \xi_{a_2}$ for $a\in 
\{b^\l, b^{\l+1}, \ldots, b^{m-1}\}$, 
where $\l = m-1$ in Case\,(ii.1), 
$\l = m_*$ in Case\,(ii.2), 
and $\l = 1$ in Case\,(ii.3).

\medskip
\noi 
$\bul$ \textbf{Case (ii.1):} $|\xi_{b^{m-1}_2}| = |\xi_{\sbb(b^m)}| \ges 1$.
\\
\indent
Recalling that 
$r^{(j)} = b^m = b^{m-1}_1$, 
we have
\begin{align}
\label{RE5}
\begin{split}
& \intt_{\xi_{b^{m-1}} 
= \xi_{b^{m-1}_1} + \xi_{b^{m-1}_2}} 
\frac{|\xi_{r^{(j)}}\xi_{b^{m-1}} |^2}{\jb{\xi_{b^{m-1}} \xi_{b^{m-1}_1} \xi_{b^{m-1}_2}}^2} d\xi_{b^{m-1}_1}\\
&\hphantom{XXXX}
\les 
\intt_{|\xi_{b^{m-1}_2}|\ges 1} 
\frac{1}{ | \xi_{b^{m-1}_2} |^2  } d \xi_{b^{m-1}_2} \les 1. 
\end{split}
\end{align}

\medskip

\noi 
$\bul$ \textbf{Case (ii.2):} 
$|\xi_{b^\l_2}| \ll 1 $, $\l=m_*+1, \ldots, m -1$,  and $|\xi_{b^{m_* }_2 }| \ges 1$
for some $m_* \in \{1, \dots, m-2\}$.
\\
\indent
In \eqref{RE4}, the frequencies  $(\xi_{b^\l}, \xi_{b^{\l}_1}, \xi_{b^{\l}_2})$ are restricted to 
the non-resonant set $A^c_{\gf(b^\l)}$.
Here,   $\gf(a)$ is as in Definition \ref{DEF:tree3}, 
denoting  the generation
(in the original ordered tree $\TT_j$)
where $a\in  \TT_j$ appears as a parent.
Thus,  we have 
 $|\xi_{b^\l} \xi_{b^\l_1} \xi_{b^\l_2}| \gg 1$ 
for $\l=1, \ldots, m-1$, which in turn implies that 
\begin{align}
\label{RE6}
\begin{split}
& |\xi_{b^{\l}_2} |  \ll 1 \les | \xi_{b^{\l+1}} | = |\xi_{b^{\l}_1}| \sim |\xi_{b^{\l}}| = |\xi_{b^{\l-1}_1}|, 
\quad \l=m_*+1, \ldots, m-1, \\
&  \text{and} 
\quad 
|\xi_{b^{m_*}_2}| \ges 1. 
\end{split}
\end{align}

\noi
In particular, we have
\begin{align}
|\xi_{r^{(j)}}| = |\xi_{b^m}| \le C^j |\xi_{b^{m_* }_1}|.
\label{RE7}
\end{align}

We first consider the case $m_*\ge 2$.
Using \eqref{RE6} and
\eqref{RE7} with \eqref{A4}, 
we can bound 
 the $m-m_*$ integrals on $\xi_{b^\l} = \xi_{b^\l_1} + \xi_{b^\l_2}$, 
$\l=m_*, \ldots, m -1$, 
in \eqref{RE4} by 
\begin{align}
&
\intt_{\substack{\xi_{b^\l} = \xi_{b^\l_1} + \xi_{b^\l_2} \\ \l=m_*, \ldots, m-1 }}
|\xi_{r^{(j)}}|^2 \prod_{\l=m_*}^{m-1 } \frac{|\xi_{b^\l}|^2 }{ \jb{\xi_{b^\l} \xi_{b^\l_1} \xi_{b^\l_2}}^{2}} 
\notag \\
&
\quad 
\les
\intt_{\substack{\xi_{b^\l} = \xi_{b^\l_1} + \xi_{b^\l_2} \\ \l=m_*, \ldots, m-1 }} 
\frac{| \xi_{r^{(j)}}\xi_{b^{m_*}}  |^2 }{ | \xi_{b^{m_*}}  \xi_{b^{m_*}_1}  \xi_{b^{m_*}_2} |^2 } 
\prod_{\l=m_*+1}^{m -1 } \frac{|\xi_{b^\l}|^2 }{ \jb{\xi_{b^\l} \xi_{b^\l_1} \xi_{b^\l_2}}^{2} } 
\notag \\
&\quad 
\le C^j
\bigg(\prod_{\l = m_*+1}^{m-1}(\gf(b^\l)+2)! \bigg)^{-8\ta}
K^{-8(m-m_*-1)\ta}
\intt_{\substack{\xi_{b^{m_*}} = \xi_{b^{m_*}_1} + \xi_{b^{m_*}_2}\\|\xi_{b^{m_*}_2}| \ges 1}} 
\frac{ d\xi_{b^{m_*}_2} }{ |  \xi_{b^{m_*}_2} |^2 } 
\label{RE7a}\\
& \hphantom{XXX} \intt_{\xi_{b^{m_* +1}} = \xi_{b^{m_* +1}_1} + \xi_{b^{m_*+1}_2}} 
\frac{ |\xi_{b^{m_*+1}}|^2d\xi_{b^{m_*+1}_1} }{ \jb{\xi_{b^{m_*+1 }} 
\xi_{b^{m_*+1}_1} \xi_{b^{m_*+1}_2}}^{2(1-\ta)} }
\notag \\
& \hphantom{XXX}
 \cdots
\intt_{\xi_{b^{m-1}} = \xi_{b^{m-1}_1} + \xi_{b^{m-1}_2}} 
\frac{|\xi_{b^{m-1}}|^2d\xi_{b^{m-1}_1}  }{ \jb{\xi_{b^{m-1 }} \xi_{b^{m-1}_1} \xi_{b^{m-1}_2}}^{2(1-\ta)} }, 
\notag 
\end{align}

\noi
provided that $\ta > 0$, 
where the integrals on the right-hand side 
are nested.

To evaluate the inner $m-m_*-1$ integrals on $\xi_{b^\l} = \xi_{b^{\l}_1} + \xi_{b^{\l}_2}$, 
$\l=m_*+1, \ldots, m-1$, 
we consider $\xi_{b^\l}$ as fixed (by the outer integrals), and 
perform a change of variables from $\xi_{b^\l_1}$ to $\ze = \xi_{b^\l} \xi_{b^\l_1} (\xi_{b^\l} - \xi_{b^\l_1})$
as in \eqref{X5}. 
Note that we have 
\[|\partial_{\xi_{b^\l_1}} \ze| = |\xi_{b^\l} (\xi_{b^\l} -2  \xi_{b^\l_1})| 
= |\xi_{b^\l} (\xi_{b^\l_2} -  \xi_{b^\l_1})|\sim |\xi_{b^\l}|^2.\] 

\noi
Thus, we have 
\begin{align}\label{RE9}
\intt_{\xi_{b^\l} = \xi_{b^\l_1} + \xi_{b^\l_2} } 
\frac{|\xi_{b^\l}|^2 }{ \jb{\xi_{b^\l} \xi_{b^\l_1} \xi_{b^\l_2}}^{2(1-\ta)} } 
d\xi_{b^\l_1}
\sim 
\intt \frac{1}{\jb{\ze}^{2(1-\ta)}} d\ze \les 1
\end{align}

\noi
for $\l=m_*+1, \ldots, m-1$, 
provided that $0 <  \ta <  \frac12$. 
Hence, from \eqref{RE7a} and \eqref{RE9}, 
we obtain
\begin{align}
\begin{split}
&
\intt_{\substack{\xi_{b^\l} = \xi_{b^\l_1} + \xi_{b^\l_2} \\ \l=m_*, \ldots, m-1 }}
|\xi_{r^{(j)}}|^2 \prod_{\l=m_*}^{m-1 } \frac{|\xi_{b^\l}|^2 }{ \jb{\xi_{b^\l} \xi_{b^\l_1} \xi_{b^\l_2}}^{2}} 
\\
& \hphantom{XXXXX}
\le C^j
\bigg(\prod_{\l = m_*+1}^{m-1}(\gf(b^\l)+2)! \bigg)^{-8\ta}
K^{-8(m-m_*-1)\ta}.
\end{split}
\label{RE10}
\end{align}

Next, we consider the case $m_* = 1$.
In this case, 
 the contribution to 
 \eqref{RE4} 
 from the 
 $m-1$ integrals on $\xi_{b^\l} = \xi_{b^\l_1} + \xi_{b^\l_2}$, 
$\l=1, \ldots, m -1$, 
is given by 
\begin{align*}
&
\intt_{\substack{\xi_{b^\l} = \xi_{b^\l_1} + \xi_{b^\l_2} \\ \l=1, \ldots, m-1 }}
 \frac{|\xi^{(1)}|^2  |\xi_{r^{(j)}}|^2 }{\jb{\xi^{(1)}}^2 }
  \frac{\prod_{\l=2}^{m-1 } |\xi_{b^\l}|^2 }{\prod_{\l=1}^{m-1 } \jb{\xi_{b^\l} \xi_{b^\l_1} \xi_{b^\l_2}}^{2}} \\
& \hphantom{XXXXX}
\les \intt_{\substack{\xi_{b^\l} = \xi_{b^\l_1} + \xi_{b^\l_2} \\ \l=1, \ldots, m-1 }}
|\xi_{r^{(j)}}|^2 
\prod_{\l=m_*}^{m-1 }
 \frac{ |\xi_{b^\l}|^2 }{ \jb{\xi_{b^\l} \xi_{b^\l_1} \xi_{b^\l_2}}^{2}} 
\end{align*}

\noi
to which we can apply
 \eqref{RE7a} and \eqref{RE9}, 
 thus yielding \eqref{RE10} in this case.

\medskip

\noi
$\bul$ \textbf{Case (ii.3):} $|\xi_{b^\l_2}| \ll 1 $, $\l=1, \ldots, m-1$
\\
\indent
As in \eqref{RE6}, we have
\begin{align}\label{RE11}
|\xi_{b^{\l}_2} | & \ll 1 \les | \xi_{b^{\l+1}} | = |\xi_{b^{\l}_1}| \sim |\xi_{b^{\l}}| , \quad \l=1, \ldots, m-1, 
\end{align}

\noi
in this case.
In particular, we have
\begin{align}
|\xi_{r^{(j)}}| = |\xi_{b^m}| \le C^{j} |\xi_{b^1}|.
\label{RE12}
\end{align}

\noi
Then, 
from \eqref{RE11}, \eqref{RE12}, 
and changes of variables as in \eqref{RE9}, 
we can bound 
 the $m-1$ integrals on $\xi_{b^\l} = \xi_{b^\l_1} + \xi_{b^\l_2}$, 
$\l=1, \ldots, m -1$, 
in~\eqref{RE4} by 
\begin{align}
&
\intt_{\substack{\xi_{b^\l} = \xi_{b^\l_1} + \xi_{b^\l_2} \\ \l=1, \ldots, m-1 }}
|\xi_{r^{(j)}}|^2  
\frac{\prod_{\l=2}^{ m-1 }|\xi_{b^\l}|^2 }{ \prod_{\l=1}^{ m-1 } \jb{\xi_{b^\l} \xi_{b^\l_1} \xi_{b^\l_2}}^{2} }
\notag \\
&
\quad 
\le C^j
\bigg(\prod_{\l = 1}^{m-1}(\gf(b^\l)+2)! \bigg)^{-8\ta}
K^{-8(m-1)\ta}
\intt_{\substack{\xi_{b^\l} = \xi_{b^\l_1} + \xi_{b^\l_2} \\ \l=1, \ldots, m-1 }}
\prod_{\l=1}^{ m-1 }\frac{|\xi_{b^\l}|^2 }{  \jb{\xi_{b^\l} \xi_{b^\l_1} \xi_{b^\l_2}}^{2(1-\ta)} } 
\label{RE13}\\
&\quad 
\le C^j
\bigg(\prod_{\l = 1}^{m-1}(\gf(b^\l)+2)! \bigg)^{-8\ta}
K^{-8(m-1)\ta}
\notag 
\end{align}

\noi
provided that $0 < \ta < \frac 12$.

\medskip

We note that  those integrals in \eqref{RE4}
which have not been treated in Cases (ii.1)-(ii.3)
are
 of the form 
 \eqref{X3}
 which are bounded in \eqref{X4}
and \eqref{X5}.
Therefore, from 
\eqref{RE3}, \eqref{RE4}, 
 Cases (ii.1)-(ii.3) (\eqref{RE5}, \eqref{RE10}, and \eqref{RE13}),  
  \eqref{X4} and \eqref{X5}
with \eqref{TJ0}, 
we obtain
\begin{align*}
\|  \jb{\xi}^{-1} \NN_2^{(j)} (\uu) \|_{L^\infty_\xi}
\le 
C^j 
j!
\bigg(\prod_{k = 1}^{j-2} (k+2)!\bigg)^{-4\ta}
K^{-4(j-2) \ta} 
\|\ft \uu\|_{L^2_\xi}^{j+1}
%
\too 0, 
\end{align*}

\noi
as $j \to \infty$, 
provided that $0 < \ta < \frac 12$.
This proves
\eqref{NX3}.

\medskip

\noi
(iii)
As in 
the periodic case presented in 
Section \ref{SEC:per}, 
the bound \eqref{NX6}
follows
from a modification of the argument in Part (ii), 
using 

\smallskip

\begin{itemize}
\item
 \eqref{PE9}
when 
 $\max\big(|\xi^{(j)}_1|, |\xi^{(j)}_2|\big) \le \dl^{-\frac 25 +\eps_0}$,

\smallskip

\item
\eqref{PE7c}
when 
 $\max\big(|\xi^{(j)}_1|, |\xi^{(j)}_2|\big) > \dl^{-\frac 25 +\eps_0}$.

\end{itemize}

\smallskip

\noi
Since the required modification is exactly the same
as in the periodic case, 
we omit details.
\end{proof}

\begin{ackno}\rm
The authors would like to thank Benjamin Harrop-Griffiths for bringing the reference \cite{GK} to their attention.
The authors  would  like to express their sincere gratitude to the anonymous
 referee 
for pointing out an error in the previous version
and for their comments, 
which improved the quality of the paper.
A.C.~was supported by  CNRS-INSMI through a grant ``PEPS Jeunes chercheurs et jeunes chercheuses 2025''.
G.L. was supported by the NSFC (grant no.~12501181)
 and by the Fundamental Research Funds for the Central Universities. 
T.O.~was supported by the European Research Council (grant no.~864138 ``SingStochDispDyn")
and
acknowledges support from  
the NSFC (grant no.~W2531005).
T.Z.~was supported by the National Natural Science Foundation of China (grant nos.~12271051 and 12371239) and by the Fundamental Research Funds for the Central Universities (FRF-BRA-26-007 and QNXM20260009).

\end{ackno}

\end{document}